\documentclass[10pt]{amsart}
\usepackage{amsmath,amsfonts,amssymb,amsthm}
\usepackage{graphics,graphicx}
\usepackage[colorlinks,hyperindex,linkcolor=blue,
citecolor=purple]{hyperref}
\hypersetup{
pdfauthor = {Francis Bischoff, Alvaro del Pino Gomez, Aldo Witte},
 pdftitle= {bk-algebroids and the variety of foliation jets},
 pdfsubject = {Lie Theory, singular foliations, log-Geometry, Lie algebroids}}
\usepackage{tikz-cd}
\usepackage{tikz}
\usepackage[all]{xy}
\usepackage{subfiles}
\usetikzlibrary{matrix,arrows,decorations.pathmorphing}

\newcommand{\ff}{\mathcal{F}}

\newcommand{\A}{{\mathcal{A}}}

\newcommand{\F}{{\mathcal{F}}}
\newcommand{\Gcal}{{\mathcal{G}}}

\newcommand{\Kcal}{{\mathcal{K}}}

\newcommand{\Ncal}{{\mathcal{N}}}

\newcommand{\RR}{\mathbb{R}}

\newcommand{\TT}{\mathbb{T}}

\newcommand{\ZZ}{\mathbb{Z}}

\newcommand{\Gr}{\operatorname{Gr}}

\newcommand{\Hom}{\operatorname{Hom}}

\newcommand{\im}{\operatorname{im}}

\newcommand{\Diff}{{\operatorname{Diff}}}

\newcommand{\hol}{{\operatorname{hol}}}

\newcommand{\FolCat}{{\underline{\operatorname{Fol}}}}
\newcommand{\FolObj}{{\operatorname{Fol}}}
\newcommand{\FlatCat}{{\underline{\operatorname{Flat}}}}
\newcommand{\FlatObj}{{\operatorname{Flat}}}
\DeclareMathOperator{\Rep}{\underline{\operatorname{Rep}}}

\newcommand{\IsotopyCat}{{\underline{\operatorname{I-Rep}}}}
\newcommand{\IsotopyObj}{{\operatorname{I-Hom}}}

\newcommand{\IsoCat}{\underline{\operatorname{I-Fol}}}

\newcommand{\TNE}{{\operatorname{TNE}}}

\newcommand{\Conn}{{\operatorname{Conn}}}

\newcommand{\Tan}{{\operatorname{Tan}}}

\newcommand{\Fr}{{\operatorname{Fr}}}

\newcommand{\set}[1]{\lbrace #1\rbrace}
\newcommand{\inp}[1]{\langle #1\rangle}

\newcommand{\rr}{\mathbb{R}}

\newcommand{\X}{{\mathfrak{X}}}
\newcommand{\G}{\mathcal{G}}
\newcommand{\N}{\mathcal{N}}

\newcommand{\C}{\mathcal{C}}

\DeclareMathOperator{\Lie}{Lie}

\DeclareMathOperator{\Der}{Der}
\DeclareMathOperator{\Aut}{Aut}

\DeclareMathOperator{\Sym}{Sym}
\DeclareMathOperator{\Gau}{Gau}

\renewcommand{\hat}{\widehat}

\DeclareMathOperator{\At}{At}
\DeclareMathOperator{\Ad}{Ad}
\DeclareMathOperator{\at}{at}
\DeclareMathOperator{\ad}{ad}

\usepackage{thmtools}
\declaretheorem[style=definition,qed=$\diamondsuit$]{definition}
\declaretheorem[style=definition,qed=$\triangle$,sibling=definition]{example}
\declaretheorem[style=plain,sibling=definition]{theorem}
\declaretheorem[style=plain,sibling=definition]{lemma}
\declaretheorem[style=plain,sibling=definition]{proposition}
\declaretheorem[style=plain,sibling=definition]{corollary}
\declaretheorem[style=definition,qed=$\diamondsuit$,sibling=example]{claim}
\declaretheorem[style=definition,sibling=example]{question}
\declaretheorem[style=definition,qed=$\diamondsuit$,sibling=claim]{remark}

\declaretheorem[style=definition,sibling=example,qed=$\triangle$]{assumption}

\newtheorem*{claim*}{Claim}
\numberwithin{theoremalpha}{section}

\numberwithin{equation}{section}
\numberwithin{definition}{section}
\numberwithin{theorem}{section}
\numberwithin{proposition}{section}
\numberwithin{lemma}{section}
\numberwithin{example}{section}
\numberwithin{remark}{section}
\numberwithin{corollary}{section}
\numberwithin{question}{section}
\numberwithin{problem}{section}
\numberwithin{assumption}{section}

\newtheoremstyle{named}{}{}{\itshape}{}{\bfseries}{.}{.5em}{\thmnote{#3}#1}
\theoremstyle{named}
\newtheorem*{namedtheorem}{}

\numberwithin{equation}{section}
\address{University of Regina, 3737 Wascana Parkway, Regina, Saskatchewan, S4S 0A2, Canada}
\email{Francis.Bischoff@uregina.ca}

\address{Utrecht University, Department of Mathematics, Budapestlaan 6, 3584 CD Utrecht, The Netherlands}
\email{a.delpinogomez@uu.nl}

\address{University of Hamburg, Faculty of Mathematics, Informatics and Natural Sciences, Bundesstraße 55,20146 Hamburg}
\email{aldowitte@hotmail.nl}

\title{$b^k$-algebroids and the variety of foliation jets}
\author{Francis Bischoff, \'Alvaro del Pino and Aldo Witte}
\begin{document}
\begin{abstract}

We introduce and classify singular foliations of $b^{k+1}$-type, which formalize the properties of vector fields that are tangent to a submanifold $W \subset M$ to order $k$. When $W$ is a hypersurface, these structures are Lie algebroids generalizing the $b^{k+1}$-tangent bundles introduced by Scott.

We prove that singular foliations of $b^{k+1}$-type are encoded by \emph{$k$-th order foliations}--jets of distributions that are involutive up to order $k$, equivalently described as foliations on the $k$-th order neighborhood of $W$. Using this encoding, we construct topological groupoids of $k$-th order foliations and employ the holonomy invariant to show that these groupoids fiber over certain character stacks, yielding Riemann-Hilbert style classifications up to local isomorphism and isotopy.

We also study the problem of extending a $k$-th order foliation to a $(k+1)$-st order foliation. We prove that this is obstructed by a characteristic class that arises as a section of a vector bundle over the relevant character stack. 

\end{abstract}

\maketitle

\tableofcontents

\section{Introduction} \label{sec:introduction}

Given a manifold $M$ with a smooth hypersurface $W \subset M$, the \emph{b-tangent bundle} $TM(- \log W)$ is the Lie algebroid over $M$ whose sections are the $C^{\infty}(M)$-module $\mathfrak{X}(M,W)$ of vector fields on $M$ tangent to $W$. This algebroid was first introduced in algebraic geometry as the logarithmic tangent bundle. Its origins go at least as far back as Deligne's work on Hodge theory \cite{deligne1971hodge} and on flat connections with regular singularities \cite{MR0417174}. It was further generalized by Saito \cite{saito1980theory} in his study of free divisors. In the $C^{\infty}$-setting, the $b$-tangent bundle was rediscovered by Melrose \cite{MR1348401} during his investigation of index theory on non-compact manifolds. In this setting, the hypersurface $W$ is taken to be the boundary of $M$, thought of as an ``ideal boundary at infinity.''

The $b$-tangent bundle is a central example in the theory of Lie algebroids due both to its simplicity and to its utility in studying singular geometric structures in diverse settings, as demonstrated in \cite{GUILLEMIN2014864, Gualtieri-Li-2012, gualtieri2018stokes, MR3245143, cavalcanti2020self, MR4876292, del2022regularisation}. This success has motivated investigations of Lie algebroids admitting broader classes of singularities, including the work of \cite{MR3805052, MR4229238, klaasse2018poisson}. Prominent among these generalizations are the \emph{$b^{k+1}$-tangent bundles} introduced by Scott \cite{MR3523250} and extensively studied by Miranda and collaborators \cite{MR4523256, MR4257086, MR4236806, MR3952555,MN25}. Informally, these algebroids have sections consisting of vector fields on $M$ that are tangent to $W$ \emph{to order $k$}.\footnote{In the literature, the $b^{k+1}$-tangent bundle is defined as vector fields $(k+1)$-tangent to $D$. Unfortunately, this is inconsistent with the terminology coming from jets. Our indexing for $b^{k+1}$-tangent bundles is consistent with Scott's, but we have a shift by 1 in what we call $k$-tangent.} However, there is a subtlety here in that the notion of $k$-th order tangency is not intrinsically well-defined. Scott formalizes this notion by choosing additional data: the $k$-jet of a defining function for $W$. Consequently, for any given hypersurface $W \subset M$, there exist multiple distinct $b^{k+1}$-tangent bundles rather than a unique one. 

In this paper we generalize the $b^{k+1}$-tangent bundles by introducing the following class of Lie algebroids. 
\begin{definition} \label{hypersurfacealgebroiddef}
Let $M$ be a manifold of dimension $n$ and let $W \subset M$ be a hypersurface. A \emph{hypersurface algebroid} is a Lie algebroid $\A \Rightarrow M$ of rank $n$ whose orbits are precisely the connected components of $W$ and $M \backslash W$. 
\end{definition}
This definition is equivalent to imposing conditions on the anchor map $\rho: \A \to TM$ of the algebroid: $\rho$ must be an isomorphism on the complement $M \setminus W$ and have rank $n - 1$ along $W$. Consequently, the determinant $\det(\rho)$ is a section of the line bundle $\det(\A^*) \otimes \det(TM)$ that vanishes precisely along $W$. The order of this vanishing allows us to refine the above definition. 

\begin{definition}\label{hypersfalgbktype}
A hypersurface algebroid $\A$ is of \emph{$b^{k+1}$-type} if $\det(\rho)$ vanishes to order $k+1$ along $W$. 
\end{definition}
\begin{remark}
One could consider algebroids with different orders of vanishing along distinct components of $W$. However, we restrict our attention to the case where $W$ is connected, as all the interesting phenomena already appear in this simpler setting.
\end{remark}

As the name suggests, Scott's $b^{k+1}$-tangent bundles are examples of hypersurface algebroids of $b^{k+1}$-type. The question that initially motivated this project is the converse: \emph{Are all hypersurface algebroids of $b^{k+1}$-type isomorphic to one of Scott's $b^{k+1}$-tangent bundles?}

Locally, the answer is yes:
\begin{proposition}\label{prop:localformhypersurfacealgebroid}
Let $\A \Rightarrow M$ be a hypersurface algebroid of $b^{k+1}$-type. Then for any $x \in W$ there are local coordinates $(x_1,\ldots,x_n)$ with $W = \{ x_1 = 0 \}$ such that
\begin{equation}\label{eq:localform}
\Gamma(\A) \simeq \inp{x_1^{k+1}\partial_{x_1},\partial_{x_2},\ldots,\partial_{x_n}}.
\end{equation}
\end{proposition}
\begin{proof}
This follows from the Splitting Theorem for Lie algebroids \cite{Duf01, Fer02, Wei00, BLM19}. Namely, let $x \in W$ and let $\gamma: I \to M$ be a curve which passes through $x$ and is transverse to $W$. Then, there is a neighbourhood $U$ of $x$ and an isomorphism of Lie algebroids $\A|_U \cong \gamma^{!}\A \times TV$, where $V = U \cap W$. Here $\gamma^{!}\A$ is a rank-$1$ Lie algebroid over the interval $I$ with the property that its anchor vanishes precisely at the origin to order $k+1$. Hence, $i^{!}\A$ is generated by a vector field $x_{1}^{k+1} \partial_{x_{1}}$, for $x_{1}$ a coordinate on $I$. Let $x_{2}, ..., x_{n}$ be coordinates on $V$. Then $\A|_U$ has a basis given by the vector fields $x_{1}^{k+1} \partial_{x_{1}}, \partial_{x_{2}}, ..., \partial_{x_{n}}$.
\end{proof}

On the other hand, there exist hypersurface algebroids that are not \emph{globally} isomorphic to any Scott $b^{k+1}$-tangent bundle. Indeed, we present several such examples in Section \ref{sec:Examples1}, including a family with a positive-dimensional parameter space. 

The main purpose of this paper is to provide a classification of hypersurface algebroids. In Section \ref{sec:singfols}, we show that a hypersurface algebroid of $b^{k+1}$-type for $(W,M)$ is equivalent to the data of a foliation on the $k$-th order neighbourhood of $W$. In Section \ref{sec:Atiyah}, we introduce the \emph{$k$-th order Atiyah algebroid} and show that, upon choosing a tubular neighbourhood embedding for $W$, a hypersurface algebroid is encoded by a flat connection on a principal bundle of $k$-jets of frames for the normal bundle of $W$. This encoding enables us to give a Riemann-Hilbert style classification for hypersurface algebroids in Section \ref{sec:Riemann--Hilbert}. Here is a brief summary of this result. 

Let $G_{k,1} = J^{k}\mathrm{Diff}(\mathbb{R}, 0) = \mathrm{Aut}(\mathbb{R}[z]/(z^{k+1}))$ be the Lie group of $k$-jets of diffeomorphisms of $\mathbb{R}$ fixing the origin, and let $G_{k,1}^{0}$ be the connected component of the identity. By taking the derivative of a diffeomorphism at the origin, we obtain a homomorphism $\pi_{k}: G_{k,1} \to \mathbb{R}^* \cong \mathbb{Z}/2 \times \mathbb{R}$. Assume that $W$ is compact and consider the representation variety $\mathrm{Hom}(\pi_1(W), G_{k,1})$, where $\pi_1(W)$ is the fundamental group of $W$. By applying $\pi_{k}$ and projecting to $\mathbb{Z}/2$, we obtain a map 
\[
\mathrm{Hom}(\pi_1(W), G_{k,1}) \to \mathrm{Hom}(\pi_1(W), \mathbb{Z}/2) \cong H^{1}(W, \mathbb{Z}/2). 
\]
Recall that the codomain $H^{1}(W, \mathbb{Z}/2)$ classifies real line bundles over $W$. Let $\nu_W = TM|_{W}/TW$ denote the normal bundle of $W$ in $M$, and let $\mathrm{Hom}_{\nu_W}(\pi_1(W), G_{k,1})$ denote the subspace of representations which project to $\nu_W$ under the above map. This space is acted upon by $G_{k,1}^{0}$, and the quotient space $M^0_{k}(W,M) = \mathrm{Hom}_{\nu_W}(\pi_{1}(W), G_{k,1})/G_{k,1}^0$ is a \emph{character variety}\footnote{However, note that in the present context we are considering $M^0_{k}(W,M)$ to be the set of $G_{k,1}^{0}$ orbits equipped with the quotient topology, rather than an algebraic variety as is usually the case.}. The following theorem is one of the main results of this paper and is obtained from Theorem \ref{RHuptoisotopy} combined with Lemma \ref{lem:Dprincbundle}. 

\begin{namedtheorem}[Corollary \ref{cor:hypersurfclass}]
There is a homeomorphism between the space of isotopy classes of hypersurface algebroids of $b^{k+1}$-type for $(W, M)$ and the character variety 
\[
M^0_{k}(W,M) = \mathrm{Hom}_{\nu_W}(\pi_{1}(W), G_{k,1})/G_{k,1}^0.
\]
Furthermore, Scott's $b^{k+1}$-tangent bundles correspond to the trivial representation in $ \mathrm{Hom}_{\underline{\mathbb{R}}}(\pi_{1}(W), G_{k,1})$, where $\underline{\mathbb{R}}$ is the trivial line bundle over $W$. In particular, all $b^{k+1}$-tangent bundles are isotopic to each other. 
\end{namedtheorem}
Now composing $\pi_k : G_{k,1} \rightarrow \mathbb{Z}/2 \times \rr$ with the projection to the second factor we obtain a map
\begin{equation*}
\Hom_{\nu_W}(\pi_1(W),G_{k,1}) \rightarrow \Hom(\pi_1(W),\rr) \simeq H^1(W,\mathbb{R}).
\end{equation*}
The space $H^1(W,\mathbb{R})$ classifies flat connections on $\nu_W$ and we may fix this data to obtain a subspace $\Hom_{(\nu_W,\nabla)}(\pi_1(W),G_{k,1})$ of representations. The quotient space \[ M_{k}^0(W,M,\nabla) = \Hom_{(\nu_W,\nabla)}(\pi_1(W),G_{k,1})/G_{k,1}^0\] therefore classifies isotopy classes of hypersurface algebroids with a fixed induced flat connection on the normal bundle. We explicitly describe this space in the case $k = 3$. 
\begin{namedtheorem}[Proposition \ref{prop:g31character}]
The character variety $M_{3}^0(W,M,\nabla)$ is homeomorphic to the space of isotopy classes of hypersurface algebroids of $b^4$-type for $(W,M)$ with a fixed induced connection $\nabla$ on the normal bundle of $W$. This space contains a point $0$ which lies in the closure of every other point. Removing it, we have a homeomorphism  
\begin{equation*}
M_{3}^0(W,M,\nabla) \setminus \{ 0 \} \simeq S^{n_1 + n_2 - 1},
\end{equation*}
where $n_{1}$ and $n_{2}$ are the dimensions of the cohomology groups $H^{1}(W, \nabla^{-1})$ and $H^{1}(W, \nabla^{-2})$, respectively. 
\end{namedtheorem}

\subsection{Singular foliations of $b^{k+1}$-type}
It turns out that restricting to Lie algebroids is unnecessary, and the results of this paper are most naturally formulated in the broader framework of singular foliations. Recall that a singular foliation on a manifold $M$ is a $C^{\infty}(M)$ submodule of vector fields that is both closed under the Lie bracket and locally finitely generated. From a geometric perspective, a singular foliation provides a partition of $M$ into submanifolds of varying dimensions. While the study of singular foliations dates back to the 1960s, it has received renewed interest in differential geometry following the seminal work of Androulidakis and Skandalis \cite{AS09}. In this article, we study singular foliations that naturally generalize hypersurface algebroids by extending the notion of $k$-th order tangency to submanifolds $W \subset M$ of arbitrary codimension. 

\begin{definition}\label{def:typek}
Let $M$ be a manifold and let $W \subset M$ be a submanifold of codimension $l$. A singular foliation $\ff \subset \mathfrak{X}^1(M)$ is of \textbf{$b^{k+1}$-type} for $(W,M)$ if the following conditions are satisfied: 
\begin{enumerate}
\item the leaves of the foliation are the connected components of $W$ and $M \setminus W$, 
\item for every point $w \in W$, there are local coordinates $(x_{i})$ such that $W = \{ x_{1} = ... = x_{l} = 0 \}$ and such that $\ff$ is locally spanned by the following vector fields 
\[
\inp{x^I\partial_{x_1},\ldots,x^I\partial_{x_l},\partial_{x_{l+1}},\ldots,\partial_{x_n}},
\]
where $I = (i_1,\ldots,i_l)$ run over all multi-indices of weight $k + 1$ involving the first $l$ coordinates, and where $x^I = x_1^{i_1}\cdot \ldots \cdot x_l^{i_l}$. \hfill \qedhere
\end{enumerate}
\end{definition}

Note that a singular foliation of $b^{k+1}$-type for $(W, M)$ is locally free, and hence defines a Lie algebroid, only when $W$ has codimension $1$. Nevertheless, most results in this paper can be stated and proved at the level of generality of singular foliations. Indeed, the geometry of such foliations is entirely controlled by the $k$-th order Atiyah algebroid of the normal bundle to $W$, and this algebroid's local freeness does not depend on the codimension. 

In Section \ref{sec:singfols}, we show that singular foliations of $b^{k+1}$-type are equivalent to \emph{$k$-th order foliations}. These are foliations on the $k$-th order neighbourhood of the submanifold $W \subset M$. They may be explicitly described as germs of distributions $\xi \subset TM$ around $W$ that are ``involutive up to order $k$.'' In Section \ref{sec:Atiyah} we show that, after choosing a tubular neighbourhood embedding for $W$, a singular foliation of $b^{k+1}$-type is equivalent to a Lie algebroid splitting of the $k$-th order Atiyah algebroid. The resulting plethora of equivalent definitions is summarized in Theorem \ref{th:equivalence}. 

\subsection{Riemann-Hilbert correspondence}
The classification of hypersurface algebroids given in Corollary \ref{cor:hypersurfclass} follows from a detailed study of the topological groupoid of $k$-th order foliations. This is carried out in Section \ref{sec:Riemann--Hilbert}, the technical core of this paper. In this section, we prove three flavours of a Riemann-Hilbert correspondence for $k$-th order foliations: one set-theoretic (Theorem \ref{RHthm}), one topological (Theorem \ref{topologicalRHfoliation}), and one isotopical (Theorem \ref{RHuptoisotopy}). We briefly describe the topological version of the correspondence below. 

Let $J^k\FolObj(W,E)$ denote the set of $k$-th order foliations in a fixed vector bundle $E$ along the zero section $W$, and let $J^k\Diff(E,W)$ denote the group of $k$-jets of diffeomorphisms which fix $W$ pointwise. The action groupoid $J^{k}\FolCat(W,E) :=J^k\Diff(E,W) \ltimes J^k\FolObj(W,E)$ can be equipped with the Whitney $C^{\infty}$-topology, making it into a topological groupoid classifying $k$-th order foliations. 

Let $G_{k,l} = J^{k}\mathrm{Diff}(\mathbb{R}^l, 0)$ be the Lie group of $k$-jets of diffeomorphisms of $\mathbb{R}^l$ fixing the origin, and let $\Rep(\pi_{1}(W, x), G_{k,l}) = G_{k,l} \ltimes \Hom(\pi_{1}(W, x), G_{k,l})$ be the character groupoid. The vector bundle $E$ and the choice of a framing $\phi : \mathbb{R}^l \to E_{x}$ determine a full subgroupoid $\Rep_{(\Fr^k(E),\phi)}(\pi_{1}(W, x), G_{k,l})$, which is also a topological groupoid in a natural way. The topological Riemann-Hilbert correspondence for $k$-th order foliations can be stated as follows. 

\begin{namedtheorem}[Theorem \ref{topologicalRHfoliation}]
There is a fibration of topological groupoids
\begin{equation*}
RH_{(E,\phi)} : J^{k}\FolCat(W,E) \rightarrow \Rep_{(\Fr^k(E),\phi)}(\pi_{1}(W, x), G_{k,l}). 
\end{equation*}
In particular, the orbit spaces of $J^{k}\FolCat(W,E)$ and $\Rep_{(\Fr^k(E),\phi)}(\pi_{1}(W, x), G_{k,l})$ are homeomorphic.
\end{namedtheorem}

Note that the orbit space of $J^{k}\FolCat(W,E)$ classifies singular foliations of $b^{k+1}$-type up to isomorphisms that are defined in a neighbourhood of $W$. Theorem \ref{topologicalRHfoliation} therefore gives a classification of such foliations in terms of representations of the fundamental group of $W$. We also provide an explicit description of the kernel of $RH_{(E,\phi)}$ in terms of `inner automorphisms' in Theorem \ref{innercharacterization}. 

Section \ref{sec:Riemann--Hilbert} is supplemented by Appendix \ref{sec:topgroupoids}, which recalls the fundamentals of topological groupoids, Appendix \ref{sec:DiffTop}, which recalls some of the differential geometry of jets and tubular neighbourhoods, and Appendix \ref{sec:principal}, which gives a detailed summary of the Riemann-Hilbert correspondence for smooth flat connections.

\subsection{The extension problem}
In Corollary \ref{Frobenius}, we establish a Frobenius theorem for $k$-th order foliations, which ensures that these structures can always be \emph{locally} realized as the $k$-jets of regular foliations containing $W$ as a leaf. However, global realization is generally obstructed. We therefore study the following intermediate question: 

\begin{question}[Extension problem]
Let $F$ be a $k$-th order a foliation. When does there exist a $(k+1)$-st order foliation $\tilde{F}$ with the same $k$-jet as $F$?
\end{question}

We study this question from several different perspectives in Sections \ref{sec:extending}, \ref{sec:Atiyah}, and \ref{sec:extending2}. In Corollary \ref{cor:firstextensionclass}, we show that extensions of $F$ are obstructed by an \emph{extension class} $e(F) \in H^{2}(W, \mathrm{Sym}^{k+1}(\nu_W^*) \otimes \nu_W)$, where $\nu_W$ is the normal bundle of $W$ equipped with a flat connection induced by $F$. Upon choosing a tubular neighbourhood for $W$, this class is given an explicit de Rham representative in Corollary \ref{cor:explicitextensionclass}. Moreover, we show in Proposition \ref{ComparisonmapVE} that $e(F)$ arises from a cohomology class in the group cohomology of $\pi_{1}(W)$. This result uses the Riemann-Hilbert correspondence, a Morita equivalence, and the Van Est map. As a consequence, we demonstrate in Section \ref{sec:character} that the extension class may be viewed as a section of a vector bundle over the $G_{k,l}$-character stack (Proposition \ref{prop:extsection}). The zeroes of this section occur precisely along the locus of extendable $k$-th order foliations, indicating that the relationship between the character stacks for different values of $k$ are controlled by the extension classes. Finally, we provide a detailed description of the space of $(k+1)$-st order foliations extending a given $k$-th order foliation $F$ in Theorem \ref{thm:extensionspacegroupoid}. 

\subsection{Future work}

This paper is the first in a series devoted to the geometry of $k$-th order foliations. Here we focus on classification and moduli spaces. The second paper \cite{workinprogess} will study the geometry of hypersurface algebroids and their applications to Poisson geometry. In particular, we will construct new examples of generically symplectic Poisson structures by deforming along paths in the character varieties. The third paper will investigate confluence and unfolding for hypersurface algebroids. 

\subsection{Relation to other work}

This paper arose by splitting \cite{BPW23} into two expanded papers, the second being \cite{workinprogess}. Near the time of completion of \cite{BPW23}, we learned of the preprint \cite{francis2024singular} by Michael Francis, which is based on his 2021 thesis \cite{francis2021groupoids}. In this work, Francis independently introduces the notion of $k$-th order foliations for codimension-one submanifolds, under the name \emph{foliations of transverse order k}, and proves a set-theoretic version of the Riemann-Hilbert correspondence established in our Theorem \ref{RHthm}. However, the focus of our two works is complementary: while we emphasize the geometry of hypersurface algebroids, Francis instead describes their holonomy groupoids and the resulting groupoid $C^*$-algebras. 

After the appearance of \cite{BPW23}, the article \cite{fischer2024classification} appeared, which obtains a formal classification of singular foliations admitting a given leaf and a given transverse singular foliation. From their Corollary 3.12, a bijection between sets of isomorphism classes of singular foliations and representations--which forms part of our Theorem \ref{RHthm}--follows as a consequence. 

\subsection*{Acknowledgements}

We are grateful to Marius Crainic, Marco Zambon, Leonid Ryvkin, and Simon-Raphael Fischer for several insightful discussions. We are also grateful to the organizers of the `Higher Geometric Structures along the Lower Rhine XVI' conference, where a preliminary version of this article was presented. The first author is supported by an NSERC Discovery grant. The second author is supported by the NWO grant Vidi.223.118 ``Flexibility and rigidity of tangent distributions''. The third author was supported by FWO and FNRS under EoS projects G0H4518N and G012222N, FWO-EoS Project G083118N, directly by the University of Antwerp via BOF 49756 and is supported by the Deutsche Forschungsgemeinschaft (DFG, German Research Foundation) -- SFB 1624 -- ``Higher structures, moduli spaces and integrability'' -- 506632645. 


\section{Jets of distributions}\label{sec:jets}

In this section, we begin by recalling how to take $k$-jets of sections of a vector bundle, allowing us to define $k$-jets of vector fields and $k$-jets of distributions. Using these constructions, we introduce $k$-th order distributions and establish the appropriate integrability condition, leading to one of the central concepts of the paper: $k$-th order foliations. For additional details on jets, see Appendix \ref{ssec:jetsSections}.

\subsection{$k$-th order distributions}
Let $M$ be a manifold, let $W \subset M$ be an embedded submanifold and let $i_{W}: W \hookrightarrow M$ be the inclusion. In this section we define $k$-th order distributions. Intuitively, these are submodules of the $k$-jets of vector fields which can be realized by germs of distributions around $W$.

Let $\mathcal{O}$ denote the sheaf of smooth functions on $M$ and $\mathcal{I}_W$ the vanishing ideal of $W$. 

\begin{definition}
Let $E \to M$ be a vector bundle. The \emph{sheaf of $k$-jets of $E$ along $W$} is the sheaf of vector spaces over $W$ defined by $j_{W}^{k}(E) = i_{W}^{-1}(E/\mathcal{I}_W^{k+1}E). $
\end{definition}

Applying the $k$-jets to the algebra of smooth functions $\mathcal{O}$ (namely, the sections of the trivial bundle $M \times \rr $) gives rise to a sheaf of algebras over $W$
\[
\mathcal{N}_{W}^{k} := j_{W}^{k}(\mathcal{O}),
\]
which we think of as the sheaf of functions on the \emph{$k$-th order neighbourhood of $W$.} A choice of a tubular neighbourhood of $W$ induces an isomorphism 
\begin{equation} \label{nbhdnormalform}
\mathcal{N}_{W}^{k} \cong \bigoplus_{i = 0}^{k} \mathrm{Sym}_{\mathcal{O}_{W}}^{i}(\nu_{W}^{*}), 
\end{equation}
where $\nu_{W}$ is the normal bundle of $W$ and $\mathcal{O}_{W}$ is the sheaf of smooth functions on $W$. Given a vector bundle $E$, the sheaf of $k$-jets $j_{W}^{k}(E)$ is a locally free sheaf of $\mathcal{N}_{W}^{k}$-modules and hence may be thought of as a vector bundle over the $k$-th order neighbourhood.

There is a natural way of associating $\mathcal{N}^{k}_{W}$-submodules of $j_{W}^{k}(E)$ to $\mathcal{O}$-submodules of $E$ and vice versa. For this, we make use of the natural projection map 
\begin{equation*}
p_{k} : E \to (i_{W})_{*}(j_{W}^{k}E),
\end{equation*}
which extracts the $k$-jet of a section of $E$. Let $q_{k} : i_{W}^{-1}(E) \to j_{W}^{k}E$ be the corresponding map implied by the adjunction between $(i_W)_*$ and $i_{W}^{-1}$.  Then for an $\mathcal{O}$-submodule $\mathcal{F} \subset E$ we define its $k$-jet to be
\[
I_{k}(\mathcal{F}) := q_{k}(i_{W}^{-1}(\mathcal{F})) \subset j_{W}^{k}(E).
\]
Indeed, note that if $F \subset E$ is a subbundle, then $I_{k}(F) \cong j_{W}^{k}(F)$. Conversely, given $\mathcal{S} \subset j_{W}^{k}(E)$ a $\N^k_W$-submodule, define 
\[
P_{k}(\mathcal{S}) := p_{k}^{-1}( (i_{W})_{*} \mathcal{S}) \subset E, 
\] 
which consists of those sections of $E$ whose $k$-jet is contained in $\mathcal{S}$. Note that $I_{k}P_{k}(\mathcal{S}) = \mathcal{S}$, whereas $P_{k} I_{k}(\mathcal{F}) \supseteq \mathcal{F}$ is usually a larger subsheaf. Indeed, for a vector subbundle $F \subset E$, 
\begin{equation}\label{eq:tandescription}
P_{k}I_{k}(F) = F + \mathcal{I}_W^{k+1}E,
\end{equation}
the sheaf of all sections of $E$ which are contained in $F$ up to order $k$. 

Consider the morphism
\begin{equation} \label{eq:restrictionToW}
r_{k} : j_{W}^{k}(E) \to E|_{W}
\end{equation}
given by restricting sections to $W$ (alternatively, it is obtained by quotienting out the subsheaf $\mathcal{I}_WE$). Using this map, we can impose a constant rank condition on a submodule $\mathcal{S} \subset j^k_W(E)$ by requiring that $r_{k}(\mathcal{S})$ is a subbundle of $E|_W$. This allows us to make the following useful definition characterizing the subsheaves of $j^{k}_{W}E$ which arise from germs of subbundles of $E$. 

\begin{definition} 
A \emph{subbundle} of $j^{k}_{W}E$ is a locally free $\mathcal{N}_{W}^{k}$-submodule $\mathcal{S}$ such that $r_{k}(\mathcal{S}) \subseteq E|_{W}$ is a vector subbundle. 
\end{definition}
\begin{lemma} \label{realizabilitylemma}
A subsheaf $\mathcal{S} \subseteq j^{k}_{W}E$ is a subbundle if and only if $\mathcal{S} = j^{k}_{W}F$, for $F$ the germ of a subbundle of $E$ around $W$. 
\end{lemma}
\begin{proof}
It suffices to work using a tubular neighbourhood embedding $\pi: \nu_{W} \to W$ and hence we may use the isomorphism of Equation \ref{nbhdnormalform}. We can also assume that $E = \pi^*(V)$, for $V$ a vector bundle over $W$. Hence 
\[
j^{k}_{W}(E) \cong \bigoplus_{i = 0}^{k} \mathrm{Sym}_{\mathcal{O}_{W}}^{i}(\nu_{W}^{*}) \otimes V. 
\]
By viewing an element of $\mathrm{Sym}_{\mathcal{O}_{W}}^{i}(\nu_{W}^{*}) \otimes V$ as a polynomial section of $\pi^*(V)$, we obtain an $\mathcal{O}_{W}$-linear map $e: j^{k}_{W}(E) \to i_{W}^{-1}E$ which is a splitting of the projection $i_{W}^{-1}E \to j^{k}_{W}(E)$ arising in the definition of $k$-jets.

One direction of the lemma is clear: if $F$ is a subbundle of $E$, then $j^{k}_{W}F$ is a subbundle of $j^{k}_{W}E$. For the converse, let $\mathcal{S}$ be a subbundle of $j^{k}_{W}E$ and let $U = r_{k}(\mathcal{S}) $ be the corresponding subbundle of $E|_{W} \cong V$. Given a local basis $s_{1}, ..., s_{p}$ of $\mathcal{S}$, their image $\{ r_{k}(s_{i}) \}$ forms a local spanning set of $U$. By applying change of bases, we may assume that $r(s_{1}), ..., r(s_{l})$  forms a local basis of $U$, for some $0 \leq l \leq p$, and that the remaining basis elements get sent to zero: $r(s_{i}) = 0$. But such an $s_{i}$ is annihilated by $\mathrm{Sym}^{k}(\nu_{W}^*)$, which is not possible if it is part of a basis. Hence, $l = p$ and $e(s_{1}), ..., e(s_{p})$ are \emph{pointwise} linearly independent sections of $E$ near $W$. 

Now choose a subbundle $C \subset V$ which is complementary to $U$. This induces the decomposition 
\[
j^{k}_{W}(E) \cong j^{k}_{W}(\pi^*(U)) \oplus j^{k}_{W}(\pi^*(C)). 
\]
Let $pr_{U} : j^{k}_{W}(E) \to  j^{k}_{W}(\pi^*(U))$ be the induced projection map. The previous paragraph shows that we also have the decomposition $j_{W}^{k}(E) = \mathcal{S} \oplus j^{k}_{W}(\pi^*(C))$. By checking locally using a basis, we see that $p = pr_{U}|_{\mathcal{S}} : \mathcal{S} \to j^{k}_{W}(\pi^*(U))$ is an isomorphism. Hence $\varphi = p^{-1} \oplus id_{ j^{k}_{W}(\pi^*(C))}$ is an $\mathcal{N}_{W}^{k}$-linear automorphism of $j^{k}_{W}(E) $ sending $ j^{k}_{W}(\pi^*(U)) $ to $\mathcal{S}$. Let $\hat{\varphi}$ be an extension of $\varphi$ as a linear automorphism of $i_{W}^{-1}E$. Therefore, $F := \hat{\varphi}(i_{W}^{-1} \pi^*(U))$ is a subbundle of $i_{W}^{-1}E$ with the property that $j^{k}_{W}F = \mathcal{S}$. 
\end{proof}

The main objects of interest in this paper are subbundles of $j_W^k(TM)$:
\begin{definition}
A \emph{$k$-th order distribution} around $W$ is defined to be a subbundle of $j_{W}^{k}(TM)$. Given such a distribution $\mathcal{S}$, define 
\[ \mathrm{Tan}^{k}(W, \mathcal{S}) = P_{k}(\mathcal{S}) \subset \X(M),  \]
the $\mathcal{O}$-submodule of vector fields \emph{$k$-tangent} to $\mathcal{S}$.  
\end{definition}

Given a $k$-th order distribution $\mathcal{S}$, there exists a germ of distribution $\xi \subset TM$ around $W$ representing it by Lemma \ref{realizabilitylemma}. Equation \eqref{eq:tandescription} then provides the following explicit description of $\mathrm{Tan}^{k}(W, \mathcal{S})$:
\[ \mathrm{Tan}^{k}(W, \mathcal{S}) = \xi + \mathcal{I}_W^{k+1}TM.\hfill \]

\subsection{Involutivity and $k$-th order foliations} \label{ssec:LieRiehart}

The Lie bracket of vector fields does not descend to a bracket on the $k$-jets $j_{W}^{k}(TM)$. Instead, it descends to a map 
\[
j_{W}^{k}(TM) \times j_{W}^{k}(TM) \to j_{W}^{k-1}(TM). 
\]
This is because the ideal $\mathcal{I}_W$ is not preserved by differentiation. The solution to this problem is to restrict our attention to the subsheaf $\mathcal{A}_{W} \subset TM$ of vector fields which are tangent to $W$ since, by definition, they do preserve $\mathcal{I}_W$.

\begin{proposition} \label{derivationtheorem}
The Lie bracket of vector fields descends to a Lie bracket on $I_{k}(\mathcal{A}_{W}) \subset j^k_W(\mathcal{A}_W)$, endowing it with the structure of a Lie-Rinehart algebra over $\mathcal{N}_{W}^{k}$. That is: 
\begin{itemize}
\item[a)] It is a sheaf of $\mathcal{N}_{W}^{k}$-modules, 
\item[b)] It acts on $\mathcal{N}_{W}^{k}$ by derivations, 
\item[c)] Its Lie bracket satisfies the Leibniz rule with respect to $\mathcal{N}_{W}^{k}$. 
\end{itemize} 
Moreover, $I_{k}(\mathcal{A}_{W}) \cong \mathrm{Der}(\mathcal{N}_{W}^{k})$, the sheaf of derivations of $\mathcal{N}_{W}^{k}$. 
\end{proposition}
\begin{proof}
The bracket on $\mathcal{A}_{W}$ descends to $I_{k}(\mathcal{A}_{W})$ because $\mathcal{A}_{W}$ preserves $\mathcal{I}_W$. For the same reason, the differentiation action of $\mathcal{A}_{W}$ on $\mathcal{O}$ descends to an action of $I_{k}(\mathcal{A}_{W})$ on $\mathcal{N}_{W}^{k}$. That $I_{k}(\mathcal{A}_{W})$ is a Lie-Rinehart algebra over $\mathcal{N}_{W}^{k}$ then follows from the fact that $\mathcal{A}_{W}$ is a Lie-Rinehart algebra over $\mathcal{O}$. 

The action of $I_{k}(\mathcal{A}_{W})$ on $\mathcal{N}_{W}^{k}$ induces a natural map $I_{k}(\mathcal{A}_{W}) \to \mathrm{Der}(\mathcal{N}_{W}^{k})$, which we must show to be an isomorphism. We choose a tubular neighbourhood, allowing us to make use of Equation \ref{nbhdnormalform}. Let $\mathcal{I}_{(k)}$ be the ideal of $\mathcal{N}_{W}^{k}$ generated by $\nu_{W}^{*}$. Standard arguments show that an element of $\mathrm{Der}(\mathcal{N}_{W}^{k})$ is uniquely determined by a pair $(\omega, B)$, with 
\[
\omega \in \mathcal{N}_{W}^{k} \otimes TW, \qquad B \in \mathrm{Hom}(J^1(\nu_{W}^{*}), \mathcal{I}_{(k)}), 
\]
such that $B|_{T^*W \otimes \nu_{W}^{*}} = \omega \otimes id$. Here, $J^1(\nu_{W}^{*})$ is the jet bundle of $\nu_{W}^*$ and it lives in the sequence 
\[
0 \to T^*W \otimes \nu_{W}^* \to J^1(\nu_{W}^{*}) \to \nu_{W}^* \to 0. 
\]
Indeed, $\omega$ corresponds to the restriction of the derivation to $\mathcal{O}_{W}$ and $B$ corresponds to its restriction to $\nu_{W}^*$. The choice of a connection $\nabla$ on $\nu_{W}^*$ induces a splitting of the jet sequence and induces the isomorphism 
\[
\mathrm{Der}(\mathcal{N}_{W}^{k}) \cong  (\mathcal{N}_W^k \otimes TW) \oplus (\mathcal{I}_{(k)} \otimes \nu_{W}).
\]
The connection $\nabla$ also induces a decomposition $T\mathrm{tot}(\nu_{W}) = \pi^*(\nu_{W} \oplus TW)$, allowing us to write $\mathcal{A}_{W} = \pi^*(TW) \oplus \mathcal{I}_W  \pi^*(\nu_{W})$, and hence to conclude $I_{k}(\mathcal{A}_{W}) \cong  (\mathcal{N}_W^k \otimes TW) \oplus (\mathcal{I}_{(k)} \otimes \nu_{W})$, establishing the desired isomorphism. 
\end{proof}

For the following definition, recall the restriction map $r_k$ defined in Equation \ref{eq:restrictionToW}. 
\begin{definition}
A \emph{$k$-th order foliation} around $W$ is a subbundle $F \subset j_{W}^{k}(TM)$ such that $r_{k}(F) = TW$ and which is closed under the Lie bracket. 
\end{definition}
\begin{remark}
The condition $r_{k}(F) = TW$ implies that $F \subset  \mathrm{Der}(\mathcal{N}_{W}^{k})$. Therefore, it makes sense to ask for $F$ to be closed under the Lie bracket. 
\end{remark}
\begin{remark}
The reader may have noticed that it is natural to require more generally that $r_{k}(F) \subset TW$ be a regular foliation in the above definition. We do not consider this more general definition because we are ultimately interested in studying singular foliations of $b^{k+1}$-type. 
\end{remark}

\begin{remark}\label{rem:grass}
Assume that $M$ has dimension $n$ and that $W$ has codimension $l$. Let $\Gr(TM, n - l) \to M$ be the Grassmanian bundle of dimension $(n-l)$-subspaces of $TM$. Since it is a fibre bundle, it is possible to take its bundle of $k$-jets $J^k\Gr(TM,n-l)$. A subbundle $F \subset j^k_W(TM)$ may be viewed as a \emph{holonomic} section of the restricted jet-bundle $J^k\Gr(TM,n-l)|_W$. See Appendix \ref{sec:DiffTop} for more details.
\end{remark}

\subsubsection{The action of diffeomorphisms}

We conclude this section by discussing the action of diffeomorphisms on $k$-th order foliations. Given the germ of a diffeomorphism $f$ of $M$ fixing $W$ pointwise, we can consider its $k$-jet $j^k_Wf$ along $W$. This can be defined in the standard way, as in Appendix \ref{sec:DiffTop}, or through its action on the sheaf $\mathcal{N}_{W}^{k}$. Indeed, the pullback of functions descends to the $k$-th order neighbourhood to define a map of sheaves $(j^k_Wf)^* : \mathcal{N}_{W}^{k} \to \mathcal{N}_{W}^{k}$. The pushforward map $df$ on vector fields also descends to define an isomorphism of Lie algebras $df: I_{k}(\mathcal{A}_{W}) \to I_{k}(\mathcal{A}_{W})$ which satisfies the following equations 
\begin{equation} \label{pushforwardidentities}
df( g X) = (j^k_Wf^{-1})^*(g) df(X), \qquad (j^k_Wf)^*( df(X)(g) ) = X( (j^k_Wf)^*(g)), 
\end{equation}
where $X \in  I_{k}(\mathcal{A}_{W})$ and $g \in \mathcal{N}_{W}^k$. This implies that the map $df$ acts by pushing forward $k$-th order foliations. Furthermore, by the second identity of Equation \ref{pushforwardidentities}, the action of $df$ only depends on $j^k_Wf$. 

\begin{lemma}  \label{lem:factoringViakJets}
The group of $k$-jets of diffeomorphisms of $M$ that fix $W$ pointwise acts by pushforward on the set of $k$-th order foliations. 
\end{lemma}


\section{Singular foliations of $b^{k+1}$-type}\label{sec:singfols}
In this section we will show that singular foliations of $b^{k+1}$-type correspond to the $k$-th order foliations introduced in the previous section.

%

Let $F \subset \mathrm{Der}(\mathcal{N}_{W}^{k})$ be a $k$-th order foliation and let $\xi \subset TM$ be a germ of a distribution around $W$ representing it. Given two sections $X, Y \in \Gamma(\xi)$, their bracket $[X,Y]$ may not be in $\Gamma(\xi)$, but the projection $q_{k}([X,Y])$ to $j_{W}^{k}(TM)$ does lie in $F = j_{W}^{k}(\xi)$. Therefore, 
\[
q_{k}[\xi, \xi] \subseteq j^{k}_{W}(\xi),
\]
meaning that $\xi$ is \emph{involutive up to degree $k$}. Moreover, this implies that 
\[
[\xi, \xi] \subseteq \Tan^{k}(W, F) = \xi + \mathcal{I}_W^{k+1} TM. 
\]
Since $\xi$ is tangent to $W$, it also preserves $\mathcal{I}_W$, and hence $\Tan^{k}(W, F)$ is closed under the Lie bracket. 

\begin{proposition} \label{Singfoliskorderfol}
Let $F \subset \mathrm{Der}(\mathcal{N}_{W}^{k})$ be a $k$-th order foliation. Then:
\begin{itemize}
\item[a)] $\Tan^{k}(W, F)$ is a singular foliation of $b^{k+1}$-type.
\item[b)] $\Tan^{k}(W, F)$ is locally free if and only if $W$ has codimension $1$, in which case it defines a Lie algebroid over $M$.
\item[c)] The assignment $F \mapsto \Tan^{k}(W, F)$ establishes a bijection between $k$-th order foliations and singular foliations of $b^{k+1}$-type for $(W,M)$. 
\end{itemize}
\end{proposition}
\begin{proof}
By the preceding discussion and the decomposition $\Tan^{k}(W, F) = \xi + \mathcal{I}_W^{k+1} TM$, it is clear that $\Tan^{k}(W, F) $ is locally defined, locally finitely generated, and closed under the Lie bracket. Hence, it is a singular foliation. Given a point $p \in W$, choose a local foliation $C \subset TM$ which is complementary to $W$. Then $C$ is also complementary to $\xi$ and we can write $\Tan^{k}(W, F) = \xi \oplus \mathcal{I}_W^{k+1}C$. Let $i: L \to M$ be a leaf of $C$ that intersects $W$ at the point $p$. Then by the splitting theorem for singular foliations \cite{MR4065067}, we have an isomorphism 
\[
\Tan^{k}(W, F) \cong i^{!}\Tan^{k}(W, F) \times TW. 
\] 
Here, $i^{!}\Tan^{k}(W,F)$ is simply the collection of all $X|_{L}$, where $X \in \Tan^{k}(W, F)$ is tangent to $L$, which we see is equal to $\mathcal{I}_W^{k+1}TL$. In local coordinates $(z_{1}, ..., z_{l})$ centred at $p$, this is generated by $z^{I} \partial_{z_{i}}$, where $z^{I} = z_{1}^{i_{1}} ... z_{l}^{i_{l}}$, with $I = (i_{1}, ..., i_{l})$ running over all multi-indices of weight $k+1$. This establishes the desired local model showing that $\Tan^{k}(W, F)$ is indeed of $b^{k+1}$-type. 

Using the decomposition $\Tan^{k}(W, F) = \xi \oplus \mathcal{I}_W^{k+1}C$, and noting that both $\xi$ and $C$ are vector bundles, we see that $\Tan^{k}(W, F)$ is locally free if and only if $\mathcal{I}_W^{k+1}$ is, which is the case precisely when $W$ has codimension $1$. 

Let $\mathcal{A}$ be a foliation of $b^{k+1}$-type. Using the local model, we see that near a point $p \in W$, there is a foliation $\xi$ tangent to $W$ such that 
\[
\mathcal{A} = \xi + \mathcal{I}_W^{k+1} TM = \Tan^{k}(W, j_{W}^{k}\xi) = P_{k}(j_{W}^{k}\xi). 
\]
Therefore, $I_{k}(\mathcal{A}) = j^{k}_{W}\xi$. Hence, $I_{k}(\mathcal{A})$ is a $k$-th order foliation, since it is a local definition. It then follows that $I_{k}$ and $P_{k}$ determine a bijection between $k$-th order foliations and singular foliations of $b^{k+1}$-type. 
\end{proof}
\begin{remark}
There is a unique $0$-th order foliation given by the tangent bundle $TW$. The associated singular foliation of $b^1$-type is $\mathcal{A}_{W} = \Tan^{0}(W, TW)$, the vector fields tangent to $W$. 
\end{remark}

\begin{corollary}[Frobenius theorem] \label{Frobenius} Let $F \subset \mathrm{Der}(\mathcal{N}_{W}^{k})$ be a $k$-th order foliation. Then around any point $p \in W$, there are local coordinates $(x_{1}, ..., x_{r}, z_{1}, ... , z_{l})$ centred at $p$ for which $W = \{ z_{1} = ... = z_{l} = 0 \}$ and such that $F = j_{W}^{k}\langle \partial_{x_{1}}, ..., \partial_{x_{r}} \rangle$. 
\end{corollary}
\begin{proof}
By Proposition \ref{Singfoliskorderfol}, there is a foliation of $b^{k+1}$-type $\mathcal{A}$ such that $F = I_{k}(\mathcal{A})$. Choosing local coordinates that put $\mathcal{A}$ into the $b^{k+1}$-local model, we obtain the desired normal form for $F$. 
\end{proof}


\section{Extensions: Part 1}\label{sec:extending}
Given a vector bundle $E$, there are forgetful maps $u_{k} : j^{k}_{W}(E) \to j^{k-1}_{W}(E)$ between the jet spaces. In particular, there are maps $\mathrm{Der}(\mathcal{N}_{W}^{k}) \to \mathrm{Der}(\mathcal{N}_{W}^{k-1})$ and therefore a $k$-th order foliation $F$ determines a $(k-1)$-st order foliation $u_{k}(F)$. The \textbf{Extension Problem} asks whether, for a given $k$-th order foliation $F$, there is a $(k+1)$-st order foliation $\tilde{F}$ such that $u_{k+1}(\tilde{F}) = F$. We will return to this problem from several different perspectives in this paper. In this section, we explain how this problem is controlled by the restriction $\Tan^{k}(W, F)|_{W}$, which is a Lie algebroid over $W$ even when $\Tan^{k}(W, F)$ is not locally free. 

\begin{proposition} \label{restrictedLiealgebroid}
Let $\mathcal{A}$ be a foliation of $b^{k+1}$-type and let $\mathcal{A}|_{W} = i_{W}^{-1}(\mathcal{A}/\mathcal{I}_{W} \mathcal{A})$. Then $\mathcal{A}|_{W} $ is a Lie algebroid over $W$ and the anchor map defines the following exact sequence 
\begin{equation} \label{restrictedanchorsequence}
0 \to \mathrm{Sym}^{k+1}(\nu_{W}^*) \otimes \nu_{W} \to \mathcal{A}|_{W}  \to TW \to 0. 
\end{equation}
\end{proposition}
\begin{proof}
Because $\mathcal{A}$ is tangent to $W$, it preserves $\mathcal{I}_W$ and hence the Lie bracket on sections descends to $\mathcal{A}|_{W}$. Moreover, restricting sections of $\mathcal{A}$ to $W$ defines a bracket preserving map to $TW$. To see that $\mathcal{A}|_{W}$ is locally free, choose a distribution $\xi$ that is tangent to order $k$ to $\mathcal{A}$ and let $C$ be a complementary distribution. Then $\mathcal{A} = \xi \oplus \mathcal{I}_W^{k+1}C$, so that $\mathcal{A}|_{W} = \xi|_{W} \oplus (\mathcal{I}_W^{k+1}C/\mathcal{I}_W^{k+2}C) \cong TW \oplus (\mathrm{Sym}^{k+1}(\nu_{W}^*) \otimes \nu_{W})$. Hence, $\mathcal{A}|_{W}$ is a Lie algebroid. It is also clear that the image of the anchor $\rho$ is $TW$. To determine the kernel of the anchor map, note first that $\mathcal{I}_W^{k+1} TM$ is a subsheaf of $\mathcal{A}$ and this induces a map 
\[
\mathrm{Sym}^{k+1}(\nu_{W}^*) \otimes TM|_{W} \cong i^{-1}_{W}(\mathcal{I}_W^{k+1} TM/ \mathcal{I}_W^{k+2} TM) \to \mathcal{A}|_{W}
\]
whose image lies in the kernel of $\rho$. The subsheaf $\mathcal{I}_W^{k+1} \xi \subset \mathcal{I}_W^{k+1} TM$ gets sent to $0$, implying that the above map descends to $\mathrm{Sym}^{k+1}(\nu_{W}^*) \otimes \nu_{W} \to \mathrm{Ker}(\rho)$. Using the complementary distribution $C$ shows that this is an isomorphism.  
\end{proof}

Let $F$ be a $k$-th order foliation around $W$ and let $\mathcal{A} = \mathrm{Tan}^{k}(W,F)$ be the corresponding singular foliation of $b^{k+1}$-type. An extension of $F$ to a $(k+1)$-st order foliation $\tilde{F}$ is equivalent to a sub-singular foliation $\tilde{\mathcal{A}} \subset \mathcal{A}$ of $b^{k+2}$-type. Using this equivalence, we obtain the following solution to the extension problem. 

\begin{theorem} \label{extensionbijpart1}
Let $F$ be a $k$-th order foliation around $W$ and let $\mathcal{A} = \mathrm{Tan}^{k}(W,F)$ be the corresponding singular foliation of $b^{k+1}$-type. Then extensions of $F$ to order $k+1$ are in bijection with Lie algebroid splittings $\sigma : TW \to \mathcal{A}|_{W}$ of sequence \eqref{restrictedanchorsequence}. 
Explicitly, given a Lie algebroid splitting $\sigma$ the corresponding singular foliation of $b^{k+2}$-type is given by
\begin{equation}\label{eq:elem}
\tilde{\mathcal{A}} = \set{X \in \A : X|_W \in \im \sigma},
\end{equation}
where $X|_W$ denotes the section of $\mathcal{A}|_{W} = i_{W}^{-1}(\mathcal{A}/\mathcal{I}_{W} \mathcal{A})$ induced by $X$.
\end{theorem}
\begin{proof}
Let $\tilde{\mathcal{A}} \subset \mathcal{A}$ be a subsingular foliation of $b^{k+2}$-type. This induces a Lie algebroid morphism $f: \tilde{\mathcal{A}}|_{W} \to \mathcal{A}|_{W}$. Represent $\tilde{\mathcal{A}}$ by a germ of distribution $\xi$ and let $C$ be a complementary distribution. Then, as in the proof of Proposition \ref{restrictedLiealgebroid}, $\tilde{\mathcal{A}}|_{W} = TW \oplus (\mathcal{I}_W^{k+2}C/\mathcal{I}_W^{k+3}C)$ and $\mathcal{A}|_{W} = TW \oplus  (\mathcal{I}_W^{k+1}C/\mathcal{I}_W^{k+2}C)$. In this decomposition, $f$ is the identity on the first summand and vanishes on the second. Hence, it factors through a Lie algebroid splitting $\sigma: TW \to \mathcal{A}|_{W}$. 

In the other direction, suppose we are given a splitting $\sigma$. This determines a vector bundle map $\mathcal{A}|_{W} \to \mathrm{Sym}^{k+1}(\nu_{W}^*) \otimes \nu_{W}$ and hence a morphism of sheaves $g: \mathcal{A} \to i_{W,*}(\mathrm{Sym}^{k+1}(\nu_{W}^*) \otimes \nu_{W})$. Let $\tilde{\mathcal{A}} = \mathrm{Ker}(g)$, which is precisely the subsheaf of Equation \eqref{eq:elem}. To show that this subsheaf has the desired properties, choose a decomposition $\mathcal{A} = \xi \oplus \mathcal{I}_W^{k+1}C$ as above. By locally lifting a basis of the subalgebroid $\sigma(TW) \subset \mathcal{A}|_{W}$ to $\mathcal{A}$, we obtain a (locally defined) distribution $\eta \subset TM$ such that  $\mathcal{A} = \eta \oplus \mathcal{I}_W^{k+1}C$ and $\eta/\mathcal{I}_W\mathcal{A} = \sigma(TM)$. Then, since the image of $\eta$ in $\mathcal{A}|_{W}$ is closed under the bracket
\[
[\eta, \eta] \subset \eta + \mathcal{I}_W \mathcal{A} \subset \eta + \mathcal{I}_W^{k+2}TM. 
\]
Hence $j_{W}^{k+1}(\eta)$ is a $(k+1)$-st order foliation and $\eta \oplus \mathcal{I}_W^{k+2}C$ is a $b^{k+2}$-algebroid. But this is precisely the subsheaf $\tilde{\mathcal{A}}$. 
\end{proof}
\begin{remark}
In the case where $W$ has codimension one, the construction in Theorem \ref{extensionbijpart1} is known as the \emph{elementary modification} of the Lie algebroid $\A$ along the Lie subalgebroid $\sigma(TW)$. Hence, Theorem \ref{extensionbijpart1} allows us to construct any hypersurface algebroid of $b^{k+1}$-type $\A$ as a sequence of elementary modifications starting from $TM$.

In \cite{Gualtieri-Li-2012} it is shown that a blow-up procedure of a Lie groupoid along a closed Lie subgroupoid corresponds to an elementary modification of the Lie algebroids. For holomorphic hypersurface algebroids of $b^{k+1}$-type over curves this relationship is explained in \cite{gualtieri2018stokes} and used to construct integrations. However, in general this procedure for constructing integrations may fail because there is no guarantee that the Lie subalgebroid $\sigma(TW)$ integrates to a closed Lie subgroupoid. 

Groupoids associated to hypersurface algebroids were considered in \cite{francis2024singular}. 
\end{remark}

Recall that there is a unique $0$-th order foliation around $W$ and that the associated singular foliation of $b^{1}$-type is $\mathcal{A}_{W}$, the sheaf of vector fields tangent to $W$. Furthermore, recall the following result.
\begin{lemma}
The restriction of $\mathcal{A}_W$ to $W$ is isomorphic to the Atiyah algebroid of the normal bundle: 
\[ \mathcal{A}_{W}|_{W} \cong \at(\nu_{W}).\]
\end{lemma}
\begin{proof}
This can be seen by observing that the action of $\mathcal{A}_W$ on $\mathcal{O}$ restricts to an action on $\mathcal{I}_{W}$, which descends to a representation of $\mathcal{A}_{W}|_{W}$ on $\nu_{W}^{*} = i_{W}^{-1}(\mathcal{I}_{W}/\mathcal{I}_{W}^2)$. This induces an isomorphism $\mathcal{A}_{W}|_{W} \to \at(\nu_{W}^*) \cong \at(\nu_{W})$. 
\end{proof}
As a result, for $k = 0$ the sequence \eqref{restrictedanchorsequence} has the following form 
\begin{equation*}
0 \rightarrow {\rm End}(\nu_W) \rightarrow \at(\nu_{W}) \rightarrow TW \rightarrow 0,
\end{equation*}
and its algebroid splittings coincide with flat connections on $\nu_{W}$. Combined with Theorem \ref{extensionbijpart1} we obtain the following result. 

\begin{corollary} \label{existenceofflatconnection}
There is a bijection between $1$-st order foliations around $W$ and flat connections on the normal bundle $\nu_{W}$. In particular, if there exists a $k$-th order foliation around $W$, for some $k \geq 1$, then the normal bundle $\nu_{W}$ admits a flat connection. 
\end{corollary}

Now let $F$ be a $k$-th order foliation for $k \geq 1$. This induces a flat connection on the normal bundle $\nu_{W}$ and hence on the kernel $\mathrm{Sym}^{k+1}(\nu_{W}^*) \otimes \nu_{W}$ of sequence \eqref{restrictedanchorsequence}. In particular, $\Omega_{W}^{\bullet}(\mathrm{Sym}^{k+1}(\nu_{W}^*) \otimes \nu_{W})$ acquires the structure of a chain complex. 

\begin{definition}\label{def:extension}
Let $F$ be a $k$-th order foliation around $W$, for $k \geq 1$. The \emph{extension class} $e(F) \in H^{2}(W, \mathrm{Sym}^{k+1}(\nu_{W}^*) \otimes \nu_{W})$ of $F$ is defined to be the extension class of the Lie algebroid sequence \eqref{restrictedanchorsequence}. It can be represented by the following closed $2$-form 
\[
\Omega(X, Y) = [ \sigma(X), \sigma(Y)] - \sigma([X,Y]),
\]
where $\sigma : TW \to \mathcal{A}|_{W}$ is any vector bundle splitting (i.e. it need not respect the Lie bracket). 
\end{definition}

\begin{corollary}\label{cor:firstextensionclass}
Let $F$ be a $k$-th order foliation around $W$, for $k \geq 1$. Then $F$ admits an extension to a $(k+1)$-st order foliation if and only if its extension class vanishes: $e(F) = 0$. 
\end{corollary}

\section{Subsheaves of invariant functions}\label{sec:invfuns}
Much like actual foliations, $k$-th order foliations can be described locally as `level sets of submersions'. In this section we describe how this can be made precise, giving us an alternative characterisation of $k$-th order foliations. Note that the material in this section is not necessary for the remainder of the paper. 

 Assume now that $W$ has codimension $l$ and that $k \geq 1$. A function $g \in \mathcal{N}_{W}^{k}$ that vanishes along $W$ (i.e. for which $r_{k}(g) = 0$) has a well-defined normal derivative $dg \in \nu_{W}^{\ast}$. 

\begin{definition}
A tuple $f = (f_{1}, ..., f_{l})$ with $f_{i} \in \mathcal{N}_{W}^{k}(U)$ defined over an open neighbourhood $U \subset W$ is said to be a \textbf{local defining function} if each $f_{i}$ vanishes along $W$ and the tuple $(df_{1}, ..., df_{l})$ forms a basis of $\nu_{W}^{*}|_{U}$. 
\end{definition}

Let $\mathcal{O}_{\mathbb{R}^l}$ denote the smooth functions on $\mathbb{R}^l$ and let 
\[
N(k) = j_{0}^{k}(\mathcal{O}_{\mathbb{R}^l}) = \mathbb{R}[z_{1}, ..., z_{l}]/I^{k+1}
\]
be the algebra of $k$-jets of smooth functions. Here, $I = (z_{1}, ..., z_{l})$ is the vanishing ideal of the origin, which is generated by the coordinate functions. Given a local defining function $f$ over $U$, there is an injective morphism of $\mathbb{R}$-algebras 
\[
N(f) : N(k) \to \mathcal{N}_{W}^{k}(U), \qquad h(z_{1}, ..., z_{l}) \mapsto h(f_{1}, ..., f_{l}). 
\]

\begin{definition} 
A \textbf{sheaf of invariant functions} in $\mathcal{N}_{W}^{k}$ is a subsheaf of $\mathbb{R}$-algebras $\mathcal{C} \subset \mathcal{N}_{W}^{k}$ with the property that for every $p \in W$, there is an open neighbourhood $p \in U \subset W$ and a local defining function $(f_{1}, ..., f_{l}) \in \mathcal{C}(U)$, such that the induced homomorphism 
\[
N(f): N(k) \to \mathcal{C}|_{U}
\]
is an isomorphism. 
\end{definition}

We will show that a sheaf of invariant functions is equivalent to the data of a $k$-th order foliation. First, given a $k$-th order foliation $F$, we define its sheaf of invariant functions as follows: 
\[
\mathcal{C}_{F} := \{ f \in \mathcal{N}_{W}^{k} \ | \ V(f) = 0 \text{ for all } V \in F \} \subset \mathcal{N}_{W}^{k}. 
\]
Conversely, given a sheaf of invariant functions $\mathcal{C}$, we define its annihilator as follows: 
\[
F_{\mathcal{C}} := \{ V \in \mathrm{Der}(\mathcal{N}_{W}^{k}) \ | \ V(f) = 0 \text{ for all } f \in \mathcal{C} \} \subset \mathrm{Der}(\mathcal{N}_{W}^{k}). 
\] 
\begin{lemma}
The sheaf $\mathcal{C}_{F}$ is a sheaf of invariant functions whereas $F_{\mathcal{C}}$ is a $k$-th order foliation. 
\end{lemma}
\begin{proof}
That $\mathcal{C}_{F}$ is a sheaf of invariant functions follows from a calculation using the local model provided by Corollary \ref{Frobenius}. To see that $F_{\mathcal{C}}$ is a $k$-th order foliation, we use the fact that for any local defining function $f = (f_{1}, ..., f_{l})$, there are local coordinates $(x_{1}, ..., x_{r}, z_{1}, ..., z_{l})$ such that $W = \{ z_{1} = ... = z_{l} = 0 \}$ and such that $f_{i}$ is the $k$-jet of $z_{i}$. 
\end{proof}
\begin{proposition}
There is a bijection between $k$-th order foliations around $W$ and sheaves of invariant functions in $\mathcal{N}_{W}^{k}$. 
\end{proposition}

\section{The $k$-th order Atiyah algebroid} \label{sec:Atiyah}

We now fix a tubular neighbourhood embedding for $W \subset M$. Hence, we view $W$ as the zero section in a rank $l$ vector bundle $\pi : E \to W$. This choice of tubular neighbourhood is significant and many of the upcoming constructions and statements will depend on it. The most important consequence of choosing a tubular neighbourhood is that it makes available the \emph{$k$-th order Atiyah algebroid}, a Lie algebroid over $W$ whose properties we study in this section. Ultimately, this algebroid will enable us to construct the \emph{holonomy} of a $k$-th order foliation. 

\begin{definition}
The \emph{$k$-th order Atiyah algebroid} $\at^{k}(E)$ is the subsheaf of $\mathrm{Der}(\mathcal{N}_{W}^{k})$ consisting of $k$-jets of $\pi$-projectable vector fields. 
\end{definition}

\begin{lemma} \label{Atiyahalgebroiddecompsition}
Let $\pi : E \rightarrow W$ be a vector bundle. The $k$-th order Atiyah algebroid $\at^{k}(E)$ is a Lie algebroid over $W$. It has the following canonical decomposition
\begin{equation} \label{Atiyahdecomposition}
\at^{k}(E) = \at^{1}(E) \oplus (\bigoplus_{i = 2}^{k}\mathrm{Sym}^{i}(E^*) \otimes E), 
\end{equation}
where $\at^{1}(E)$ is the Atiyah algebroid, consisting of linear vector fields on the total space of $E$, and where the second summand consists of $k$-jets of vertical vector fields vanishing to order at least $2$ along $W$. This second summand is an ideal and its bracket with $\at^{1}(E)$ is induced by the canonical $\at^{1}(E)$ representation on $E$.
\end{lemma}
\begin{proof}
Because of the map $\pi : E \to W$, $\mathrm{Der}(\mathcal{N}_{W}^{k})$ inherits the structure of a vector bundle over $W$. The subsheaf $\at^{k}(E)$ is easily seen to be a subbundle (cf. the proof of Proposition \ref{derivationtheorem}) which is also closed under the Lie bracket. In fact, because its sections are $\pi$-projectable, $\at^{k}(E)$ is a Lie algebroid over $W$.

For $k = 1$, we can identify $\at^{1}(E)$ with the usual Atiyah algebroid. There is a Lie algebroid morphism $ \at^{k}(E) \to \at^{1}(E)$ given by projecting a $k$-jet to a $1$-jet. This map has a splitting because a section of $\at^1(E)$ gives rise to a global linear vector field on $E$. This induces the decomposition of Equation \ref{Atiyahdecomposition}. 
\end{proof}

The anchor map for $\at^{k}(E)$ is given by the projection $\pi_{*}$ and the induced anchor sequence is given by 
\begin{equation} \label{anchorsequence}
0 \to \ad^{k}(E) \to \at^{k}(E) \to TW \to 0,
\end{equation}
where $\ad^{k}(E) = \bigoplus_{i = 1}^{k}\mathrm{Sym}^{i}(E^*) \otimes E$ is the bundle of $k$-jets of vertical vector fields. 

\begin{lemma} \label{Explicitbracket}
The Lie bracket on $\ad^{k}(E) = \bigoplus_{i = 1}^{k}\mathrm{Sym}^{i}(E^*) \otimes E$ is given by 
\[
[p_{1} \otimes v_{1}, p_{2} \otimes v_{2}] = (p_{1} \cdot \iota_{v_{1}} p_{2}) \otimes v_{2} - (p_{2} \cdot \iota_{v_{2}} p_{1}) \otimes v_{1},
\]
where $\iota_{v}$ is the derivation of $\mathcal{N}_{W}^{k}$ extending the pairing between $E$ and $E^*$.  
\end{lemma}

\subsection{Lie algebroid splittings} \label{Liealgsplittings}

Now let $F \subset \mathrm{Der}(\mathcal{N}_{W}^{k})$ be a $k$-th order foliation. The intersection $F \cap \at^{k}(E)$ consists of sections of $F$ that are $\pi$-projectable. If $\xi \subset TE$ is the germ of a distribution $k$-tangent to $F$, then $d\pi$ induces an isomorphism $\xi \cong \pi^*(TW)$. The intersection $F \cap \at^{k}(E)$ is then canonically identified with $TW$, the horizontal vector fields that are pulled-back from $W$. From this we conclude the following two statements: 
\begin{itemize}
\item $F \cap \at^{k}(E)$ is a Lie subalgebroid of $\at^{k}(E)$ which projects isomorphically to $TW$,
\item the $k$-th order foliation $F$ is generated over $\mathcal{N}_{W}^{k}$ by the intersection $F \cap \at^{k}(E)$. 
\end{itemize} 
Hence we conclude the following useful result. 
\begin{theorem} \label{splittingprop}
Given the choice of a tubular neighbourhood for $W$, there is a bijection between $k$-th order foliations $F \subset \mathrm{Der}(\mathcal{N}_{W}^{k})$ and Lie algebroid splittings $\sigma : TW \to \at^{k}(E)$. Moreover, this bijection only depends on the $k$-jet of the tubular neighbourhood embedding.
\end{theorem}

In combination with the decomposition of Equation \ref{Atiyahdecomposition}, Theorem \ref{splittingprop} gives rise to an explicit description of $k$-th order foliations in terms of differential forms. To see this, recall that a flat connection $\nabla$ on $E$ endows $\Omega_{W}^{\bullet}(\ad^{k}(E))$ with the structure of a dgla. 

\begin{corollary}[Maurer-Cartan equation] \label{MCeq}
Lie algebroid splittings $\sigma : TW \to \at^{k}(E)$ can be decomposed as $\sigma = \nabla + \sum_{i = 1}^{k-1}\eta_{i}$, where $\nabla : TW \to \at^1(E)$ is a flat linear connection on $E$ and $\eta_{i} \in \Omega_{W}^{1}(\mathrm{Sym}^{i+1}(E^*) \otimes E)$ satisfy the Maurer-Cartan equation 
\[
d^{\nabla} \eta_{r} + \frac{1}{2} \sum_{i + j = r} [\eta_{i}, \eta_{j}] = 0,
\]
for all $ r = 1, ..., k-1$.
\end{corollary}

Another interesting consequence of Theorem \ref{splittingprop} is that it provides a canonical way of extending any $k$-th order foliation to a germ of distribution. Indeed, by viewing sections of $\at^{k}(E)$ as polynomial vector fields on $E$, the splitting $\sigma$ provides a distribution $\xi_{\sigma} \subset TE$. In terms of the Maurer-Cartan data of Corollary \ref{MCeq}, it is generated by the vector fields 
\[
\sigma(X) = \nabla_{X} + \sum_{i = 1}^{k-1}\eta_{i}(X),
\]
for $X \in \mathfrak{X}(W)$. 

Finally, the singular foliation of $b^{k+1}$-type $\mathcal{A}(\sigma)$ associated to a splitting $\sigma$ has the following description 
\begin{equation} \label{bkfoldecomp}
\mathcal{A}(\sigma) = \xi_{\sigma} \oplus \mathcal{I}_W^{k+1} \pi^*(E). 
\end{equation}
This immediately leads to a convenient description of the Lie algebroid $\mathcal{A}(\sigma)|_{W}$ of Proposition \ref{restrictedLiealgebroid}. 
\begin{corollary} \label{restrictedextendedatiyah}
Let $\mathcal{A}(\sigma)$ be a singular foliation of $b^{k+1}$-type associated to a splitting $\sigma: TW \to \at^{k}(E)$. Then the Lie algebroid $\mathcal{A}(\sigma)|_{W}$ is isomorphic to the fibre product of $\sigma$ and $\at^{k+1}(E) \to \at^{k}(E)$: 
\[
\mathcal{A}(\sigma)|_{W} \cong TW \oplus_{\at^{k}(E)} \at^{k+1}(E). 
\]
\end{corollary}
\begin{proof}
Let $\tilde{F} \subset \at^{k+1}(E)$ denote the subsheaf of $\pi$-projectible $(k+1)$-jets of sections of $\mathcal{A}(\sigma)$. Using the decomposition of Equation \ref{bkfoldecomp}, we see that it is isomorphic to $TW \oplus (\mathrm{Sym}^{k+1}(E^*) \otimes E) $, where $TW$ is embedded using $\sigma$. Hence, this is a Lie subalgebroid of the form claimed in the statement. The morphism $i^{-1}_{W}(\mathcal{A}(\sigma)) \to \mathcal{A}(\sigma)|_{W}$ descends to a morphism $\tilde{F} \to  \mathcal{A}(\sigma)|_{W}$ and the decomposition of Equation \ref{bkfoldecomp} shows that it is an isomorphism. 
\end{proof}

\begin{remark}
Recall Theorem \ref{extensionbijpart1} which states that splittings of the sequence \eqref{restrictedanchorsequence} correspond to extensions of a $k$-th order foliation to order $k+1$. Corollary \ref{restrictedextendedatiyah} provides an immediate proof of this statement, as the description of $\A(\sigma)|_W$ makes it clear that a splitting of sequence \eqref{restrictedanchorsequence} corresponds to a splitting of $\at^{k+1}(E)$, which is an extension by Theorem \ref{splittingprop}. However, do note that this proof comes at the expense of using a tubular neighbourhood embedding via the use of $\at^k(E)$.
\end{remark}
Since there is a non-bracket preserving splitting of sequence \eqref{restrictedanchorsequence} induced by $\xi_{\sigma}$, we obtain a concrete expression for the extension class. 
\begin{corollary}\label{cor:explicitextensionclass}
Let $\sigma = \nabla + \sum_{i = 1}^{k-1}\eta_{i} : TW \to \at^{k}(E)$ be a Lie algebroid splitting corresponding to a $k$-th order foliation $F(\sigma)$. Then the extension class is given by the following expression 
\[
e(F(\sigma)) = \frac{1}{2} \sum_{i + j = k} [\eta_{i}, \eta_{j}] \in H^2(W;\Sym^{k+1}(E^*)\otimes E). 
\]
\end{corollary}

\subsection{Holonomy}

As a transitive Lie algebroid over $W$, $\at^{k}(E)$ is the Atiyah algebroid of a principal bundle. In this section we will explain this perspective and use it to construct the holonomy invariant of a $k$-th order foliation. 

\subsubsection{The structure group} \label{structuregroup}
Let $G_{k,l} := J^{k}\mathrm{Diff}(\mathbb{R}^l, 0)$ be the Lie group of $k$-jets of diffeomorphisms of $\mathbb{R}^l$ fixing the origin. Equivalently, $G_{k,l} \cong \mathrm{Aut}(N(k))$, the group of algebra automorphisms of $N(k) = \mathbb{R}[z_{1}, ..., z_{l}]/I^{k+1}$. By forgetting the algebra structure on $N(k)$, we get an embedding $G_{k, l} \to \mathrm{GL}(N(k))$ corresponding to the natural representation of $G_{k,l}$ on the finite dimensional vector space underlying $N(k)$. Hence, $G_{k,l}$ is a linear algebraic group. 

Setting $k = 1$, we find $G_{1,l} = \mathrm{GL}(\mathbb{R}^l)$. By sending a $k$-jet to a $(k-1)$-jet, we get a tower of surjective homomorphisms 
\begin{equation} \label{tower}
... \to G_{k,l} \to G_{k-1,l} \to ... \to G_{2,l} \to G_{1,l} = \mathrm{GL}(\mathbb{R}^l). 
\end{equation}
Furthermore, since an element of $G_{1,l}$ may be interpreted as a linear automorphism of $\mathbb{R}^l$, there is a natural inclusion $j_{k}: G_{1,l} \to G_{k,l}$ for all $k$, which splits the map $\pi_{k} : G_{k,l} \to G_{1,l}$. Hence, we have a semi-direct product decomposition 
\begin{equation}\label{Levy}
G_{k,l} \cong K_{k,l} \rtimes \mathrm{GL}(\mathbb{R}^l),
\end{equation}
where $K_{k,l} = \mathrm{Ker}(\pi_{k})$, the subgroup of automorphisms whose linear approximation is the identity. Given $\phi \in K_{k,l}$, we have $\phi^{*}(z_{i}) = z_{i} + \mathcal{O}(z^2)$. This implies that the image of $\phi$ in $\mathrm{GL}(N(k))$ is unipotent. Therefore, the subgroup $K_{k,l}$ is unipotent and Equation \ref{Levy} is the Levy decomposition of $G_{k,l}$. When $l = 1$, the group $G_{k,1}$ is solvable. Furthermore, since the exponential map is an algebraic isomorphism for unipotent groups, we have 
\[
K_{k,l} \cong {\rm Lie}(K_{k,l}) = \bigoplus_{i = 1}^{k-1} \mathrm{Sym}^{i+1}( (\mathbb{R}^l)^* ) \otimes \mathbb{R}^l,
\]
mirroring the decomposition of $\at^{k}(E)$ in Lemma \ref{Atiyahalgebroiddecompsition}. The conjugation action of $\mathrm{GL}(\mathbb{R}^l)$ on $K_{k,l}$ corresponds, through this isomorphism, to the symmetric and tensor powers of the defining representation of $\mathrm{GL}(\mathbb{R}^l)$ on $\mathbb{R}^l$. Hence, we obtain the following. 

\begin{lemma} \label{lem:Gkldecomp}
The group $G_{k,l}$ is a linear algebraic group with the following Levy decomposition 
\[
G_{k,l} \cong \big( \bigoplus_{i = 1}^{k-1} \mathrm{Sym}^{i+1}( (\mathbb{R}^l)^* ) \otimes \mathbb{R}^l \big) \rtimes \mathrm{GL}(\mathbb{R}^l),
\]
where $\mathrm{GL}(\mathbb{R}^l)$ acts on $ \mathrm{Sym}^{i+1}( (\mathbb{R}^l)^* ) \otimes \mathbb{R}^l $ via the symmetric and tensor products of the defining representation. The product of elements in the direct sum is defined via the Baker-Campbell-Hausdorff formula applied to the Lie bracket of Lemma \ref{Explicitbracket}. 
\end{lemma} 

\begin{corollary} \label{cor:equivarianthomotopy}
There is a family of $G_{1,l}$-equivariant homomorphisms $h_{t} : G_{k,l} \to G_{k,l}$ interpolating between $h_{1} = id_{G_{k,l}}$ and $h_{0} = j_{k} \circ \pi_{k}$. For $t > 0$, $h_{t}$ is given by 
\begin{equation*}
g \mapsto j_{k}(t^{-1} id)^{-1} \cdot g \cdot j_{k}(t^{-1} id).
\end{equation*}
In particular, $G_{k,l}$ and $G_{1,l}$ are homotopy equivalent. 
\end{corollary}
\begin{proof}
Using the decomposition of Lemma \ref{lem:Gkldecomp}, an element of $G_{k,l}$ may be written as 
\[
g = (v_{1}, ..., v_{k-1}, A), 
\]
where $v_{i} \in \mathrm{Sym}^{i+1}( (\mathbb{R}^l)^* ) \otimes \mathbb{R}^l $ and $A \in \mathrm{GL}(\mathbb{R}^l)$. Therefore 
\[
h_{t}(g) = (t v_{1}, ..., t^{k-1} v_{k-1}, A). 
\]
This map extends to $t = 0$ to give the homomorphism $h_{0} = j_{k} \circ \pi_{k}$. Furthermore, these maps are $G_{1,l}$-equivariant since $t^{-1} id$ is central in $G_{1,l}$. 
\end{proof}

The direct sum decomposition of $K_{k,l}$ given above implies the following result. 
\begin{corollary} \label{canonicalsplitting}
For every $k,l$ there is a set-theoretic splitting $s_k: G_{k, l} \to G_{k+1,l}$ of the projection $G_{k+1, l} \to G_{k,l}$. The splitting $s_{k}$ depends on the linear structure on $\mathbb{R}^l$. 
\end{corollary}

The description of $G_{k,l}$ simplifies considerably in the case $k = 2$ where $K_{k,l}$ becomes abelian, since the Baker-Campbell-Hausdorff formula reduces to addition in the vector space underlying the Lie algebra. 
\begin{corollary}\label{cor:g2lgroup}
There is an isomorphism 
\[
G_{2,l} \cong \big( \mathrm{Sym}^{2}( (\mathbb{R}^l)^* ) \otimes \mathbb{R}^l \big) \rtimes \mathrm{GL}(\mathbb{R}^l),
\]
with product  
\[
(k_{1}, A_{1}) \ast (k_{2}, A_{2}) = (k_{1} + A_{1}(k_{2}), A_{1} A_{2}). 
\]
\end{corollary}
When $l = 1$, the group $K_{3,1}$ is also abelian, giving us the following concrete description. 
\begin{corollary}\label{cor:g31group}
There is an isomorphism $G_{3,1} \cong \mathbb{R}^2 \rtimes \mathbb{R}^*$, with product 
\[
(a_{1}, a_{2}, a_{0}) \ast (b_{1}, b_{2}, b_{0}) = (a_{1} + a_{0}^{-1}b_{1}, a_{2} + a_{0}^{-2}b_{2}, a_{0}b_{0}). 
\]
\end{corollary}

\subsubsection{The principal bundle}
\begin{definition}
Let $E \to W$ be a rank $l$ vector bundle. The \emph{$k$-th frame bundle} $\Fr^{k}(E)$ of $E$ is the principal $G_{k,l}$-bundle whose fibre over $x$ consists of $k$-jets at $0$ of diffeomorphisms $(\mathbb{R}^l,0) \to (E_{x}, 0)$. 
\end{definition}

Associated to the principal bundle $\Fr^{k}(E)$, there is an Atiyah algebroid $\at(\Fr^{k}(E))$ consisting of $G_{k,l}$-invariant vector fields on the total space of $\Fr^{k}(E)$. 
\begin{lemma} \label{kthorderAtiyahisAtiyah}
There is a Lie algebroid isomorphism $\at(\Fr^{k}(E)) \cong \at^{k}(E)$. 
\end{lemma}
\begin{proof}
The sections of $\at(\Fr^{k}(E))$ are $G_{k,l}$-invariant vector fields on $\Fr^{k}(E)$. The flow of such a vector field is a $1$-parameter family $\phi_{t}$ of $G_{k,l}$-equivariant automorphisms of $\Fr^{k}(E)$ covering a family $\gamma_{t}$ of diffeomorphisms of $W$. Equivariance implies that such an automorphism has the form $\phi_{t}(g) = f_{t} \circ g$, for $f_{t} \in J^k(E, \gamma_{t}^*E)$, the $k$-jet of an automorphism, and where $g \in J^{k}((\mathbb{R}^l,0), (E_{x}, 0)) = \Fr^{k}(E)|_{x}$. Differentiating $f_{t}$ gives the $k$-jet of a vector field $V$ on $E$. It is $\pi$-projectable since $f_{t}$ covers a diffeomorphism of $W$, and it is tangent to $W$ since $f_{t}$ preserves the zero section. Hence $V \in \at^{k}(E)$. 
\end{proof}

Lemma \ref{kthorderAtiyahisAtiyah} enables us to integrate $\at^{k}(E)$ to a Lie groupoid: this is the Atiyah groupoid 
\[
\At^{k}(E) = (\Fr^{k}(E) \times \Fr^{k}(E))/G_{k,l}
\]
of equivariant isomorphisms between the fibres of $\Fr^{k}(E)$. The set of morphisms in $\At^{k}(E)$ between points $x$ to $y$ can be identified with the set of $k$-jets of pointed diffeomorphisms from $(E_x,0)$ to $(E_{y},0)$. 

Let $\Ad^k(E) = (\Fr^k(E) \times G_{k,l})/G_{k,l}$ be the adjoint bundle. This is a bundle of groups over $W$ whose group of bisections is the gauge group of $\Fr^k(E)$: 
\[
\mathrm{Gau}(\Fr^k(E)) = \Gamma(W, \Ad^{k}(E)). 
\] 

\begin{lemma} \label{lem:gaugehomotopy}
There is a fibrewise homotopy equivalence between $\Ad^k(E)$ and $\Ad^1(E)$. This induces a homotopy equivalence between $\mathrm{Gau}(\Fr^k(E))$ and $\mathrm{Gau}(\Fr^1(E))$, both equipped with the Whitney $C^{\infty}$ topology. 
\end{lemma}
\begin{proof}
There is a canonical isomorphism $\Fr^{k}(E) \cong (\Fr^1(E) \times G_{k,l})/G_{1,l}$, with the action given by $(p,g) \ast h = (ph, j_{k}(h)^{-1}g)$. As a result, $\Ad^k(E)$ can be identified with the associated bundle of groups $(\Fr^1(E) \times G_{k,l})/G_{1,l}$, where $G_{1,l}$ acts by conjugation on the second factor. The result now follows from Corollary \ref{cor:equivarianthomotopy}. 
\end{proof}

\subsubsection{The holonomy of a foliation} \label{holonomydefsection}
Let $E \to W$ be a vector bundle and let $F$ be a $k$-th order foliation. As we saw in Theorem \ref{splittingprop}, this is equivalent to a Lie algebroid splitting $\sigma : TW \to \at^{k}(E)$. Using the identification $\at^{k}(E) \cong \at(\Fr^{k}(E)) $ of Lemma \ref{kthorderAtiyahisAtiyah}, we now observe that a $k$-th order foliation can be viewed as a flat $G_{k,l}$-connection $\nabla$ on the $k$-th frame bundle $\Fr^{k}(E)$. The parallel transport, or holonomy, of this connection is given by integrating the morphisms $\sigma$ using Lie's second theorem \cite{mackenzie2000integration, moerdijk2002integrability}. This is a Lie groupoid homomorphism 
\[
\mathrm{hol}^{\sigma}: \Pi(W) \to \At^{k}(E),
\]
where $\Pi(W)$ is the fundamental groupoid of $W$. For every path $\gamma: I \to W$, the parallel transport $\mathrm{hol}^{\sigma}(\gamma)$ is the $k$-jet of a pointed diffeomorphism from $(E_{\gamma(0)}, 0)$ to $(E_{\gamma(1)}, 0)$. Restricting to a basepoint $x \in W$, and choosing the $k$-jet of a frame $g \in \Fr^{k}(E|_{x})$, we obtain the monodromy representation of $F$
\[
\mathrm{hol}^{\sigma, g}_{x} : \pi_1(W,x) \to G_{k,l}. 
\]

Note that the holonomy of a $k$-th order foliation can be used to prove the Frobenius theorem (Corollary \ref{Frobenius}) directly. In fact, we may obtain the following more general result:

\begin{proposition}\label{prop:simplyconnected}
Let $F$ be a $k$-th order foliation, for $k \geq 1$, around a simply connected submanifold $W$. Then there exists a germ of a regular foliation $\ff$ such that $F = j_{W}^{k}(\ff)$. 
\end{proposition}
\begin{proof}
Since we are working semi-locally around $W$, we may assume that the ambient manifold is a vector bundle $E \to W$. By Corollary \ref{existenceofflatconnection}, $E$ admits a flat connection. Therefore, since $W$ is simply connected, we may assume that $E = W \times \mathbb{R}^l$. This induces the isomorphism $\At^{k}(E) \cong \mathrm{Pair}(W) \times G_{k,l}$. Furthermore, $\Pi(W) = \mathrm{Pair}(W) = W \times W$. As a result, the holonomy reduces to a homomorphism $\mathrm{hol}^{\sigma} : \mathrm{Pair}(W) \to G_{k,l}$. By viewing elements of $G_{k,l}$ as polynomial maps $\mathbb{R}^l \to \mathbb{R}^l$, we get a map $\phi: \mathrm{Pair}(W) \to \mathrm{Diff}_{0}(\mathbb{R}^l)$ to germs of diffeomorphisms of $\mathbb{R}^l$. Note that this map is not necessarily a homomorphism. 

Now fix a point $x \in W$ and consider the map 
\[
\varphi : W \times \mathbb{R}^l \to W \times \mathbb{R}^l, \qquad (y,v) \mapsto (y, \phi(y,x)(v)),
\]
which is the germ of a diffeomorphism around $W \times \{ 0 \}$. Let $\ff_0$ be the foliation on $W \times \mathbb{R}^l$ whose leaves are given by $W \times \{v \}$, for $v \in \mathbb{R}^l$, and let $\ff = \varphi_{*}(\ff_{0})$ be its pushforward under $\varphi$. Then $j_{W}^k(\ff) = F$.
\end{proof}

To deduce Corollary \ref{Frobenius}, restrict a $k$-th order foliation $F$ to a contractible open ball. The above Proposition \ref{prop:simplyconnected} shows that $F$ can be represented by an actual foliation on this ball. The Corollary now follows immediately from the usual Frobenius theorem for foliations. 

\begin{remark}
Let $\Fr^{k}(E)$ be the $k$-th frame bundle of $E$ and let $\Fr^{\infty}(E)$ be the bundle of formal pointed diffeomorphisms from $(\mathbb{R}^l, 0)$ to the fibres $(E_{x}, 0)$. This is a principal $G_{\infty,l}$-bundle, where $G_{\infty,l}$ is the group of formal pointed diffeomorphisms of $(\mathbb{R}^l, 0)$. Let $H = \ker(G_{\infty,l} \to G_{k,l})$. Recall that a $k$-th order foliation is equivalent to a flat connection $\nabla$ on $\Fr^{k}(E)$. Consider a lift $\hat{\nabla}$ of $\nabla$ to a possibly non-flat connection on $\Fr^{\infty}(E)$. Choose a point $p \in \Fr^{k}(E)$ and consider the smallest subset $L \subset \Fr^{k}(E)$ which is preserved by the connection. This determines a subgroup $K \subset G_{k,l}$ for which $L$ is a reduction of structure. The lift $\hat{\nabla}$ therefore has a reduction of structure to a subgroup $\hat{H} \subset G_{\infty,l}$ containing $H$ and such that $\hat{H}/H = K$. In \cite{fischer2024classification}, Fischer and Laurent-Gengoux obtain a classification of formal $k$-th order foliations in terms of such $\hat{H}$-bundles with connections. It is possible to lift $\nabla$ to a non-flat $G_{r,l}$-connection for any $r \geq k$ to get a `Yang-Mills connection' in the sense of \cite{fischer2024classification}. When $r = k+1$ the curvature of this connection is precisely the extension class. 
\end{remark}

The description of the holonomy of a $k$-th order foliation ends the first half of our paper. For the reader's convenience we  collect the equivalent descriptions of a $k$-th order foliation into a single theorem:
\begin{theorem}\label{th:equivalence}
Let $M$ be an $n$-dimensional manifold and let $W \subset M$ be an embedded submanifold of codimension $l$. There is a one-to-one correspondence between:
\begin{enumerate}
\item Singular foliations $\A$ of $b^{k+1}$-type for $(W,M)$.
\item $k$-th order foliations $F \subset \Der(\mathcal{N}^k_W)$ around $W$.
\item Subsheaves of invariant functions $\C \subset \Ncal^k_W$.
\end{enumerate}
After the choice of a tubular neighbourhood embedding $E \to W$, these are further in correspondence with:
\begin{itemize}
\item[(4)] Lie algebroid splittings $\sigma : TW \rightarrow \at^k(E)$.
\item[(5)] Flat principal $G_{k,l}$-connections on $\Fr^k(E)$.
\item[(6)] Tuples $(\nabla, \eta_{1}, ..., \eta_{k-1})$, where $\nabla$ is a flat linear connection on $E$ and $\eta_i \in \Omega^1_W(\Sym^{i+1}(E^*)\otimes E)$, such that the following Maurer-Cartan equation 
\begin{equation*}
d^{\nabla} \eta_{r} + \frac{1}{2} \sum_{i + j = r} [\eta_{i}, \eta_{j}] = 0,
\end{equation*}
is satisfied for all $ r = 1, ..., k-1$.
\end{itemize}
\end{theorem}

\section{Examples}\label{sec:Examples1}
In this section we will give some explicit examples of $b^{k+1}$-algebroids. In particular, we will compute the extension class and show that it is non-trivial in certain situations.

\subsection{Surfaces} \label{ex:genusgsurfaceexampleextension}

Let $S$ be a genus $g$ Riemann surface. Its de Rham cohomology ring is given by 
\[ H^{\bullet}(S) \cong \dfrac{\RR[\alpha_{i}, \beta_{i}, \omega ; \ i = 1, ..., g]}{(\alpha_{i} \alpha_{j} = \beta_{i} \beta_{j} = 0, \alpha_{i} \beta_{j} = \delta_{ij} \omega)}, \]
where the degrees are $|\alpha_{i}| = |\beta_{i}| = 1$ and $|\omega| = 2$. Choose representatives for the degree $1$ classes. Abusing notation, we denote these closed $1$-forms by $\alpha_{i}$ and $\beta_{i}$.

Consider the trivial line bundle $L = S \times \RR$ equipped with the trivial flat connection and write $t$ for the coordinate on $\RR$. According to Corollary \ref{MCeq}, two closed $1$-forms 
\[ \eta_{1} = \sum_{i = 1}^{g} (x_{i} \alpha_{i} + y_{i} \beta_{i}), \qquad \eta_{2} = \sum_{i = 1}^{g} (w_{i} \alpha_{i} + z_{i} \beta_{i}), \]
where $x_{i}, y_{i}, w_{i}, z_{i} \in \mathbb{R}$ are constants, specify a $b^{4}$-algebroid $A$ on the total space of $L$. Its sheaf of sections is generated by the vector fields $t^{4} \partial_{t}$ and $X + \eta_{1}(X) t^2 \partial_{t} + \eta_{2}(X) t^3 \partial_{t}$, for vector fields $X$ on $S$.

We have thus constructed a space of $b^4$-algebroids parametrised by $\RR^{4g}$ with coordinates $(x_{i}, y_{i}, w_{i}, z_{i})$. Each such algebroid can be extended to a $b^{5}$-algebroid if and only if the extension class $e$ from Definition \ref{def:extension} vanishes. By Corollary \ref{cor:explicitextensionclass} this class is given by
\[ e(x_{i}, y_{i}, w_{i}, z_{i}) = \eta_{1} \wedge \eta_{2} = \sum_{i,j} (x_{i} \alpha_{i} + y_{i} \beta_{i})  (w_{j} \alpha_{j} + z_{j} \beta_{j}) = \sum_{i = 1}^{g} (x_{i} z_{i} - y_{i} w_{i}) \omega. \]
Hence, solutions to the quadric $ \sum_{i = 1}^{g} (x_{i} z_{i} - y_{i} w_{i}) = 0$ correspond to $b^4$-algebroids that admit an extension to a $b^5$-algebroid.

\subsection{The Heisenberg group}\label{ex:Heisenberg1}

Let $H(\RR)$ be the Heisenberg group, the subgroup of real-valued $3 \times 3$ upper-triangular matrices with $1$'s on the diagonal. Let $H(\ZZ)$ be the subgroup with integer entries. Let $X = H(\RR)/H(\ZZ)$, which is a compact $3$-dimensional manifold. Writing an upper triangular matrix as 
\[ \begin{pmatrix} 1 & x & z \\ 0 & 1 & y \\ 0 & 0 & 1 \end{pmatrix} \]
we get local coordinates $(x,y,z)$ on $X$. The $1$-forms $a = dx$, $b = dy$ and $c=dz-ydx$ are well-defined on $X$ and satisfy $a \wedge b = dc$. They generate a cdga $(\Omega, d)$ which is a differential subalgebra of the de Rham complex of $X$. By a Gysin sequence argument, using the fact that $X$ is an $S^1$ bundle over $S^1 \times S^1$, we can see that $(\Omega, d)$ is in fact quasi-isomorphic to the de Rham complex. 

Let $L = X \times \RR$ be the trivial line bundle and equip it with the trivial flat connection. Let $t$ be the coordinate on $\mathbb{R}$. By Corollary \ref{cor:explicitextensionclass}, a $b^{5}$-algebroid $A$ on $L$ can be specified by the data of three $1$-forms which satisfy the Maurer-Cartan equation. Such a solution is provided by 
\[ \eta_{1} = a, \qquad \eta_{2} = b, \qquad \eta_{3} = -c. \]
The extension class is given by $-2 a \wedge c$, which is non-trivial in cohomology. Therefore $A$ cannot be extended to a $b^{6}$-algebroid. $A$ has a global basis given in coordinates by 
\[ X_{1} = \partial_{x} + y \partial_{z} + t^2 \partial_{t}, \ \ X_{2} = \partial_{y} + t^3 \partial_{t}, \ \ X_{3} = \partial_{z} - t^4 \partial_{t}, \ \ X_{4} = t^{5} \partial_{t}. \]

\subsection{A mapping torus}

Consider the matrix 
\[ M = \begin{pmatrix} 2 & 1 \\ 1 & 1 \end{pmatrix}. \]
Seen as a diffeomorphism of the $2$-torus, this is the famous Arnold cat map. Its eigenvalues are $r^{\pm 1} = 1 \pm \varphi^{\pm 1}$, where $\varphi = \frac{1 + \sqrt{5}}{2}$ is the golden ratio. We then consider the matrix $S = M^{-1} \oplus M^{-2} \in SL(\ZZ^4)$ as a diffeomorphism $\phi : \TT^4 \to \TT^4$ of the $4$-torus. Its action at the level of first cohomology $\phi_{*} : H^{1}(\TT^4) \to H^{1}(\TT^4) $ induces the decomposition into eigenspaces 
\[ H^{1}(\mathbb{T}^4) = E_{r} \oplus E_{r^{-1}} \oplus E_{r^2} \oplus E_{r^{-2}}, \]
where $r^{\pm 1}, r^{\pm 2}$ are the eigenvalues. We let $a_{1}, b_{1}, a_{2}, b_{2}$ be ``constant'' closed $1$-forms on the torus giving rise to a basis respecting the above eigenspace decomposition. 

Suppose $Q = (\TT^4 \times \RR)/\ZZ$ is now the mapping torus for $\phi$; we write $f: Q \to S^1$ for the corresponding fibration over the circle. We write $\alpha = f^*(d\theta)$, where $\theta$ is the coordinate on $S^1$. We now fix the non-zero constant $\lambda = \log(r)$ and consider the linear flat connection $\nabla = d - \lambda \alpha$ on the trivial line bundle over $Q$. The following forms 
\[ \eta_{1} = e^{-\lambda \theta} a_{1}, \qquad \eta_{2} = e^{-2 \lambda \theta} a_{2}, \]
are then closed for $\nabla^{-1}$ and $\nabla^{-2}$. Consequently, by Proposition \ref{MCeq}, we obtain a $b^4$-algebroid on $L = Q \times \mathbb{R}$. Using Corollary \ref{cor:explicitextensionclass}, the obstruction to extending to a $b^{5}$-algebroid is the class of
\[ e = \eta_{1} \wedge \eta_{2} = e^{-3 \lambda \theta} a_{1} \wedge a_{2}. \]

\subsection{Codimension $2$}
According to Proposition \ref{MCeq}, a $2$-nd order foliation $F$ around $W$ on the total space of a vector bundle $E$ is given by the data of a flat connection $\nabla$ on $E$ and a $d^{\nabla}$-closed $1$-form $\eta_{1} \in \Omega^1_{W}(\mathrm{Sym}^2(E^*) \otimes E)$. Furthermore, by Corollary \ref{cor:explicitextensionclass}, the extension class is given by $e(F) = \frac{1}{2}[\eta_1, \eta_1] \in H^2(W, \mathrm{Sym}^3(E^*) \otimes E)$. When the rank of $E$ is $1$, the bundle $\mathrm{Sym}^2(E^*) \otimes E$ has rank $1$ and consequently $e(F) = 0$. As a result, when $W$ has codimension $1$, any $b^3$-algebroid may be extended to a $b^4$-algebroid. This is no longer true when $W$ has codimension $2$ as we presently show. 

Let $E = W \times \mathbb{R}^2$, let $(x,y)$ be linear coordinates on $\mathbb{R}^2$, and equip $E$ with a trivial flat connection. The vector space $\mathrm{Sym}^2((\mathbb{R}^2)^*) \otimes \mathbb{R}^2$ may be identified with the $6$-dimensional space of vector fields on $\mathbb{R}^2$ whose coefficients are degree $2$ homogeneous polynomials. Hence, a $2$-nd order foliation is specified by the data of $6$ closed differential forms: 
\[
\eta_{1} = a x^2 \partial_{x} + b xy \partial_{x} + c y^2 \partial_{x} + \alpha x^2 \partial_y + \beta xy \partial_y + \gamma y^2 \partial_y, 
\]
where $a, b, c, \alpha, \beta, \gamma \in \Omega_{W}^{1, cl}$. For simplicity we set $a = b = c = 0$. Then the extension class is 
\[
e(F) = (\alpha \wedge \beta \ x^3 + 2 \alpha \wedge \gamma \ x^2 y + \beta \wedge \gamma \ xy^2 ) \partial_y. 
\]
Hence, all $3$ products $\alpha \wedge \beta, \alpha \wedge \gamma, \beta \wedge \gamma$ must vanish in cohomology in order for $F$ to admit an extension to a $3$-rd order foliation.


\section{The Riemann-Hilbert correspondence}\label{sec:Riemann--Hilbert}
The Riemann-Hilbert correspondence is the name given to various results relating flat connections on a manifold to representations of its fundamental group. For principal bundles this is classic and the statements we need are collected in Appendix \ref{sec:principal}. 

Our goal in this section is to give a detailed study of the groupoid of $k$-th order foliations and to establish a Riemann-Hilbert type correspondence. To do this we will rely on the fact, established in Section \ref{holonomydefsection}, that $k$-th order foliations on the total space of a rank $l$ vector bundle correspond to flat principal $G_{k,l}$-connections. The main issue to take care of is that in order to use this correspondence a tubular neighbourhood embedding must be chosen. Hence, we must take into account the dependence on this choice. 

In Section \ref{CatRHstatement} we will relate the category of $k$-th order foliations of codimension $l$ around a manifold $W$ to the category of $G_{k,l}$-representations of the fundamental group $\pi_{1}(W,x)$. This gives rise to a classification of these foliations up to isomorphism. In Section \ref{TopRHstatement} we will introduce a topology on the groupoid of $k$-th order foliations and upgrade our Riemann-Hilbert functor to a fibration of topological groupoids. Finally, in Section \ref{Isotopy} we will give a classification of $k$-th order foliations around $W$ in a given manifold $M$ up to isotopy.

\subsection{Riemann-Hilbert functor for $k$-th order foliations} \label{CatRHstatement}

Consider a connected manifold $W$. For notational convenience we write $j^k$ for $j^k_W$, the process of taking $k$-jets along $W$. We will assume throughout that $k \geq 1$. 

\begin{definition}
 
The category $J^{k}\FolCat(W,l)$ of \emph{$k$-th order foliations of codimension $l$} is defined as follows. Its class of objects $J^{k}\FolObj(W,l)$ consists of pairs $(M, F)$, where:
\begin{itemize}
\item $M \supset W$ is a germ of an ambient manifold such that $W$ has codimension $l$, 
\item $F$ is a $k$-th order foliation in $M$, defined along $W$.
\end{itemize}
Morphisms $j^kf : (M_1,F_1) \rightarrow (M_2,F_2)$ are $k$-jets of diffeomorphisms that fix $W$ pointwise and relate the $k$-th order foliations by pushforward. 
\end{definition}

\begin{remark} \label{rem:achoice}
Alternatively, we could have defined the morphisms in $J^{k}\FolCat(W,l)$ to be \emph{germs} of diffeomorphisms that fix $W$ pointwise. However, this does not change the set of isomorphism classes of $k$-th order foliations. Indeed, according to Lemma \ref{lem:factoringViakJets}, two diffeomorphisms with the same $k$-jet define the same action on $k$-th order foliations. Later, we will see that working with $k$-jets is more convenient when handling topologies.
\end{remark}

Now fix a basepoint $x \in W$ and let $\pi_{1}(W,x)$ be the fundamental group of $W$ based at $x$. 
\begin{definition}
The \emph{category of representations} $\Rep(\pi_{1}(W, x), G_{k,l})$ is the action groupoid 
\[ G_{k,l} \ltimes \Hom(\pi_{1}(W, x), G_{k,l}). \]
Objects are thus homomorphisms $\rho : \pi_{1}(W,x) \to G_{k,l}$. Morphisms are elements of $G_{k,l}$, which act on homomorphisms by conjugation.
\end{definition}

In this section we establish the following result.
\begin{theorem} \label{RHthm}
There is a full and essentially surjective Riemann-Hilbert functor 
\[ RH: J^{k}\FolCat(W,l) \to \Rep(\pi_{1}(W, x), G_{k,l}). \]
In particular, $RH$ induces a bijection between isomorphism classes of objects. 
\end{theorem}

\begin{remark}
The reader may want to check \cite{CC00}, where a version of Theorem \ref{RHthm} for germs of foliations is established. 
\end{remark}

\begin{remark}
Since the categories involved in Theorem \ref{RHthm} are groupoids, the claim that $RH$ induces a bijection between the sets of isomorphism classes of objects is an immediate consequence of the functor being full and essentially surjective. 
\end{remark}

Consider $(M,F) \in J^{k}\FolObj(W,l)$ and let $\Tan^{k}(W,F)$ be the (germ of) singular foliation of $b^{k+1}$-type associated to it. Write $RH(M,F) \in \Hom(\pi_{1}(W,x), G_{k,l})$ for the associated representation. The Riemann-Hilbert functor induces a surjective homomorphism between automorphism groups 
\[ RH: \Aut(M,F) \to \Aut(RH(M, F)). \]
The following theorem characterizes the kernel of this homomorphism as the subgroup of `inner automorphisms' of $(M,F)$. 
\begin{theorem} \label{innercharacterization}
Let $(M, F)$ be a $k$-th order foliation and let $\Tan^{k}(W, F)$ be the associated singular foliation of $b^{k+1}$-type. A $k$-jet of diffeomorphism $j^k\phi \in \Aut(M,F)$ is in the kernel of $RH$ if and only if it can be represented by the time-$1$ flow of a time-dependent vector field $V_{t} \in \Tan^{k}(W, F)$ that vanishes along $W$.
\end{theorem}

\subsubsection{Preliminaries}
In producing the representation $RH(M,F)$ associated to a $k$-th order foliation $(M,F)$, we will need to choose some extra data, such as a tubular neighbourhood embedding and a frame. Therefore, in order to establish Theorem \ref{RHthm}, we first replace the categories under study with equivalent categories whose objects account for the required extra data.

\begin{definition}\label{definitionofJkC}
Let $J^k\mathcal{C}$ be the category whose objects are triples $(M, F, j^k\Phi)$ with $(M,F)\in J^k\FolObj(W,l)$ and $j^k\Phi: \nu_{W} \to M$ a $k$-jet of tubular neighbourhood embedding. Morphisms
\begin{equation*}
j^kf : (M_1,F_1,j^k\Phi_1) \rightarrow (M_2,F_2,j^k\Phi_2) 
\end{equation*}
are $k$-jets of ambient diffeomorphisms $f$ that fix $W$ pointwise and relate the $k$-th order foliations by pushforward.
\end{definition}
Note that morphisms are not required to preserve tubular neighbourhood embeddings. Hence, since any submanifold admits a tubular neighbourhood, we have:
\begin{lemma}
The forgetful functor $J^k\mathcal{C} \to J^{k}\FolCat(W,l)$ is an equivalence of categories.
\end{lemma}

We will need a further auxiliary category. 
\begin{definition} \label{definitionofC}
Let $\mathcal{C}$ be the category whose objects are triples $(M, F, \Phi)$ with $(M,F)\in J^k\FolObj(W,l)$ and where $\Phi: \nu_{W} \to M$ is the germ of a tubular neighbourhood embedding. Morphisms are germs of ambient diffeomorphisms that fix $W$ pointwise and relate the $k$-th order foliations by pushforward. 
\end{definition}
The reason for introducing $\mathcal{C}$ is that our arguments will rely on the use of actual diffeomorphisms and tubular neighbourhood embeddings, even though ultimately only their $k$-jets are relevant. According to Lemma \ref{lem:polynomialRepresentatives}, every jet of diffeomorphism can be realised by a germ. Hence, we deduce:
\begin{lemma}
The functor $J^k : \C \rightarrow J^k\C$ obtained by taking $k$-jets is full and surjective.
\end{lemma}

As in Section \ref{holonomymap}, we let $\FlatCat(W,G_{k,l})$ denote the category of flat principal $G_{k,l}$-bundles over $W$. The classical Riemann-Hilbert correspondence (Theorem \ref{MainRHtheorem}) is an equivalence of categories between $\FlatCat(W,G_{k,l})$ and $\Rep(\pi_{1}(W, x), G_{k,l})$. Therefore, in order to prove Theorem \ref{RHthm} it suffices to prove the following result. 

\begin{proposition}  \label{FunctorF}
There is a full and essentially surjective functor $\F: \mathcal{C} \to \FlatCat(W,G_{k,l})$ which factors as follows 
\[
\F: \mathcal{C} \xrightarrow{J^k} J^k\C \xrightarrow{J^k\F} \FlatCat(W,G_{k,l}).
\]
Therefore we obtain a full and essentially surjective functor $J^k\F: J^k\mathcal{C} \to \FlatCat(W,G_{k,l})$.
\end{proposition}

\begin{remark} \label{explicitRHconstruction}
As we have explained, providing an explicit construction of the functor $RH$ from Theorem \ref{RHthm} requires spelling out various categorical equivalences. One of them is the equivalence $\FlatCat(W,G_{k,l}) \simeq \Rep(\pi_1(W,x),G_{k,l})$ from Theorem \ref{MainRHtheorem}, which involves the choice of a basepoint $x$ and a framing $\tilde{\phi}: (\RR^l, 0) \to (\nu_{W}|_{x}, 0)$. If we wanted to spell this out, we could replace $J^k\mathcal{C}$ with an equivalent category $J^k\mathcal{C}'$ whose objects also include the data of the framing. \end{remark}

\subsubsection{Step 1: Constructing the functor} \label{tediousconstruction}

We now define the functor $\F: \mathcal{C} \to \FlatCat(W, G_{k,l})$ of Proposition \ref{FunctorF} and show that it admits the desired factorization. Consider an object $(M, F, \Phi) \in \mathcal{C}$. The map $\Phi : \nu_{W} \to M$ is a germ of tubular neighbourhood embedding and we use it to pullback the foliation $F$ to the normal bundle. By Theorem \ref{splittingprop}, this allows us to view $F$ as a flat connection $\sigma$ on $\Fr^{k}(\nu_{W})$ with holonomy $\hol^{\sigma} : \Pi(W) \to \At^{k}(\nu_{W})$. Hence we make the following prescription. 
\begin{definition}
At the level of objects, the functor $\F$ is given by:
\[ \F(M, F, \Phi) := (\Fr^{k}(\nu_{W}), \sigma). \hfill \qedhere \]
\end{definition}
As already noted in Theorem \ref{splittingprop}, the pull-back $\Phi^*F$ only depends on $j^k\Phi$ (by Lemma \ref{lem:factoringViakJets}) and from this it follows that $\F(M, F, \Phi)$ only depends on $j^k\Phi$. Hence at the level of objects $\F$ does factor through $J^k\mathcal{C}$.

\begin{remark}
Continuing Remark \ref{explicitRHconstruction}, assume that we have an object $(M, F, j^k \Phi, g) \in J^k\mathcal{C}'$, where $(M, F, j^{k} \Phi) \in J^k \mathcal{C}$ and $g = j^{k}(\tilde{\phi}) \in \Fr^{k}(\nu_{W})$ is the $k$-jet of an embedding $\tilde{\phi}: (\RR^l, 0) \to (\nu_{W}|_{x}, 0)$. We can now apply the classical Riemann-Hilbert functor (Theorem \ref{MainRHtheorem}) to the flat connection $ \F(M, F, \Phi) = (\Fr^{k}(\nu_{W}), \sigma)$ in order to produce a $G_{k,l}$-representation. We set
\[ RH(M, F, j^k\Phi, g) := \hol_{x}^{\sigma, g} : \pi_{1}(W,x) \to G_{k,l}, \]
as defined in Section \ref{holonomydefsection}.
\end{remark}

In order to define the functor $\F$ on the level of morphisms, we make use of the fact that $\mathcal{C}$ is generated by two wide sub-categories. One of these consists of those morphisms which are compatible with the tubular neighbourhood embeddings, the other consists of morphisms which purely involve changing the tubular neighbourhoods. 
\begin{definition}
The category $\mathcal{C}$ is equipped with two wide subcategories $\mathcal{C}_{tn}, \mathcal{C}_{id} \subset \mathcal{C}$. 
\begin{itemize}
\item The subcategory $\mathcal{C}_{tn} \subset \mathcal{C}$ consists of morphisms
\[ f: (M_{1}, F_{1}, \Phi_{1}) \to (M_{2}, F_{2}, \Phi_{2}) \]
such that $f \circ \Phi_{1} = \Phi_{2} \circ \nu(f)$, where $\nu(f) : \nu_{W, M_{1}} \to \nu_{W, M_{2}}$ is the induced morphism between normal bundles.
\item The subcategory $\mathcal{C}_{id} \subset \mathcal{C}$ consists of those morphisms whose underlying diffeomorphism germ is the identity. \qedhere
\end{itemize}
\end{definition}
\begin{remark}
Given a morphism $f \in \mathcal{C}_{tn}$, the tubular neighbourhood embeddings provide a linearization 
\[
 f = \Phi_2 \circ \nu(f) \circ \Phi_{1}^{-1}. 
\]
Therefore, $f$ is completely determined by its $1$-jet. 
\end{remark}

Given any morphism in $\mathcal{C}$, we can either pushforward the tubular neighbourhood from the domain, or pullback the tubular neighbourhood from the codomain. This allows us to establish the following factorization result.  
\begin{lemma} \label{factorization}
Any morphism $f$ in $\mathcal{C}$ has unique factorizations $f = g_{1} \circ h_{1} = h_{2} \circ g_{2}$, where $g_{1}, g_{2} \in \mathcal{C}_{tn}$ and $h_{1}, h_{2} \in \mathcal{C}_{id}$. 
\end{lemma}
We will first define $\F$ on the two subcategories $\mathcal{C}_{tn}$ and $\mathcal{C}_{id}$. In Definition \ref{functorFdef} below, we will use Lemma \ref{factorization} to extend the definition of $\F$ to all of $\mathcal{C}$.

Given a morphism $f: (M_{1}, F_{1}, \Phi_{1}) \to (M_{2}, F_{2}, \Phi_{2})$ in $\mathcal{C}_{tn}$, the vector bundle morphism $\nu(f) : \nu_{W, M_{1}} \to \nu_{W, M_{2}}$ preserves the induced $k$-th order foliations. Taking the $k$-jet, we thus get an induced map $j^{k}(\nu(f)) : \Fr^{k}( \nu_{W, M_{1}}) \to \Fr^{k}( \nu_{W, M_{2}})$ which sends the flat connection $\sigma_{1}$ to $\sigma_{2}$. Therefore, we give the following definition. 
\begin{definition} \label{def:ctn}
The functor $\F : \mathcal{C}_{tn} \to \FlatCat(W,G_{k,l})$ is defined by sending a morphism $f: (M_{1}, F_{1}, \Phi_{1}) \to (M_{2}, F_{2}, \Phi_{2})$ in $\mathcal{C}_{tn}$ to
\begin{equation*}
\F(f) := j^k(\nu(f)). \hfill \qedhere
\end{equation*}
\end{definition}
This definition is readily seen to be functorial in $f$.

The definition of $\F : \mathcal{C}_{id} \to \FlatCat(W,G_{k,l})$ is more involved. Consider $(M,F) \in J^k\FolObj(W,l)$, and define $\mathcal{C}_{id,(M,F)}$ to be the full subcategory of $\mathcal{C}_{id}$ consisting of objects which have underlying $k$-th order foliation $(M,F)$. A morphism in $\mathcal{C}_{id,(M,F)}$ is therefore a pair of germs of tubular neighbourhood embeddings. 

Fix a manifold representing the germ $M$, which we still denote by $M$. This allows us to consider the space of tubular neighbourhood embeddings $\TNE(W,M)$, which is contractible (see Subsection \ref{ssec:TNE}). Write $F \times [0,1]$ for the $k$-th order foliation on $M \times [0,1]$ obtained as the pullback of $F$ under the projection $M \times [0,1] \rightarrow M$. 
\begin{definition} \label{def:definitionOfF2}
Fix $(M,F)$ and consider a pair of tubular neighbourhood embeddings $\Phi_1$ and $\Phi_2$ representing a morphism in $\mathcal{C}_{id,(M,F)}$. Connect them by a path $\Phi'$, which is viewed as a tubular neighbourhood embedding $\Phi' : \nu_{W} \times [0,1] \to M \times [0,1]$ that restricts to $\Phi_1$ and $\Phi_2$ at the two endpoints. 

Pullback $F \times [0,1]$ to $\nu_{W} \times [0,1]$ using $\Phi'$. This defines a flat connection $\sigma'$ on $\Fr^{k}(\nu_{W}) \times [0,1]$ that restricts to the flat connections $\F(M,F,\Phi_{1})$ and $\F(M,F,\Phi_{2})$ at the two endpoints. Then, 
\[ \F(\Phi_1,\Phi_2): \F(M,F,\Phi_{1}) \rightarrow \F(M,F,\Phi_{2})  \]
is defined to be the parallel transport
\[ \hol^{\sigma'}_{[0,1]}: \Fr^{k}(\nu_{W}) \rightarrow \Fr^{k}(\nu_{W}) \]
of $\sigma'$ along $[0,1]$. It is a morphism of flat principal bundles.
\end{definition}

Note that since $\F(\Phi_1,\Phi_2)$ is defined as the parallel transport of a connection, it is automatically functorial. 
However, since the above definition of $\F(\Phi_1,\Phi_2)$ makes use of several auxiliary choices, we must show in the following two lemmas that it is well-defined. 
\begin{lemma}\label{lem:stabilityofF}
Let $\Phi_{t} \in \TNE(W,M)$ be a family of tubular neighbourhood embeddings, for $t \in [0,1]$, with the property that $\tilde{\Phi} : \nu_{W} \times [0,1] \to M \times [0,1]$, defined by $\tilde{\Phi}(v,t) = (\Phi_{t}(v), t)$, is a tubular neighbourhood embedding. Let $(M,F) \in J^{k}\FolObj(W,l)$ be a $k$-th order foliation and let $\sigma_{t}$ be the corresponding family of flat connections on $\Fr^{k}(\nu_{W})$ induced by $\Phi_t$. Let $\gamma : [0,1] \to W$ be a path such that $\Phi_t|_{\gamma(0)}$ and $\Phi_{t}|_{\gamma(1)}$ are independent of $t$. Then the holonomy $\hol^{\sigma_t}_{\gamma}$ along $\gamma$ is independent of $t$. 
\end{lemma}
\begin{proof}
Without loss of generality, we will prove that $\hol^{\sigma_0}_{\gamma} = \hol^{\sigma_1}_{\gamma}$. Consider the following paths in $W \times [0,1]$. For $s \in [0,1]$, define $\gamma_{s}(t) = (\gamma(t), s)$, and for $w \in W$, define $p_{w}(t) = (w, t)$. Then the concatenated paths $p_{\gamma(1)} \ast \gamma_0$ and $\gamma_1 \ast p_{\gamma(0)}$ are homotopic rel endpoints. Let $\sigma'$ be the connection on $\Fr^{k}(\nu_{W}) \times [0,1]$ induced by pulling back $F \times [0,1]$ along $\tilde{\Phi}$. The parallel transport satisfies 
\[
\hol^{\sigma'}_{p_{\gamma(1)} } \circ \hol^{\sigma'}_{\gamma_0 } = \hol^{\sigma'}_{ \gamma_1} \circ  \hol^{\sigma'}_{ p_{\gamma(0)}}.  
\]
Now observe that the restriction of $\sigma'$ to $W \times \{ s \}$ is the connection $\sigma_{s}$ and hence $\hol^{\sigma'}_{\gamma_s} = \hol^{\sigma_s}_{\gamma} $. Furthermore, because $\Phi_t|_{\gamma(0)}$ is independent of $t$, the restriction of $\sigma'$ to the submanifold $\{ \gamma(0) \} \times [0,1]$ is the trivial connection on $\Fr^{k}(\nu_{W}|_{\gamma(0)}) \times [0,1]$, and hence $\hol^{\sigma'}_{ p_{\gamma(0)}} = id$. The analogous statement holds at $\gamma(1)$. This concludes the argument. 
\end{proof}

\begin{lemma} \label{lem:independencePaths}
The morphism $\F(\Phi_1,\Phi_2)$ does not depend on the choice of path $\Phi'$ connecting $\Phi_{1}$ and $\Phi_{2}$. Moreover, it only depends on the $k$-jets $j^k\Phi_1$ and $j^k\Phi_2$.
\end{lemma}
\begin{proof}
Let $\Phi'(v,t) = (\Phi'_t(v), t)$ be a path of tubular neighbourhood embeddings connecting $\Phi_1$ and $\Phi_2$. First, we claim that $\F(\Phi_1,\Phi_2)$ is independent of reparametrizations of $\Phi'$. To this end, let $f : [0,1] \to [0,1]$ be a smooth map preserving the endpoints and consider $\Phi''(v,t) = (\Phi'_{f(t)}(v), t)$, a new path of tubular neighbourhood embeddings. Let $\sigma'$ and $\sigma''$ be the connections on $\Fr^{k}(\nu_{W}) \times [0,1]$ obtained using $\Phi'$ and $\Phi''$, respectively. We will apply Lemma \ref{lem:stabilityofF} to show that $ \hol^{\sigma'}_{[0,1]} =  \hol^{\sigma''}_{[0,1]}$. Indeed, let $h_{s}(t) = (1-s)t + sf(t)$ be the linear interpolation, and let 
\[
\Psi_s : \nu_{W} \times [0,1] \to M \times [0,1], \qquad (v, t) \mapsto (\Phi'_{h_s(t)}(v), t),
\]
viewed as a path of tubular neighbourhoods connecting $\Phi'$ to $\Phi''$. If we restrict $\Psi_{s}$ to $W \times \{ 0 \}$ we get $\Phi_1$, and if we restrict it to $W \times \{ 1 \}$ we get $\Phi_2$. In both cases, the restriction of $\Psi_s$ is independent of $s$. Therefore, since the endpoints of the paths used to compute $\F(\Phi_1,\Phi_2)$ lie on $W \times \{0 \}$ and $W \times \{ 1\}$, we are in the setting of Lemma \ref{lem:stabilityofF}. This establishes the claim. 

Using reparametrizations, we can assume that our paths of tubular neighourhood embeddings $\Phi'_t(v)$ are constant in $t$ for $t$ close to $0$ and $1$. This allows us to smoothly concatenate them. Therefore, given two paths $\Phi'$ and $\Phi''$ connecting $\Phi_{1}$ and $\Phi_2$, it is possible to glue them together to obtain a tubular neighbourhood embedding $\Psi_1 : \nu_{W} \times S^1 \to M \times S^1$. We can now show that $\F(\Phi_1, \Phi_2)$ is path independent by showing that the holonomy of the induced connection on $\Fr^{k}(\nu_{W}) \times S^1$ around the loops $\{ x \} \times S^1$ is trivial. This will in turn follow by showing that $\Psi_1$ can be extended to a TNE $\tilde{\Psi} : \nu_{W} \times D^2 \to M \times D^2$, where $D^2$ is the unit disc. 

We do this in two steps. First, consider the TNE $\Psi_2 = \Phi_1 \times id_{S^1}: \nu_{W} \times S^1 \to M \times S^1$. Applying Lemma \ref{lem:betterCerf}, there is a path of TNE $\Psi'$ connecting $\Psi_1$ to $\Psi_2$. Second, the TNE $\Psi_2$ can be extended to the TNE $\Psi''  =  \Phi_1 \times id_{D^2} : \nu_{W} \times D^2 \to M \times D^2$. Gluing together $\Psi'$ and $\Psi''$ yields the desired extension $\tilde{\Psi}$. As a result, we conclude that $\F(\Phi_1,\Phi_2)$ is path independent. 

To show that $\F(\Phi_1,\Phi_2)$ only depends on the $k$-jets, it suffices to show that it is the identity if $j^k\Phi_{1} = j^k\Phi_{2}$. To show this, let $\Phi'$ be a path connecting $\Phi_1$ and $\Phi_2$ through tubular neighbourhood embeddings with a fixed $k$-jet (Lemma  \ref{lem:dilating}). By Theorem \ref{splittingprop}, the induced connection $\sigma'$ on $\Fr^{k}(\nu_{W}) \times [0,1]$, and hence its parallel transport $\hol^{\sigma'}_{[0,1]}$ along $[0,1]$, only depends on $j^k\Phi'$. However, $j^k\Phi'$ can be represented by the constant tubular neighbourhood embedding $\Phi_{1} \times [0,1]$, which has trivial parallel transport, as desired.
\end{proof}

Now that we have defined $\F$ on the two subcategories $\mathcal{C}_{tn}$ and $\mathcal{C}_{id}$, we can extend its definition to $\C$ using Lemma \ref{factorization}.

\begin{definition} \label{functorFdef}
The functor $\F: \mathcal{C} \to \FlatCat(W,G_{k,l})$ is defined by sending an object $(M, F, \Phi) \in \mathcal{C}$ to $(\Fr^{k}(\nu_{W}), \sigma)$ and a morphism $f$ to $\F(h) \circ \F(g)$, where $f = h \circ g$ is the unique factorization with $h \in \mathcal{C}_{id}$ and $g \in \mathcal{C}_{tn}$, and where $\F(g)$, $\F(h)$ are respectively given in Definitions \ref{def:ctn} and \ref{def:definitionOfF2}. 
\end{definition}

This concludes the construction of $\F$. In order to prove functoriality of $\F$, we use the following proposition. 

\begin{proposition} \label{Fwelldef}
Given $g_{1}, g_{2} \in \mathcal{C}_{tn}$ and $h_{1}, h_{2} \in \mathcal{C}_{id}$ with $g_{1} \circ h_{1} = h_{2} \circ g_{2}$, we have 
\[ \F(g_{1}) \circ \F(h_{1}) = \F(h_{2}) \circ \F(g_{2}). \]
\end{proposition}
\begin{proof} 
First, underlying such a collection of morphisms, there is a pair of $k$-th order foliations with tubular neighbourhood embeddings $(M_{i}, F_{i}, \Phi_{i})$, $i = 1, 2$, and a single underlying morphism $f: (M_{1}, F_{1}) \to (M_{2}, F_{2})$ between the manifolds relating the foliations. Let $f_{*}(\Phi_{1}) = f \circ \Phi_{1} \circ \nu(f^{-1}) : \nu_{W, M_{2}} \to M_{2}$ and let $f^{*}(\Phi_{2}) = f^{-1} \circ \Phi_{2} \circ \nu(f) : \nu_{W, M_{1}} \to M_{1}$, both tubular neighbourhood embeddings. Then $g_{1}$ is the map $f$ which relates $f^{*}(\Phi_{2})$ to $\Phi_{2}$ and $g_{2}$ is the map $f$ which relates $\Phi_{1}$ to $f_{*}(\Phi_{1})$. Similarly, $h_{1}$ is the identity map going from $\Phi_{1}$ to $f^{*}(\Phi_{2})$ and $h_{2}$ is the identity map going from $f_{*}(\Phi_{1})$ to $\Phi_{2}$. 

Now let $\Phi' : \nu_{W,M_{1}} \times [0,1] \to M_{1} \times [0,1]$ be a path of tubular neighbourhood embeddings connecting $\Phi_{1}$ to $f^{*}(\Phi_{2})$. This induces a flat connection $\sigma'$ on $Fr^{k}(\nu_{W,M_{1}}) \times [0,1]$ such that the holonomy in the $[0,1]$-direction defines $\F(h_{1})$. Next, note that 
\[
f \times id : (M_{1} \times [0,1], F_{1} \times [0,1]) \to  (M_{2} \times [0,1], F_{2} \times [0,1])
\]
is a morphism of $k$-jets of foliations and $(f \times id)_{*}(\Phi')$ is a path of tubular neighbourhood embeddings for $W \times [0,1] \subset M_{2} \times [0,1]$ connecting $f_{*}(\Phi_{1})$ to $\Phi_{2}$. Hence, we get an induced flat connection $\sigma''$ on $\Fr^{k}(\nu_{W,M_{2}}) \times [0,1]$ whose holonomy in the $[0,1]$-direction defines $\F(h_{2})$. Furthermore, we have the induced map 
\[
j^{k}(\nu(f)) \times id : \Fr^{k}( \nu_{W, M_{1}}) \times [0,1] \to \Fr^{k}( \nu_{W, M_{2}}) \times [0,1]
\]
which sends the flat connection $\sigma'$ to $\sigma''$ and which restricts to $\F(g_2)$ along $\Fr^{k}( \nu_{W, M_{1}})  \times \{ 0 \} $ and to $\F(g_{1})$ along $\Fr^{k}( \nu_{W, M_{1}}) \times \{ 1 \}$. Since morphisms of flat connections intertwine holonomy, this shows that $\F(g_{1}) \circ \F(h_{1}) = \F(h_{2}) \circ \F(g_{2})$. 
\end{proof}
\begin{corollary}
The map $\F: \mathcal{C} \to \FlatCat(W,G_{k,l})$ of Definition \ref{functorFdef} is a functor.
\end{corollary}
\begin{proof}
First, note that the maps $\F$ defined on $\C_{id}$ and on $\C_{tn}$ were shown to be functors. Since the identity $id \in \C$ factors as a composition of identities, it is also clear that $\F(id) = id$. Hence, it remains to show that for a composable pair $f_{1}, f_{2} \in \C$, we have $\F(f_1 \circ f_2) = \F(f_1) \circ \F(f_2)$. 

Consider the factorizations $f_{1} = h_1 \circ g_1$ and $f_2 = h_{2} \circ g_{2}$. Then the composite $f_{1} \circ f_{2} = h_1 \circ g_1 \circ  h_{2} \circ g_{2}$ is not properly factored. Indeed, we must refactor the middle two terms as $g_1 \circ  h_{2} = h_{1}' \circ g_{2}'$, where $h_{1}' \in \C_{id}$ and $g_{2}' \in \C_{tn}$. This is graphically depicted in the following commutative diagram
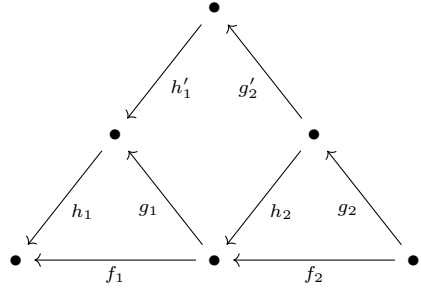
\begin{center}
\begin{tikzcd}
& & \bullet \ar[ddl,"h_1'"] & & &\\
& & & & &\\
& \bullet \ar[ddl,"h_1"] & & \bullet \ar[ddl,"h_2"] \ar[uul,"g_2'"] &  &\\
& & & & &\\
\bullet & & \bullet \ar[ll,"f_1"] \ar[uul,"g_1"] & & \ar[ll,"f_2"] \ar[uul,"g_2"] \bullet
\end{tikzcd}
\end{center}
We then have
\begin{align*}
\F(f_1 \circ f_2) &= \F(h_1 \circ g_1 \circ h_2 \circ g_2) = \F(h_1 \circ h_1' \circ g_2' \circ g_2) = \F(h_1\circ h_1') \circ \F(g_2' \circ g_2)\\ 
&= \F(h_1) \circ \F(h_1') \circ \F(g_2') \circ \F(g_2) = \F(h_1) \circ \F(g_1) \circ \F(h_2) \circ \F(g_2) = \F(f_1)\circ \F(f_2),
\end{align*}
where in the penultimate equality we apply Proposition \ref{Fwelldef} to obtain $\F(h_1')\circ \F(g_2') = \F(g_1) \circ \F(h_2)$. 
\end{proof}

Finally, we now verify that $\F$ factors through $J^k\C$.
\begin{proposition}
Given morphisms $f_0,f_1 \in \C$ we have:
\begin{equation*}
J^k(f_0) = J^k(f_1) \implies \F(f_0) = \F(f_1).
\end{equation*}
\end{proposition}
\begin{proof}
First, let $f : (M, F, \Phi) \to (M, F, \Phi)$ be a morphism in $\C$ such that $j^kf = j^k(id_M)$. It factors as $f = h \circ g$, where $g \in \C_{tn}$ has underlying diffeomorphism given by $f$ and $h : (M, F, f \circ \Phi \circ \nu(f)^{-1}) \to (M, F, \Phi)$ is in $\C_{id}$. Then $\F(g) = j^k(\nu(f)) = id$ and $\F(h) = \F(f \circ \Phi \circ \nu(f)^{-1}, \Phi) = id$, where the last identity follows from Lemma \ref{lem:independencePaths} and the fact that $j^k(f \circ \Phi \circ \nu(f)^{-1}) = j^k(\Phi)$. Hence $\F(f) = id$. 

Now consider the morphisms $f_{0}, f_{1} \in \C$ with $J^k(f_0) = J^k(f_1)$. Therefore, we have
\[
f_{0} : (M_{1}, F_{1}, \Phi_{1}) \to (M_{2}, F_{2}, \Phi_{2}), \qquad f_{1} : (M_{1}, F_{1}, \Phi'_{1}) \to (M_{2}, F_{2}, \Phi'_{2}),
\]
such that $j^k(\Phi_{i}) = j^{k}(\Phi'_{i})$, for $i = 1, 2$. Let $h_{i} : (M_{i}, F_{i}, \Phi_{i}) \to (M_{i}, F_{i}, \Phi'_{i}) \in \mathcal{C}_{id}$. By Lemma \ref{lem:independencePaths}, $\mathcal{F}(h_{i}) = id$ for $i = 1, 2$. Finally, consider the automorphism $f = f_0^{-1} h_{2}^{-1} f_{1} h_{1}$. Then $j^k(f) = j^k(id_{M_{1}})$, and so by the first paragraph we have $\F(f) = id$. Hence $\F(f_0) = \F(f_1)$. 
\end{proof}

\subsubsection{Intermezzo: Some properties of $\F$}

One drawback of our definition is that the functor $\F$ is not always straightforward to compute on morphisms. However, $\F$ does behave as expected (i.e. it amounts to taking $k$-jets) when we restrict it to germs of diffeomorphisms that respect the fibered structure of vector bundles.
\begin{lemma} \label{holonomycomputation}
Let $(E_{i}, F_{i}, id) \in \mathcal{C}$, for $i = 1, 2$, where $E_{i}$ are vector bundles and $id$ is the identity tubular neighbourhood embedding. Let $f: (E_{1}, F_{1}, id) \to (E_{2}, F_{2}, id)$ be a morphism which commutes with the projection maps $\pi: E_{i} \to W$. Then $\F(f) = j^{k}(f)$, the $k$-jet of $f$ along $W$. 
\end{lemma}
\begin{proof}
The map $f$ factors as $f = h \circ g$, where $g \in C_{tn}$ and $h : (E_2, F_{2}, \Phi) \to (E_2, F_{2}, id)$ is in $\C_{id}$, and where $\Phi = f \circ \nu(f)^{-1} : E_2 \to E_2$ is a TNE. Hence $\F(f) = \F(\Phi, id) \circ j^{k}(\nu(f))$. 

To compute $\F(\Phi, id)$ we choose the following path of tubular neighbourhood embeddings 
\[
\Phi'(v, t) = ((1-t)\Phi(v) + tv, t) : E_{2} \times [0,1] \to E_2 \times [0,1]
\]
and pullback the foliation $F \times [0,1]$ to get a connection $\sigma'$ on $\Fr^{k}(E_{2}) \times [0,1]$. Let $\sigma_0$ be the connection on $\Fr^{k}(E_{2}) \times [0,1]$ induced by $F \times [0,1]$ and note that $\hol^{\sigma_0}_{[0,1]} = id$. Then, since $\Phi'$ is fibre preserving, $j^{k}(\Phi')$ defines a gauge transformation of $\Fr^{k}(E_{2}) \times [0,1]$ which sends $\sigma'$ to $\sigma_{0}$. Therefore 
\[
\F(\Phi, id) = \hol_{[0,1]}^{\sigma'} = j^{k}(\Phi'_{1})^{-1} \circ \hol^{\sigma_0}_{[0,1]} \circ j^{k}(\Phi'_{0}) = j^k(\Phi).
\]
We conclude 
\[
\F(f) = \F(\Phi, id) \circ j^{k}(\nu(f)) = j^{k}(\Phi) \circ j^{k}(\nu(f)) = j^{k}(f).
\] 
\end{proof}

Lastly, we show that $\F$ is compatible with taking the normal derivative. To make sense of this, observe that a morphism of principal bundles  $g: \Fr^{k}(\nu_{W, M_{1}}) \to \Fr^{k}(\nu_{W, M_{2}})$ corresponds to the $k$-jet of a fibre preserving map between the underlying vector bundles. By taking the normal derivative, this induces a bundle morphism $\nu_{W, M_{1}} \to \nu_{W, M_{2}}$ that we denote $\nu(g)$.

\begin{lemma} \label{Flinearization}
Let $f: (M_{1}, F_{1}, \Phi_{1}) \to (M_{2}, F_{2}, \Phi_{2})$ be a morphism in $\mathcal{C}$ and let $\F(f) : (\Fr^{k}(\nu_{W, M_{1}}), \sigma_{1}) \to (\Fr^{k}(\nu_{W, M_{2}}), \sigma_{2})$ be the corresponding morphism of flat principal bundles. Then $\nu(\F(f)) = \nu(f)$. 
\end{lemma}
\begin{proof}
By definition $\F(f) = \F(\Phi') \circ j^{k}(\nu(f))$, for $\Phi'$ a path of TNEs, and so it suffices to show that $\nu(\F(\Phi')) = id$. Let $\sigma'$ be the flat connection on $\Fr^{k}(\nu_{W,M_2}) \times [0,1]$ obtained by pulling back $F_2 \times [0,1]$ by $\Phi'$. Then $\F(\Phi') = \hol_{[0,1]}^{\sigma'}$ is the parallel transport of $\sigma'$ along $[0,1]$. To compute $\nu(\hol_{[0,1]}^{\sigma'})$ we take linear approximations. The $k$-th order foliation $F_2 \times [0,1]$ induces the flat connection $\pi^*(\nabla_{F_2})$ on $\nu_{W, M_2} \times [0,1] = \pi^{*}(\nu_{W, M_2})$, where $\pi : W \times [0,1] \to W$ is the projection map, and $\nabla_{F_2}$ is the flat connection on $\nu_{W,M_2}$ induced by $F_{2}$. Since $\nu(\Phi') = id$ by definition, the linear approximation $\nu(\sigma')$ is also given by $\pi^*(\nabla_{F_2})$. As a result, the linear approximation $\nu(\hol_{[0,1]}^{\sigma'})$, which is the parallel transport of the linear approximation $\nu(\sigma') = \pi^*(\nabla_{F_2})$ along $[0,1]$, is the identity map. 
\end{proof}

\subsubsection{Step 2: Essential surjectivity}

Our next goal is to show that the functor $\F$ is essentially surjective. For this we will need the following auxiliary result. 

\begin{lemma} \label{representability}
Every principal $G_{k,l}$-bundle over $W$ is isomorphic to $\Fr^{k}(E)$ for some vector bundle $\pi: E \to W$. 
\end{lemma}

Note that it suffices to prove Lemma \ref{representability} for the universal $G_{k,l}$-bundle $EG_{k,l}$, since all others are obtained from it by pullback. For this reason, we introduce a model of $EG_{k,l}$ which is particularly well-suited to our arguments. 

\begin{lemma} \label{classifyingspacemodel}
Consider the space 
\begin{align*}
\big\{ j_{0}^{k} f \ | \    \begin{array}{lr}
\qquad f : (\mathbb{R}^l, 0) \to (\mathbb{R}^{\infty}, 0) \text{ is an embedding germ whose}   \\
\text{ image is contained in an $l$-dimensional linear subspace}
\end{array}    \big\},
\end{align*}
where $j_{0}^{k} f$ denotes the $k$-jet at $0$. It is contractible and has a free $G_{k,l}$-action given by reparametrisation in the source. In particular, it is a model for $EG_{k,l}$. Moreover, its quotient by $G_{k,l}$ is the grassmannian of $l$-planes $\Gr(\mathbb{R}^{\infty}, l)$, which is then a model for $BG_{k,l}$. 
\end{lemma}
\begin{proof}
Let us write $E$ and $B$ for the models of $EG_{k,l}$ and $BG_{k,l}$ under consideration, respectively. Observe that an element $j_{0}^{k}f \in E$ defines a linear subspace $A \subset \mathbb{R}^{\infty}$. Concretely, $A$ is the image of the differential $d_{0}f$. It is clear that the $G_{k,l}$-action is free and that the orbit of $j_{0}^{k}f$ consists of all $k$-jets with image $A$. Hence, the quotient $E/G_{k,l}$ is identified with $B$, the grassmannian of $l$-planes. 

It remains to prove contractibility. We first define a deformation retraction $h$ from $E$ onto the subspace $E\mathrm{GL}_{l}$ of linear embeddings. It is given by the linear interpolation 
\[
h(j_0^{k}f, s) = s d_{0}f + (1 - s) j_{0}^{k}f,
\]
for $0 \leq s \leq 1$. This finishes the proof because $E\mathrm{GL}_{l}$ is contractible. 
\end{proof}

\begin{corollary}
Let $\pi: \gamma \to \Gr(\mathbb{R}^{\infty}, l)$ be the tautological rank $l$ vector bundle. Then $EG_{k,l} = \Fr^{k}(\gamma)$. 
\end{corollary}

To see that the functor $\F$ is essentially surjective, let $(P, \nabla)$ be a principal $G_{k,l}$-bundle equipped with a flat connection. By Lemma \ref{representability}, there is a vector bundle $E \to W$ and an isomorphism $\Fr^{k}(E) \cong P$. Let $\sigma$ be the induced flat connection on $\Fr^{k}(E)$. By Lemma \ref{kthorderAtiyahisAtiyah} and Theorem \ref{splittingprop}, the connection $\sigma$ corresponds to a $k$-th order foliation $F$ along $W$. Then $\F(E, F, id) \cong (P, \nabla)$, where $id$ is the identity tubular neighbourhood embedding. 

\subsubsection{Step 3: Surjectivity on morphisms}

Our final goal is to show that the functor $\F$ is full. Choose two objects of $\mathcal{C}$, which we may assume to be of the form $(E_{i}, F_{i}, id)$, for $i = 1, 2$, where $E_{i}$ are vector bundles and the tubular neighbourhood embeddings are taken to be the identity. Denote the associated flat connections on the $k$-th order frame bundles by $\sigma_{1}$ and $\sigma_{2}$ respectively. Hence, we want to show that the following map is surjective 
\[
\F: \Hom_{\mathcal{C}}((E_{1}, F_{1}, id), (E_{2}, F_{2}, id)) \to \Hom_{ \FlatCat(W, G_{k,l})}((\Fr^{k}(E_{1}), \sigma_{1}), (\Fr^{k}(E_{2}), \sigma_{2})). 
\]
Let $T: \Fr^{k}(E_{1}) \to \Fr^{k}(E_{2})$ be a morphism sending $\sigma_{1}$ to $\sigma_{2}$. It can be viewed as the $k$-jet of a fibrewise map from $E_{1}$ to $E_{2}$ fixing $W$ pointwise, and it can be represented by a germ of diffeomorphism $f$ (Lemma \ref{lem:polynomialRepresentatives}). By construction this map sends $F_{1}$ to $F_{2}$ and so is a morphism in $\mathcal{C}$. Furthermore, by Lemma \ref{holonomycomputation}, $\F(f) = T$. This completes the proof of Proposition \ref{FunctorF} and therefore of Theorem \ref{RHthm}. \hfill\qedsymbol

\subsubsection{Proof of Theorem \ref{innercharacterization}}

In this section we characterize the kernel of $RH$ and thus establish Theorem \ref{innercharacterization}.

According to the classical Riemann-Hilbert correspondence (Theorem \ref{MainRHtheorem}), it is enough to characterise the kernel of $J^k\F$ instead, which in turn is determined by the kernel of $\F$. Let $(M, F)$ be a $k$-th order foliation and let $\Tan^{k}(W,F)$ be the associated singular foliation of $b^{k+1}$-type. We must show that an automorphism $\phi \in \Aut(M,F)$ is in the kernel of $\F$ if and only if it is the time-1 flow of a time-dependent vector field $V_t \in \Tan^{k}(W,F)$ which vanishes on $W$. Without loss of generality we may assume that $M = E$ is a vector bundle over $W$, which is equipped with the identity tubular neighbourhood embedding $\Phi = id$. 

Suppose first that $\phi = \phi_{1}$ is the time-$1$ flow of some $(V_t)_{t \in [0,1]} \in \Tan^{k}(W,F)$ vanishing along $W$. By \cite[Proposition 1.6]{AS09} the flow $\phi_{t}$ preserves $\Tan^{k}(W, F)$ and hence $F$. Since $V_{t}|_{W} = 0$, and assuming that $k \geq 1$, the linear approximation of $V_{t}$ vanishes, implying that $\nu(\phi_{t}) = id$. In other words, $\phi_{t}$ is a path of tubular neighbourhood embeddings connecting $id$ to $\phi$. These can be assembled into a tubular neighbourhood embedding $\Phi' : E \times [0,1] \to E \times [0,1]$ by setting $\Phi'(v,s) := (\phi_s(v),s)$. By construction, $\Phi'$ is itself the time-$1$ flow of $\tilde{V} = (sV_{ts},0) \in \Tan^{k}(W \times [0,1],F \times [0,1])$. $\Phi'$ is thus an inner automorphism of $F \times [0,1]$.

Now, $\F(\phi)$ can be computed as the parallel transport from $\Fr^k(E) \times \{ 1 \}$ to $\Fr^{k}(E) \times \{0\}$ of the connection $\sigma'$ obtained by pulling back $F \times [0,1]$ using $\Phi'$. However, this pullback is $F \times [0,1]$ itself, since $\Phi'$ is an inner automorphism. It follows that its parallel transport in the direction of $[0,1]$ is trivial. This implies that $\F(\phi) = id$, establishing one direction of the theorem. 

Now suppose that $\phi$ is in the kernel of the Riemann-Hilbert map, meaning that $\F(\phi) = id$. According to Lemma \ref{Flinearization}, $\nu(\phi) = id$, making $\phi : E \to E$ a tubular neighbourhood embedding. In order to compute $\F(\phi)$, we choose a family of tubular neighbourhood embeddings $\Phi' : E \times [0,1] \to E \times [0,1]$ interpolating between $id$ and $\phi$. Then we pullback the foliation $F \times [0,1]$ along $\Phi'$ to define a connection $\sigma'$ on $\Fr^k(E) \times [0,1]$ and define $\F(\phi)$ to be its parallel transport along $[1,0]$. In our case, this is the identity. Let $g_{t}$ denote the parallel transport along $[0,t]$ and let $G: \Fr^{k}(E) \times [0,1] \to \Fr^{k}(E) \times [0,1]$ be the gauge transformation $G(p,t) = (g_t(p),t)$. Applying it to $\sigma'$, we get $G^{*} \sigma'$, which is the connection induced by $F \times [0,1]$. We can lift $g_{t}$ to a family $\psi_{t}$ of fibrewise (local) diffeomorphisms of $E$ fixing $W$ pointwise. Since $g_{0} = g_{1} = id$, we can assume that $\psi_{0} = \psi_{1} = id$. Let $\Psi : E \times [0,1] \to E \times [0,1]$ be the map $\Psi(v,t) = (\psi_{t}(v),t)$, and let $\Phi = \Phi' \circ \Psi$. We can write $\Phi(v,t) = (\phi_{t}(v),t)$, where $\phi_{t} : E \to E$ is a family of diffeomorphisms interpolating between $id$ and $\phi$ and fixing $W$ pointwise. By construction, $\Phi$ is an automorphism of $F \times [0,1]$. Therefore $\Phi_{*}(\partial_{t}) = (V_{t}, \partial_{t})$, where $V_{t} \in \Tan^{k}(W,F)$ is a time-dependent vector field $k$-tangent to $F$. But then $\phi_{t}$ is the flow of $V_{t}$, since $\phi_0 = id$. Furthermore, $V_{t}$ vanishes along $W$ because $\phi_{t}$ fixes $W$ pointwise. This completes the other direction of the theorem.  \hfill\qedsymbol

\subsection{Topological Riemann-Hilbert functor} \label{TopRHstatement}

We will now state and prove a topological version of Theorem \ref{RHthm}. This will make use of the notions of topological groupoids and fibrations, which are recalled in Appendix \ref{sec:topgroupoids}. 

In this section we fix a closed and connected manifold $W$. We also fix a rank $l$ vector bundle $E \to W$, as well as a framing $\tilde{\phi} : \mathbb{R}^l \to E_{x}$ above the point $x \in W$ with induced framing $\phi : G_{k,l} \rightarrow \Fr^k(E_x)$.

\subsubsection{The character stack}

Since $\pi_{1}(W,x)$ is finitely presented, $\Hom(\pi_{1}(W,x), G_{k,l})$ is an algebraic variety. As a result, the action groupoid 
\[ \Rep(\pi_{1}(W, x), G_{k,l}) = G_{k,l} \ltimes \Hom(\pi_{1}(W,x), G_{k,l}) \]
is a groupoid in algebraic varieties, which presents the \emph{$G_{k,l}$-character stack} of $W$. In particular, it has the structure of a topological groupoid. 

By Corollary \ref{connectivityofHoms}, the space $\Hom(\pi_{1}(W,x), G_{k,l})$ decomposes into subspaces $\Hom_{(\Fr^k(E'),\phi')}(\pi_{1}(W, x), G_{k,l})$ which are indexed by isomorphism classes of framed principal $G_{k,l}$-bundles $(\Fr^k(E'),\phi')$. Each of these subspaces are themselves unions of connected components of $\Hom(\pi_{1}(W,x), G_{k,l})$. Furthermore, each subspace $\Hom_{(\Fr^k(E'),\phi')}(\pi_{1}(W, x), G_{k,l})$ determines a full topological subgroupoid of $\Rep(\pi_{1}(W, x), G_{k,l})$: 
\[
\Rep_{(\Fr^k(E'),\phi')}(\pi_{1}(W, x), G_{k,l}) = G_{k,l}(\phi') \ltimes \Hom_{(\Fr^k(E'),\phi')}(\pi_{1}(W, x), G_{k,l}),
\]
where $G_{k,l}(\phi')$, as in Lemma \ref{surjectivegauge}, is an open subgroup of $G_{k,l}$, which is also determined by the framed principal bundle $(\Fr^k(E'),\phi')$. In fact, since $G_{k,l}$ is homotopic to $\mathrm{GL}(\mathbb{R}^l)$ by Corollary \ref{cor:equivarianthomotopy}, it has only two connected components, and therefore, by Remark \ref{Gphiconj}, $G_{k,l}(\phi')$ only depends on the underlying vector bundle $E'$. 

In this section, we will be concerned with the subgroupoid $\Rep_{(\Fr^k(E),\phi)}(\pi_{1}(W, x), G_{k,l})$ associated to the fixed framed principal bundle $(\Fr^k(E), \phi)$. 

\subsubsection{The groupoid of flat connections} \label{gpdflatcondefsec}

Let $\FlatCat(W, \Fr^{k}(E))$ denote the full subgroupoid of $\FlatCat(W, G_{k,l})$ consisting of flat connections on $\Fr^{k}(E)$. We identify it with $\FlatCat(W, \phi, \Fr^{k}(E))$ (Definition \ref{def:framedflatbundle}) using the chosen framing. Therefore, it is the action groupoid 
\[ \FlatCat(W, \Fr^{k}(E)) = \Gau(\Fr^k(E)) \ltimes \FlatObj(\Fr^k(E)) \]
where $\Gau(\Fr^k(E))$ is the gauge group of $\Fr^{k}(E)$ and $\FlatObj(\Fr^k(E))$ is the space of flat connections on $\Fr^{k}(E)$. By Lemma \ref{continuousaction}, it is a topological groupoid, with both $\Gau(\Fr^k(E))$ and $\FlatObj(\Fr^k(E))$ equipped with the Whitney $C^{\infty}$-topology. By Theorem \ref{topologicalRHforPB}, the restricted Riemann-Hilbert map 
\[
RH_{(\Fr^{k}(E), \phi)} : \FlatCat(W, \Fr^{k}(E)) \to \Rep_{(\Fr^k(E),\phi)}(\pi_{1}(W, x), G_{k,l}) 
\]
is a Morita equivalence. 

\subsubsection{The groupoid of $k$-th order foliations}  \label{gpdkjetfoldefsec}

Let $J^k\Diff(E,W)$ be the group of $k$-jets of diffeomorphisms of $E$ along $W$ which fix $W$ pointwise and let $J^{k}\FolObj(W,E)$ be the set of $k$-th order foliations in $E$ along $W$. They can be combined (Lemma \ref{lem:factoringViakJets}) into the action groupoid
\[ J^{k}\FolCat(W,E) = J^k\Diff(E,W) \ltimes J^{k}\FolObj(W,E), \]
which is a full subcategory of $J^{k}\FolCat(W,l)$. By taking the identity tubular neighbourhood embedding $id: E \to E$, we may view $J^{k}\FolCat(W,E)$ as a full subcategory of $J^k\mathcal{C}$ (recall Definition \ref{definitionofJkC}).

We now equip $J^{k}\FolCat(W,E)$ with a suitable topology. To this end, consider the Grassmannian bundle $\Gr(TE,n-l) \rightarrow E$ of corank-$l$ subspaces and let 
\[
H \subset J^k_W\Gamma(\Gr(TE,n-l))
\] be the subspace of involutive $k$-jets of corank-$l$ distributions tangent to $W$ (Example \ref{ex:grass2}). The following result follows from Lemma \ref{realizabilitylemma} and Remark \ref{rem:grass}. 
\begin{lemma}\label{lem:homeoWithFlat}
$J^{k}\FolObj(W,E)$ is in canonical bijection with $H$. 
\end{lemma}
This identification allows us to topologise $J^{k}\FolObj(W,E)$ using the $C^\infty$-topology. We also endow $J^k\Diff(E,W)$ with the $C^\infty$-topology. In Example \ref{ex:grass3} it is shown that  $J^k\Diff(E,W)$ acts continuously on $J^{k}\FolObj(W,E)$. 
\begin{corollary} \label{kjetfibrationproperty}
$J^{k}\FolCat(W,E)$ is a topological groupoid.
\end{corollary}

Recall now that we have a bijection between $J^{k}\FolObj(W,E)$ and $\FlatObj(\Fr^k(E))$ (Theorem \ref{splittingprop}). Furthermore, recall that the latter, as a space of connections, is endowed with the $C^\infty$-topology.
\begin{lemma} \label{lem:canident}
The bijection $J^{k}\FolObj(W,E) \rightarrow \FlatObj(\Fr^k(E))$ is a homeomorphism.
\end{lemma}
\begin{proof}
As recalled above and explained in Example \ref{ex:grass2}, $J^{k}\FolObj(W,E)$ is a subspace of $J^k_W\Gamma(B)$, where $B \rightarrow E$ is the bundle of corank-$l$ subspaces of $TE$ that are transverse to the projection $\pi: E \rightarrow M$. Consider then the bundle $B' \rightarrow E$ of monomorphisms $\pi^*TM \rightarrow TE$ that are right inverses of $d\pi$. The fibered map $B' \rightarrow B$ that assigns to each element of $B'$ its image is a homeomorphism. By functoriality we then have a homeomorphism $J^k_W\Gamma(B') \rightarrow J^k_W\Gamma(B)$ that allows us to see $J^{k}\FolObj(W,E)$ as a subspace of $J^k_W\Gamma(B')$. 

Consider then the set $S$ of vector bundle splittings $TW \rightarrow \at^k(E)$. Equivalently, this is the set of connections on $\Fr^k(E)$ and is topologised as such, using the $C^\infty$-topology. $\FlatObj(\Fr^k(E))$ is the subspace of flat connections, which corresponds to the Lie algebroid splittings. Recall that sections of $\at^k(E)$ are $k$-jets of projectable vector fields. It thus follows that $S$ is in canonical bijection with $J^k_W\Gamma(B')$. Showing that this bijection is a homeomorphism will conclude the proof of the lemma. This is immediate when spelled out in local coordinates, where both spaces are seen to correspond to tuples of fibrewise polynomial projectable vector fields. 
\end{proof}

\subsubsection{Statement of the theorem}
Restricting the functor $J^k\F: J^k\mathcal{C} \to \FlatCat(W,G_{k,l})$ of Proposition \ref{FunctorF} to $J^{k}\FolCat(W,E)$ yields a functor 
\[ J^k\F: J^{k}\FolCat(W,E) \to \FlatCat(W, \Fr^{k}(E)). \]
Composing this with the functor $RH_{(\Fr^{k}(E), \phi)}$ from subsection \ref{gpdflatcondefsec} gives $RH_{(E, \phi)} = RH_{(\Fr^{k}(E), \phi)} \circ J^k\F$, the \emph{Riemann-Hilbert functor for $k$-th order foliations in $E$}. We establish the following topological version of Theorems \ref{RHthm} and \ref{innercharacterization}.
\begin{theorem} \label{topologicalRHfoliation}
The Riemann-Hilbert functor 
\[ RH_{(E, \phi)} : J^{k}\FolCat(W,E) \to \Rep_{(\Fr^k(E),\phi)}(\pi_{1}(W, x), G_{k,l}) \]
is a fibration of topological groupoids. In particular:
\begin{itemize}
\item The orbit spaces of $\Rep_{(\Fr^k(E),g)}(\pi_{1}(W, x), G_{k,l})$ and $J^{k}\FolCat(W,E)$ are homeomorphic.
\item The kernel of $RH_{(E, \phi)}$ is a bundle of topological groups $\mathcal{K}$ over $J^{k}\FolObj(W,E)$. 
\end{itemize}
Furthermore, the fibre of $\mathcal{K}$ above a given $k$-th order foliation $F$ is the group of $k$-jets of diffeomorphisms that can be represented by the time-$1$ flows of vector fields that are both $k$-tangent to $F$ and that vanish along $W$. 
\end{theorem}

\subsubsection{Proof of Theorem \ref{topologicalRHfoliation}}
We will establish Theorem \ref{topologicalRHfoliation} in several steps. Note first that it suffices to prove that $RH_{(E, \phi)}$ is a fibration. The remaining parts of the theorem then follow from this result along with Lemma \ref{fibrationlemma} and Theorem \ref{innercharacterization}. 

By Theorem \ref{topologicalRHforPB}, $RH_{(\Fr^{k}(E), \phi)}$ is a Morita equivalence. Therefore, because the base map of $J^{k}\F$ is a homeomorphism by Lemma \ref{lem:canident}, it suffices to show that 
\[ J^k\F: J^{k}\FolCat(W,E) \to \FlatCat(W, \Fr^{k}(E))\]
is a continuous surjective topological submersion. We start by showing continuity. 
\begin{lemma} \label{continuityofJkF}
The functor $J^k\F$ is continuous. 
\end{lemma}
\begin{proof}
Because it will be convenient, we slightly modify the description of $J^k \F$ using Proposition \ref{Fwelldef}. Let $j^kf \in J^k\FolCat(W,E)$ be a morphism with source $F$. Represent it by a local diffeomorphism $f$ by taking a polynomial representative and note that $f^{-1}\circ \nu(f)$ is a tubular neighbourhood embedding (using the definition from Appendix \ref{ssec:TNE}). Let $\Phi'$ be a path of tubular neighbourhood embeddings from $id$ to $f^{-1}\circ \nu(f)$. Use $\Phi'$ to pullback $F \times [0,1]$ to another $k$-th order foliation with underlying flat connection $\sigma'$ on $\Fr^k(E) \times [0,1]$. Then $J^k\F(f) = j^k(\nu(f)) \circ \hol_{[0,1]}^{\sigma'}$, where $\hol_{[0,1]}^{\sigma'}$ is the holonomy of $\sigma'$ in the $[0,1]$ direction. We must therefore prove continuity of $\hol_{[0,1]}^{\sigma'}$. 

According to Lemma \ref{lem:polynomialRepresentatives}, taking polynomial representatives $j^kf \mapsto f$ is a continuous map $J^k\Diff(E,M) \rightarrow C^\infty(E,E)$. The map sending $f \mapsto f^{-1} \circ \nu(f)$ is continuous as well, and takes values in tubular neighbourhood embeddings. According to Lemma \ref{lem:betterCerf}, the latter form a contractible space and, moreover, the assignment $f^{-1} \circ \nu(f) \mapsto \Phi'$ is continuous. Lastly, taking the pullback and performing parallel transport are continuous for the Whitney topologies.
\end{proof}

By Theorem \ref{RHthm}, the functor $J^{k}\F$ is full, and by Lemma \ref{lem:canident} it is a homeomorphism at the level of objects. As a result, $J^{k}\F$ is surjective at the level of morphisms. It therefore only remains to show that $J^{k}\F$ is a topological submersion. The following lemma shows that checking this property can be simplified.

\begin{lemma}
Let $\Kcal = \ker(J^k\F)$ be the kernel of $J^k\F$ and assume that the source map $s: \Kcal \to J^{k}\FolObj(W,E)$ is a topological submersion. Then $J^k\F$ is a topological submersion. 
\end{lemma}
\begin{proof}
We will use the identification between $J^{k}\FolObj(W,E)$ and $\FlatObj(\Fr^k(E))$ throughout. We must show that every morphism $(\phi_0, \sigma_0) \in J^{k}\FolCat(W,E)$ is in the image of a local section defined around $(g_{0}, \sigma_{0}) := J^k\F(\phi_{0}, \sigma_{0})$.

First note that any gauge transformation $g \in \Gau(\Fr^k(E))$ can be viewed as the $k$-jet of a fibrewise polynomial map $\phi_{g}$ of $E$ (Lemma \ref{lem:polynomialRepresentatives}). This defines a continuous map
\[ \rho: \FlatCat(W, \Fr^{k}(E)) \to J^{k}\FolCat(W,E). \]
By Lemma \ref{holonomycomputation}, $J^k\F \circ \rho = id$. We can then set $(k_0, \sigma_0) := \big( \rho(g_0, \sigma_0) \big)^{-1} \ast (\phi_0, \sigma_0) \in \mathcal{K}$, leading to the factorization $(\phi_0, \sigma_0) = \rho(g_0, \sigma_0) \ast (k_0, \sigma_0)$.

Since $s : \mathcal{K} \to J^{k}\FolObj(W,E)$ is a topological submersion, there is an open set $U \subset J^k\FolObj(W,E)$ containing $\sigma_{0}$, and a continuous map $l : U \to \mathcal{K}$ such that $l(\sigma_{0}) = (k_{0}, \sigma_{0})$ and $s \circ l = id$. Since $\mathcal{K}$ is a bundle of groups we have $t = s$ and hence it also follows that $t \circ l = id$. Observe that $\Gau(\Fr^k(E)) \times U \subset \FlatCat(W, \Fr^{k}(E)) $ is an open neighbourhood of $(g_{0}, \sigma_{0})$. With this in mind we define the map 
\[
L : \Gau(\Fr^k(E)) \times U \to  J^{k}\FolCat(W,E), \qquad (g, \sigma) \mapsto \rho(g, \sigma) \ast l(\sigma). 
\]
Then $J^k\F \circ L = id$ and $L(g_{0}, \sigma_{0}) = \rho(g_0, \sigma_0) \ast (k_0, \sigma_0) = (j^k\phi_0, \sigma_0) $. This shows that $J^k\F$ is indeed a topological submersion.
\end{proof}

Finally, we finish the proof of Theorem \ref{topologicalRHfoliation}.
\begin{lemma}
The map $s: \mathcal{K} \to J^k\FolObj(W,E)$ is a topological submersion. 
\end{lemma} 
\begin{proof}
Let $(\phi_0, \sigma_0) \in \mathcal{K}$. We will construct a continuous map $\lambda : J^k\FolObj(W,E) \to \mathcal{K}$ such that $s \circ \lambda = id$ and $\lambda(\sigma_0) = (\phi_0, \sigma_0)$. First, let $V_{t} \in \Tan^{k}(W, \sigma_0)$ be a time-dependent vector field on $E$ which vanishes along $W$ and whose time-$1$ flow has $\phi_0$ as its $k$-jet. Such a vector field exists by Theorem \ref{innercharacterization}. In order to define the map $\lambda$, we will deform $V_{t}$ along variations of $\sigma_{0}$. 

Recall that the $k$-th order foliation $\sigma_0$ can be viewed as a splitting of $\at^{k}(E)$ (Theorem \ref{splittingprop}), and that this determines a (possibly non-integrable) distribution $\xi_{\sigma_0} \subset TE$ that is fibrewise polynomial. Then $\Tan^{k}(W, \sigma_0) = \xi_{\sigma_0} \oplus \mathcal{I}^{k+1} \pi^*(E)$. Using this decomposition, we write $V_{t} = U_{t} + K_{t}$ uniquely, where $U_{t} \in \xi_{\sigma_{0}}$ and $K_{t} \in \mathcal{I}^{k+1} \pi^*(E)$. Note that both $U_{t}$ and $K_{t}$ vanish along $W$. 

Given another $k$-th order foliation $\sigma$, there is a corresponding distribution $\xi_{\sigma} \subset TW$. This induces the direct sum decomposition $TE = \xi_{\sigma} \oplus \pi^*(E)$, and hence a projection map $p_{\sigma} : TE \to \xi_{\sigma}$. Define 
\[ V_{t}^\sigma = p_{\sigma}(U_{t}) + K_{t} \in \Tan^{k}(W, \sigma), \]
which vanishes along $W$. Let $\phi^{\sigma}$ be the time-$1$ flow of $V_{t}^\sigma$. Then by Theorem \ref{innercharacterization}, $(\phi^{\sigma}, \sigma) \in \mathcal{K}$. Hence we define $\lambda(\sigma) = (\phi^{\sigma}, \sigma)$, which is evidently a section of the source map. Furthermore, since $V_{t}^{\sigma_{0}} = V_{t}$, it is clear that $\lambda(\sigma_{0}) = (\phi_{0}, \sigma_{0})$. Finally, $\lambda$ is a continuous map because $V_{t}^\sigma$ varies continuously with $\sigma$ and because the flow depends continuously on the vector field. This completes the proof of Theorem \ref{topologicalRHfoliation}.
\end{proof}

\subsection{Isotopy classification} \label{Isotopy}
Let $M$ be a manifold and let $W$ be a closed and connected submanifold of codimension $l$. In this section we will classify $k$-th order foliations in $M$ along $W$ up to isotopies in $M$. When the codimension $l = 1$, this will also lead to a classification of hypersurface algebroids. 

Before we begin, we make a few comments about different possible classification results. In Sections \ref{CatRHstatement} and \ref{TopRHstatement}  we studied the local classification problem for $k$-th order foliations along $W$. Local means that we focused on the classification up to $k$-jets of diffeomorphisms defined in a neighbourhood of $W$. In Remark \ref{rem:achoice} we observed that the set of isomorphism classes of $k$-th order foliations does not change if we consider them up to \emph{germs} of diffeomorphisms instead, but that working with $k$-jets has several advantages when considering topologies. We will further elaborate on the issues with germ topologies in Remark \ref{germproblem}. 

Given an ambient manifold $M$ containing $W$, one may wish to classify $k$-th order foliations up to diffeomorphisms of $M$. However, this is extremely difficult and does not admit a general answer. This is because locally defined diffeomorphisms do not always extend to the entirety of $M$ and determining when they do is a difficult problem about the homotopy type of the diffeomorphism group. On the other hand, the related problem of classifying foliations up to \emph{isotopy} does admit a general answer. This is because of the isotopy extension theorem, which states that isotopies defined in a local neighbourhood of $W$ can always be extended to all of $M$. 

\subsubsection{Setup and main result}
Consider the following groups: $\Diff(M,W)$, the group of diffeomorphisms of $M$ which fix $W$ pointwise, $\Diff_0(M, W)$, the connected component of the identity (i.e. the group of isotopies), and $J^k\Diff(M,W)$ and $J^k\Diff_0(M, W)$, the respective groups of $k$-jets along $W$. All of these groups act on  $J^k\FolObj(W,M)$, the set of $k$-th order foliations in $M$ defined along $W$.

\begin{definition}
The groupoid of $k$-th order foliations in $M$ along $W$ taken up to isotopy is the action groupoid
\[ \IsoCat^{k}(W,M) := \Diff_0(M, W) \ltimes J^k\FolObj(W,M). \qedhere \] 
\end{definition}

Our goal is to construct a Riemann-Hilbert functor from $\IsoCat^{k}(W,M)$ to a version of the character variety which accounts for isotopy (see Appendix \ref{isotopyRH} for details). We will construct this functor in several steps. 

To this end, it will be useful to consider the following action groupoid 
\[
J^{k}_{0}\FolCat(W,M) := J^k\Diff_0(M,W) \ltimes J^k\FolObj(W,M),
\]
which, upon choosing a tubular neighbourhood embedding $\Phi: \nu_{W} \to M$, is isomorphic to an open subgroupoid of $J^{k}\FolCat(W,\nu_{W})$. The operation of taking the $k$-jet of a diffeomorphism defines a functor 
\[ \mathcal{G} : \IsoCat^{k}(W,M)  \to J^{k}_{0}\FolCat(W,M), \]
which is a fibration of topological groupoids by Lemma \ref{lem:takingJets}.

\begin{remark} \label{germproblem}
We return to Remark \ref{rem:achoice} to discuss our choice of working with $k$-jets of diffeomorphisms rather than germs. The key issue is that the germ analogue of $\G$ is not a fibration because it does not admit enough continuous sections (similarly, the germ analogue of Lemma \ref{lem:takingJets} is not true). This is because the natural topology in germ space (namely, the pullback topology from $J^\infty$) does not guarantee continuity when taking representatives.

It is instructive to keep in mind the following example: let $g: \mathbb{R} \to \mathbb{R}$ be a smooth function which is `flat' at the origin, in the sense that $j_0^{\infty} g = 0$, and which is positive away from $0$. Define $f_{t} : \mathbb{R} \to \mathbb{R}$ to be the path of germs of functions around $0$ with $f_{t} = 0$ for $t$ rational and $f_{t} = g$ otherwise. This path of germs is continuous but cannot be lifted to a continuous family of representatives. This is a well-known phenomenon (see for instance the discussion in \cite[p.35]{Gr86}), which is addressed by introducing a suitable quasi-topology on the set of germs. Our jet approach avoids the use of quasi-topologies. 
\end{remark}

Next, consider the following subgroupoid of flat connections  
\[ \FlatCat_0(W, \Fr^{k}(\nu_{W})) = \Gau^0(\Fr^k(\nu_{W})) \ltimes \FlatObj(\Fr^k(\nu_{W})), \]
where $\Gau^0(\Fr^k(\nu_{W}))$ is the identity component of the gauge group of $\Fr^k(\nu_{W})$. Choose a tubular neighbourhood $\Phi: \nu_{W} \to M$ and let $\Phi^{*} : J^{k}_{0}\FolCat(W,M) \to J^{k}\FolCat(W,\nu_{W})$ be the corresponding functor. Now consider the composition 
\[
J^{k}\F \circ \Phi^* \circ \mathcal{G} : \IsoCat^{k}(W,M) \to \FlatCat_0(W, \Fr^{k}(\nu_{W})),
\]
where $J^{k}\F$ is the functor defined in Proposition \ref{FunctorF}. In the proof of Theorem \ref{topologicalRHfoliation}, the functor $J^{k}\F$ was shown to be a fibration. Therefore, the composite $J^{k}\F \circ \Phi^* \circ \mathcal{G}$ is also a fibration. 

Finally, choose a basepoint $x \in W$ and a framing $\phi : G_{k,l} \to \Fr^k(E|_x)$, and recall the action groupoid 
\[
\IsotopyCat(\Fr^k(\nu_{W}),\phi) = G_{k,l}^0 \ltimes \IsotopyObj(\Fr^k(\nu_{W}),\phi)
\]
from Appendix \ref{isotopyRH}, where $G_{k,l}^0$ is the connected component of the identity and $\IsotopyObj(\Fr^k(\nu_{W}),\phi)$ is a covering space of the representation variety $\Hom_{(\Fr^k(\nu_{W}),\phi)}(\pi_{1}(W, x), G_{k,l})$. In Theorem \ref{prop:RHisotopy} we show that there is a Riemann-Hilbert functor 
\[ RH^{0}_{(\Fr^{k}(\nu_{W}), \phi)} : \FlatCat_{0}(W, \Fr^{k}(\nu_{W})) \to \IsotopyCat(\Fr^k(\nu_{W}),\phi) \] 
which is a fibration and a Morita equivalence. Composing all of the functors, we define 
\[
RH^{0}_{(M,\phi)} = RH^{0}_{(\Fr^{k}(\nu_{W}), \phi)} \circ J^{k}\F \circ \Phi^* \circ \mathcal{G} : \IsoCat^{k}(W,M)  \to \IsotopyCat(\Fr^k(\nu_{W}),\phi), 
\]
the \emph{Riemann-Hilbert functor for $k$-th order foliations in $M$ along $W$ taken up to isotopy}. The proof of the following theorem is now immediate. 

\begin{theorem} \label{RHuptoisotopy}
The Riemann-Hilbert functor for $k$-th order foliations in $M$ along $W$ taken up to isotopy 
\[
RH^{0}_{(M,\phi)} : \IsoCat^{k}(W,M)  \to \IsotopyCat(\Fr^k(\nu_{W}),\phi) 
\]
is a fibration of topological groupoids. In particular:
\begin{itemize}
\item $RH^{0}_{(M,\phi)}$ induces a homeomorphism between the space of isotopy classes of $k$-th order foliations in $M$ along $W$ and the orbit space of $\IsotopyCat(\Fr^k(\nu_{W}),\phi)$.
\item The kernel of $RH^{0}_{(M,\phi)}$ is a bundle of topological groups $\mathcal{R}$ over $J^k\FolObj(W,M)$. 
\end{itemize}
Furthermore, given a $k$-th order foliation $F$ and an isotopy $\psi \in \Diff_0(M, W)$, the pair $(\psi, F) \in \mathcal{R}$ if and only if the germ of $\psi$ along $W$ coincides with the time $1$-flow of a time-dependent vector field $V_{t} \in \Tan^{k}(W, F)$ that vanishes along $W$.
\end{theorem}

\subsubsection{Classification of hypersurface algebroids}
We end this section by specializing to the case of a codimension $1$ hypersurface $W \subset M$ where the classification dramatically simplifies. Indeed, the argument of Section \ref{sec:linebundles} shows that in this case
\[
\IsotopyCat(\Fr^k(\nu_{W}),\phi) = G_{k,1}^0 \ltimes \Hom_{(\Fr^k(\nu_{W}),\phi)}(\pi_{1}(W, x), G_{k,1}). 
\]
Furthermore, Proposition \ref{Singfoliskorderfol} implies that there is a bijection between $J^k\FolObj(W,M)$ and the set of hypersurface algebroids of $b^{k+1}$-type for $(W,M)$ (Definitions \ref{hypersurfacealgebroiddef} and \ref{hypersfalgbktype}). Since this bijection is equivariant with respect to the action of $\Diff_0(M, W)$, we may also identify $\IsoCat^{k}(W,M)$ with the groupoid of hypersurface algebroids of $b^{k+1}$-type for $(W,M)$ taken up to isotopy. Therefore, as a special case of Theorem \ref{RHuptoisotopy}, we obtain the following result which we will revisit in Section \ref{sec:character}. 
\begin{corollary}\label{cor:hypersurfclass}
Let $M$ be a manifold and let $W$ be a closed and connected hypersurface of codimension $1$. Then there is a homeomorphism between the space of isotopy classes of hypersurface algebroids of $b^{k+1}$-type for $(W,M)$ and the quotient space
\[
 \Hom_{(\Fr^k(\nu_{W}),\phi)}(\pi_{1}(W, x), G_{k,1})/G_{k,1}^0. 
\]
\end{corollary}
\begin{remark} \label{rem:bktangentbundleScott}
Recall that the $b^{k+1}$-tangent bundles defined by Scott \cite{MR3523250} required the choice of the $k$-jet of a defining function for $W$. Such a choice trivializes the normal bundle $\nu_{W}$ and implies that the holonomy representation of the associated $k$-th order foliation is trivial. Therefore, we conclude that all $b^{k+1}$-tangent bundles are isotopic and get sent to the trivial representation under the Riemann-Hilbert functor. 
\end{remark}

\section{Extensions: Part 2}\label{sec:extending2}

In this section we revisit the extension problem for $k$-th order foliations through the lens of the Riemann-Hilbert correspondence. We will show that the extension class of a $k$-th order foliation, given in Definition \ref{def:extension}, can also be viewed as a group cohomology class controlling the obstruction to a certain lifting problem. As we will see, the key to this reformulation lies in a combination of a Morita equivalence and the Van Est map from \cite{MR1103911}. 

\subsection{Recap on group extensions}
We consider group extensions of the form
\begin{equation*}
1 \rightarrow A \rightarrow B \rightarrow C \rightarrow 1,
\end{equation*}
with $A$ abelian. Given $\sigma : C \rightarrow B$ a set-theoretic splitting, we may define a $C$-module structure on $A$ as follows
\begin{equation*}
c\cdot a := \sigma(c)\cdot a\cdot \sigma(c)^{-1}. 
\end{equation*}
This structure does not depend on the choice of splitting as $A$ is abelian. 

Let $C^{\bullet}(C,A)$ denote the complex of group cochains valued in $A$. Recall that the groups in the complex are given by $C^{k}(C, A) = \mathrm{Hom}(C^{\times k}, A)$. Given a $k$-cochain $\alpha \in C^{k}(C,A)$, the differential is defined by 
\[
\delta(\alpha)(c_{1}, ..., c_{k+1}) = c_{1} \cdot \alpha(c_{2}, ..., c_{k+1}) + \sum_{i = 1}^{k} (-1)^{k} \alpha(c_{1}, ..., c_{i} c_{i+1}, ..., c_{k+1}) + (-1)^{k+1} \alpha(c_{1}, ..., c_{k}). 
\]

Given a splitting $\sigma : C \rightarrow B$, we define a $2$-cochain 
\[
\alpha_{\sigma}: C \times C \to A, \qquad (c_{1}, c_{2}) \mapsto \sigma(c_{1}) \sigma(c_{2}) \sigma(c_1 c_2)^{-1}.
\]
Then $\delta(\alpha_{\sigma}) = 0$ and there is an isomorphism between $B$ and $A \rtimes_{\sigma} C$, where the later is the set $A \times C$ equipped with the following product 
\[
(a_1, c_1) (a_2, c_2) = (a_1 (c_1 \cdot a_2) \alpha_{\sigma}(c_1,c_2), c_1 c_2). 
\]
The cohomology class $[\alpha_{\sigma}] \in H^2(C,A)$ is independent of the choice of splitting $\sigma$. It obstructs the existence of splittings which are also group homomorphisms. For these reasons, the cohomology class $E(B) := [\alpha_{\sigma}]$ is referred to as the \textbf{extension class} of the short exact sequence. 

The extension class is functorial in $C$. Indeed, if $\rho : D \to C$ is a homomorphism, there is an induced chain map $\rho^* : C^{\bullet}(C, A) \to C^{\bullet}(D, A)$ which allows us to pullback extension classes. We can also pullback the short exact sequence to get 
\[
1 \to A \to \rho^*(B) \to D \to 1. 
\]
Then $E(\rho^*(B)) = \rho^*(E(B)) \in H^{2}(D, A)$. This allows us to establish the following. 

\begin{lemma}\label{lemm:functorialextension}
Consider a short exact sequence of groups $1 \rightarrow A \rightarrow B \rightarrow C \rightarrow 1$ with extension class $E(B) \in H^2(C, A)$, and let $\rho : D \to C$ be a group homomorphism. Then the cohomology class $\rho^{*}E(B) \in H^{2}(D, A)$ vanishes if and only if the map $\rho$ lifts to a homomorphism $\rho: D \to B$. 
\end{lemma}

\subsection{The extension class revisited}
We again consider the group $G_{k,l} = J^k\Diff(\rr^l,0)$. Recall from Section \ref{structuregroup} that we have the natural homomorphism $p_k : G_{k,l} \rightarrow G_{k-1,l}$ which forgets the $k$-th order part. Let $\pi_k : G_{k,l} \rightarrow G_{1,l} = \mathrm{GL}(\rr^l)$ denote the map obtained by projecting all the way to the $1$-jet. 

Each $p_{k+1}$ induces a short exact sequence of groups
\begin{equation}\label{eq:groupseq}
0 \rightarrow A_{k+1,l} \rightarrow G_{k+1,l} \rightarrow G_{k,l} \rightarrow 0,
\end{equation}
with abelian kernel $A_{k+1, l}$. The following lemma determines the $G_{k,l}$-module structure on $A_{k+1, l}$.  
\begin{lemma} \label{modulestructure}
The group $A_{k+1, l}$ is isomorphic to the additive group underlying the vector space 
\[
\Sym^{k+1}((\rr^{l})^*) \otimes \rr^l.
\]
 Therefore, $A_{k+1, l}$ is naturally a $G_{1,l} = \mathrm{GL}(\rr^l)$-representation. The $G_{k,l}$-module structure on $A_{k+1,l}$ is the one induced by $\pi_{k}$. 
\end{lemma}
\begin{proof}
An element $\psi \in A_{k+1,l}$ is a $(k+1)$-jet of diffeomorphism whose $k$-jet is the identity. Consequently it has the form: 
\[
\psi(x) = (x_{i} + \sum_{|J| = k+1} \alpha_{i, J} x^{J} )_{i},
\]
where $J = (j_{1}, ..., j_{l})$ is a multi-index of weight $k +1$ and $x^{J} = x_1^{j_{1}} x_2^{j_{2}} ... x_l^{j_{l}}$. This provides the set-theoretic isomorphism to $\Sym^{k+1}(\rr^{l,*}) \otimes \rr^l$. We now check this is indeed a group isomorphism. Given another element 
\[
\phi(x) = (x_{i} + \sum_{|J| = k+1} \beta_{i, J} x^{J} )_{i},
\]
the composition is given by 
\begin{align*}
\psi \circ \phi(x) &= ( \phi(x)_{i} +  \sum_{|J| = k+1} \alpha_{i, J} (\phi(x))^{J})_{i} \\
&= (x_{i} + \sum_{|J| = k+1} \beta_{i, J} x^{J} +  \sum_{|J| = k+1} \alpha_{i, J} (x + \mathcal{O}(x^{k+1}))^{J})_{i} \\
&= (x_{i} + \sum_{|J| = k+1}( \alpha_{i,J} + \beta_{i, J}) x^{J})_{i}.
\end{align*}

Let $\psi \in A_{k+1, l}$, let $\phi \in G_{k+1, l}$, and let $\varphi = \pi_{k+1}(\phi) \in G_{1,l} \subset G_{k+1,l}$. Note that 
\[
\phi(x) = \varphi(x) + \gamma(x) =\varphi(x) + \mathcal{O}(x^2). 
\]
We also write $\psi(x) = (x_{i} + \sum_{|J| = k+1} \alpha_{i, J} x^{J} )_{i}$. We claim that $\psi \circ \phi = \psi \circ \varphi + \gamma$ and that $\phi \circ \psi = \varphi \circ \psi + \gamma$. 
First, 
\begin{align*}
\psi (\phi(x)) &=  (\phi_{i} + \sum_{|J| = k+1} \alpha_{i, J} \phi^{J} )_{i} \\
&= (\varphi_{i} + \gamma_{i} + \sum_{|J| = k+1 } \alpha_{i, J} (\varphi + \mathcal{O}(x^2))^J)_{i} \\
&= (\varphi_{i} + \sum_{|J| = k+1} \alpha_{i, J} \varphi^{J} )_{i} + \gamma \\
&= \psi (\varphi(x)) + \gamma(x). 
\end{align*}
Here, we've used the fact that $ (\varphi + \mathcal{O}(x^2))^J$ = $\varphi^J$, since terms of degree greater than $k+1$ are killed and $J$ has weight $k+1$. In a similar way, $\psi(x)_{i} \psi(x)_{j} = x_{i} x_{j}$. Hence $\gamma(\psi(x)) = \gamma(x)$ and so we see 
\[
\phi(\psi(x)) = \varphi(\psi(x)) + \gamma(\psi(x)) = \varphi(\psi(x)) + \gamma(x). 
\]
Let $C_{\phi}(\psi) = \phi \circ \psi \circ \phi^{-1}$ denote the conjugation action. We claim that $C_{\phi}(\psi) = C_{\varphi}(\psi)$. To see this, note first that we have $\phi \circ \psi = C_{\phi}(\psi) \circ \phi$, which we can expand as 
\[
\varphi \circ \psi + \gamma = C_{\phi}(\psi) \circ \varphi + \gamma. 
\]
Canceling out $\gamma$, we get $C_{\phi}(\psi) \circ \varphi = \varphi \circ \psi$, or $C_{\phi}(\psi) = \varphi \circ \psi \circ \varphi^{-1} = C_{\varphi}(\psi)$.
\end{proof}

As a result of Lemma \ref{modulestructure}, the extension class of $G_{k+1, l}$ can be viewed as 
\[
E_{k+1, l} := E(G_{k+1,l}) \in H^2(G_{k,l}, \Sym^{k+1}((\rr^{l})^*) \otimes \rr^l). 
\]
Given a representation $\rho : \pi_{1}(W, x) \to G_{k,l}$, we can try to lift it to a representation 
\[
\tilde{\rho} : \pi_{1}(W,x) \to G_{k+1,l},
\]
such that $p_{k+1} \circ \tilde{\rho} = \rho$. By Lemma \ref{lemm:functorialextension}, this is obstructed by the extension class $\rho^*(E_{k+1,l})$.  Since representations correspond to $k$-th order foliations by Theorem \ref{RHthm}, we obtain the following alternate characterization of the extension problem. 
\begin{theorem}\label{th:secondextensionclass}
Let $W \subset M$ be a codimension $l$ submanifold equipped with a basepoint $x \in W$. Let $F$ be a $k$-th order foliation around $W$ and let $\rho: \pi_{1}(W,x) \to G_{k,l}$ be the corresponding representation obtained by Theorem \ref{RHthm}. Then $F$ can be extended to a $(k+1)$-st order foliation if and only if the coholomogy class 
\begin{equation*}
\rho^*E_{k+1,l} \in H^2(\pi_1(W,x),\Sym^{k+1}((\rr^{l})^*) \otimes \rr^l)
\end{equation*}
vanishes.
\end{theorem}

Note that there is an explicit expression for the extension class $E_{k+1, l}$ induced by the canonical splitting $s_{k}$ of $p_{k+1} : G_{k+1, l} \to G_{k,l}$ given in Corollary \ref{canonicalsplitting}. This mirrors the canonical expression for the extension class $e(F)$ obtained in Corollary \ref{cor:explicitextensionclass}. 

\begin{example}[$E_{3,l}$]
Using the splitting $s_2$ along with the description $G_{2,l} \cong \big( \mathrm{Sym}^{2}( (\mathbb{R}^l)^* ) \otimes \mathbb{R}^l \big) \rtimes \mathrm{GL}(\mathbb{R}^l)$ provided in Corollary \ref{cor:g2lgroup}, we give an explicit cocycle representing $E_{3,l}$. It is the map 
\[
\alpha_{s_2} : G_{2,l} \times G_{2,l} \to \mathrm{Sym}^{3}( (\mathbb{R}^l)^* ) \otimes \mathbb{R}^l, \qquad ( (k_1, A_1), (k_2, A_2)) \mapsto \frac{1}{2}[k_1, A_{1}(k_2)],
\]
where we use the bracket on $\mathrm{Lie}(K_{3,l})$ which is described in Lemma \ref{Explicitbracket}. 
\end{example}
\begin{example}[$E_{4,1}$] \label{Ex:explicitext41}
Using the splitting $s_3$ along with the description $G_{3,1} \cong \mathbb{R}^2 \rtimes \mathbb{R}^*$ provided in Corollary \ref{cor:g31group}, we give an explicit cocycle representing $E_{4,1}$. It is the map 
\[
\alpha_{s_{3}}: G_{3,1} \times G_{3,1} \to \mathbb{R}, \qquad ((a_{1}, a_{2}, a_{0}), (b_{1}, b_{2}, b_{0})) \mapsto \frac{1}{2}\big(a_1 a_0^{-1}b_2 - a_0^{-1}b_{1}a_{2} \big). \hfill \qedhere
\]
\end{example}

\subsection{Groupoid extensions} \label{ssec:groupoidext}
Let $E \to W$ be a rank $l$ vector bundle, let $x \in W$ be a basepoint and let $g \in \Fr^k(E_{x})$ be a linear framing. Given a $k$-th order foliation $F$ around $W$, there is a corresponding representation $\rho: \pi_{1}(W, x) \to G_{k,l}$ by Theorem \ref{RHthm}. 

Recall that the Riemann-Hilbert correspondence map is obtained via the holonomy of a foliation defined in Section \ref{holonomydefsection}. Indeed, the $k$-th order foliation $F$ defines a splitting $\sigma: TW \to \at^{k}(E)$, which is integrated to a Lie groupoid homomorphism $\hol^{\sigma} : \Pi(W) \to \At^{k}(E)$. Then $\rho = \hol_{x}^{\sigma, g}$, the restriction of the holonomy to $\pi_{1}(W,x)$, using the framing $g$ to identify $\At^k(E)|_{x} \cong G_{k,l}$.

We have produced two cohomology classes controlling the extension problem for $F$: 
\begin{enumerate}
\item the extension class $e(F) \in H^{2}(W, \mathrm{Sym}^{k+1}(E^*) \otimes E)$ given in Definition \ref{def:extension}, and 
\item the extension class $\rho^*E_{k+1,l} \in H^2(\pi_1(W,x),\Sym^{k+1}((\rr^{l})^*) \otimes \rr^l)$ defined in Theorem \ref{th:secondextensionclass}.
\end{enumerate}
In this section, we will compare these two classes by relating each to an intermediate groupoid cohomology class which is constructed using $\hol^{\sigma}$. 

First, using the description of $\A(\sigma)|_W$ in Proposition \ref{restrictedextendedatiyah} we immediately obtain:
\begin{lemma}
The source simply connected integration of the Lie algebroid $\A(\sigma)|_W$ is given by 
\[ \Gcal := \At^{k+1}(E) \times_{\hol^{\sigma}} \Pi(W). \]
\end{lemma}
Therefore, we obtain a short exact sequence of groupoids 
\begin{equation}\label{eq:groupoidseq}
1 \to \Sym^{k+1}(E^*)\otimes E \to \Gcal \to \Pi(W) \to 1. 
\end{equation}
which integrates the short exact Sequence \ref{restrictedanchorsequence} used to define $e(F)$. To obtain a splitting of Sequence \ref{eq:groupoidseq}, we have the following lemma. 
\begin{lemma}
The morphism of groupoids $\At^{k+1}(E) \to \At^{k}(E)$ has a canonical splitting $S_k$. Given a linear framing of $E_{x}$, the splitting $S_k$ restricts to the splitting $s_k$ of Corollary \ref{canonicalsplitting}.  
\end{lemma}
\begin{proof}
This follows from the canonical isomorphism  
\[
\At^{k}(E) \cong (\Fr^1(E) \times G_{k,l} \times \Fr^1(E) )/(G_{1,l} \times G_{1,l}),
\]
where the group action used to define the quotient is given by 
\[
(\phi_1, g, \phi_2) \cdot (h_1, h_2) = (\phi_1 \cdot h_1, h_{1}^{-1} g h_{2}, \phi_2 \cdot h_2). \hfill \qedhere
\]  
\end{proof}
Define the following splitting of Sequence \ref{eq:groupoidseq} 
\[
S_{\sigma}: \Pi(W) \to \mathcal{G}, \qquad \gamma \mapsto (S_{k} \circ \hol^{\sigma}(\gamma), \gamma). 
\]
This induces an action of $\Pi(W)$ on $\Sym^{k+1}(E^*)\otimes E$ and a groupoid cocycle 
\[
\alpha_{\sigma} : \Pi(W) \times_{W} \Pi(W) \to \Sym^{k+1}(E^*)\otimes E, \qquad (\gamma_1, \gamma_2) \mapsto S_{\sigma}(\gamma_{1}) S_{\sigma}(\gamma_2) S_{\sigma}(\gamma_1 \gamma_2)^{-1}. 
\]
Let $[\alpha_{\sigma}] \in H^{2}(\Pi(W), \Sym^{k+1}(E^*)\otimes E)$. The $\Pi(W)$-action on $\Sym^{k+1}(E^*)\otimes E$ coincides with the representation induced by the linear flat connection on $E$ underlying $F$. 

\begin{enumerate}
\item The Van Est map from \cite{MR1103911} defines a linear map 
\[
\mathrm{VE}: H^2(\Pi(W);\Sym^{k+1}(E^*)\otimes E) \rightarrow H^2(W;\Sym^{k+1}(E^*)\otimes E).
\]
Because the Sequence \ref{eq:groupoidseq} is an integration of Sequence \ref{restrictedanchorsequence}, by the proof of Theorem 5 in \cite{cra03}, we have that \[ \mathrm{VE}([\alpha_{\sigma}] ) = e(F). \]
\item Now consider the inclusion $\iota_x : \pi_{1}(W, x) \to \Pi(W)$. Because $\Pi(W)$ is transitive, this is a Morita equivalence. Therefore, the pullback $\iota_{x}^*$ defines an isomorphism of cohomology groups: 
\[
\iota_{x} : H^{2}(\Pi(W), \Sym^{k+1}(E^*)\otimes E) \to H^{2}(\pi_{1}(W, x), \Sym^{k+1}((\rr^{l})^*) \otimes \rr^l). 
\]
It is clear that $[\alpha_{\sigma}] $ gets sent to $\rho^*E_{k+1,l} $ under this map since the splitting $\alpha_{\sigma}$ restricts to the splitting $\rho^*(s_{k})$. 
\end{enumerate}

Combining these two observations, we arrive at the following comparison result. 

\begin{proposition} \label{ComparisonmapVE}
Let $E \to W$ be a vector bundle and let $F$ be a $k$-th order foliation around $W$. Given a basepoint $x \in W$ and a linear framing $g \in \Fr^k(E)$, let $\rho : \pi_{1}(W, x) \to G_{k,l}$ be the associated representation given by Theorem \ref{RHthm}. Then there is a canonically induced homomorphism of cohomology groups 
\[
\mathrm{VE} \circ  \iota_{x}^{-1} : H^2(\pi_1(W,x),\Sym^{k+1}((\rr^{l})^*) \otimes \rr^l) \to H^{2}(W, \mathrm{Sym}^{k+1}(E^*) \otimes E)
\]
induced by the Van Est map and the Morita equivalence between $\Pi(W)$ and $\pi_{1}(W,x)$. Under this map 
\[
\mathrm{VE} \circ  \iota_{x}^{-1} (\rho^*E_{k+1,l} ) = e(F). 
\]
\end{proposition}
An immediate corollary of the fact that $e(F)$ is in the image of the Van-Est map is the following. 
\begin{corollary} \label{cor:asphericalperiods}
The extension class $e(F) \in H^2(W;\Sym^{k+1}(\nu_W^*) \otimes \nu_W)$ is aspherical, that is:
\begin{equation*}
\int_{S^2} e(F) = 0,
\end{equation*}
for any map from $S^2$ into $W$.
\end{corollary}
Note that this also follows from Proposition \ref{prop:simplyconnected}. Indeed, as $S^2$ is simply connected and the extension class is functorial we find that the extension class vanishes when pulled back to any $S^2$.
\section{The $k$-jet character stack} \label{sec:character}
In this final section, we synthesize our results and study the character varieties classifying $k$-th order foliations. Given a closed and connected manifold $W$ with a chosen basepoint $x \in W$, we define 
\[
\mathcal{M}_{k,l}(W) := G_{k,l} \ltimes \Hom(\pi_1(W, x), G_{k,l})
\]
to be the $k$-jet character groupoid of $W$. The choice of a rank $l$ vector bundle $E \to W$ with framing $\phi : G_{k,l} \to \Fr^k(E|_{x})$ determines a full subgroupoid 
\[
\mathcal{M}_{k}(W,E) := G_{k,l}(\phi) \ltimes \Hom_{(\Fr^k(E),\phi)}(\pi_1(W,x),G_{k,l}). 
\]
Theorem \ref{topologicalRHfoliation} states that there is a fibration from the groupoid of $k$-th order foliations in $E$ along $W$ to $\mathcal{M}_{k}(W,E)$. As a result, the quotient space 
\[
M_{k}(W,E) := \Hom_{(\Fr^k(E),\phi)}(\pi_1(W,x),G_{k,l})/G_{k,l}(\phi)
\]
is homeomorphic to the space of isomorphism classes of $k$-th order foliations. 

Recall from Remark \ref{Gphiconj} that, since $G_{k,l}$ has two connected components, the open subgroup $G_{k,l}(\phi)$ is completely determined by the vector bundle $E$. We can determine this subgroup in certain cases. 

\begin{lemma}
Let $E \rightarrow W$ be a rank $l$ vector bundle. The subgroup $G_{k,l}(\phi)$ is either the connected component of the identity $G_{k,l}^{0}$ or the entire group $G_{k,l}$. If the rank $l$ of $E$ is odd, then $G_{k,l}(\phi) = G_{k,l}$. If the rank $l = 2$, then $G_{k,2}(\phi) = G_{k,2}$ if and only if $E$ admits a line subbundle. 
\end{lemma}
\begin{proof}
By Lemma \ref{lem:gaugehomotopy}, there is homotopy between $\mathrm{Gau}(\Fr^k(E))$ and $\mathrm{Gau}(\Fr^1(E))$ which respects the restriction to a fibre. As a result, $G_{k,l}(\phi)$ and $G_{1,l}(\phi)$ are homotopic. The later group was determined in Corollaries \ref{oddn} and \ref{evenn}. 
\end{proof}

Let $M$ be a manifold containing $W$, let $\Phi : \nu_{W} \to M$ be a tubular neighbourhood embedding, and let $\phi : G_{k,l} \to \Fr^k(\nu_{W}|_{x})$ be a framing. This determines the groupoid 
\[
\mathcal{M}_{k}^{0}(W,M) := G_{k,l}^0 \ltimes \IsotopyObj(\Fr^k(\nu_{W}),\phi).
\]
Theorem \ref{RHuptoisotopy} states that there is a fibration from the groupoid of $k$-th order foliations in $M$ along $W$, taken up to isotopy, to $\mathcal{M}_{k}^{0}(W,M)$. As a result, the quotient space 
\[
M_{k}^{0}(W,M) := \IsotopyObj(\Fr^k(\nu_{W}),\phi)/G_{k,l}^{0}
\]
is homeomorphic to the space of isotopy classes of $k$-th order foliations in $M$ along $W$. By Corollary \ref{cor:constructingH}, the space $\IsotopyObj(\Fr^k(\nu_{W}),\phi)$ is a principal $D(\Fr^k(\nu_{W}))$-bundle over $\Hom_{(\Fr^k(E),\phi)}(\pi_1(W,x),G_{k,l})$, where 
\[ D(\Fr^k(\nu_{W})) = \ker(\pi_0(\Gau(\Fr(\nu_{W}))) \rightarrow \pi_0(\mathrm{GL}(\mathbb{R}^l))), \]
and where we have used Lemma \ref{lem:gaugehomotopy} to identify $ D(\Fr^k(\nu_{W})) \cong  D(\Fr(\nu_{W}))$. This group can also be determined in certain cases by the results of Sections \ref{sec:linebundles} and \ref{sec:2bundles}. 
\begin{lemma} \label{lem:Dprincbundle}
When $W \subset M$ has codimension $1$, $D(\Fr^k(\nu_{W}))$ is trivial. When $W \subset M$ has codimension $2$, 
\[
D(\Fr^k(\nu_{W})) \cong H^{1}(W, \mathbb{Z}_{\nu}),
\]
where $\mathbb{Z}_{\nu}$ is the local coefficient system associated to the orientation bundle of $\nu_{W}$. 
\end{lemma}

\subsubsection{Morphisms between $k$-jet character varieties}
Several properties of the $k$-jet character varieties follow readily from the properties of the structure groups $G_{k,l}$ which were determined in Section \ref{structuregroup}. The homomorphism $\pi_{k} : G_{k,l} \to G_{1,l} = \mathrm{GL}(\mathbb{R}^l)$ determines a groupoid homomorphism 
\[
\pi_{k,*} : \mathcal{M}_{k,l}(W) \to \mathcal{M}_{1,l}(W)
\]
such that $\mathcal{M}_{k}(W, E) = \pi_{k,*}^{-1}(\mathcal{M}_{1}(W, E))$. The splitting $j_{k}: G_{1,l} \to G_{k,l}$ of $\pi_{k}$ induces a groupoid homomorphism 
\[
j_{k, *} : \mathcal{M}_{1,l}(W) \to  \mathcal{M}_{k,l}(W) 
\]
satisfying $\pi_{k,*} \circ j_{k, *} = id_{ \mathcal{M}_{1,l}(W) }$. By Corollary \ref{cor:equivarianthomotopy}, there is a homotopy $h_{t}: G_{k,l} \to G_{k,l}$ between $id_{G_{k,l}}$ and $j_{k} \circ \pi_{k}$ through group homomorphisms. This induces a homotopy 
\[
H: [0,1] \times \mathcal{M}_{k,l}(W) \to \mathcal{M}_{k,l}(W) 
\] 
through groupoid homomorphisms between  $id_{ \mathcal{M}_{k,l}(W) }$ and $j_{k, *} \circ \pi_{k,*} $. Since $\pi_{k,*}, j_{k, *}$ and $H_t$ are groupoid homomorphisms, they induce continuous maps between the orbit spaces that we denote with the same notation. Therefore, we obtain the following. 

\begin{lemma} \label{lem:orbitcl}
The groupoids $\mathcal{M}_{1,l}(W)$ and $\mathcal{M}_{k,l}(W)$ are homotopy equivalent through groupoid homomorphisms. As a result, the topological spaces $M_{1,l}(W)$ and $M_{k,l}(W)$ are homotopy equivalent. Furthermore, given each point $x \in M_{1,l}(W)$, the point $j_{k,*}(x)$ lies in the closure of every point in $\pi_{k,*}^{-1}(x)$.
\end{lemma}
\begin{proof}
It remains only to show that $j_{k,*}(x)$ lies in the closure of every point of $\pi_{k,*}^{-1}(x)$. Let $y \in \pi_{k,*}^{-1}(x)$ and let $\varphi \in \Hom(\pi_1(W, x), G_{k,l})$ be a representative. Then the path of equivalence classes $[h_{t} \circ \varphi]$ goes between $[h_{0} \circ \varphi] = j_{k,*}(x)$ and $[h_1 \circ \varphi] = y$. But for $t > 0$, the map $h_t$ is conjugation by $j_{k}(t \ id)$. Hence $[h_{t} \circ \varphi] = [C_{j_{k}(t \ id)} \circ \varphi] = [\varphi] = y$. In other words, $[h_{t} \circ \varphi]$ is the constant path at $y$ for $t \neq 0$.
\end{proof}
\begin{remark}
Lemma \ref{lem:orbitcl} implies that $M_{k,l}(W)$ is non-Hausdorff whenever the pre-images $\pi_{k,*}^{-1}(x)$ contain more than one point. To remedy this, it is often useful to consider the subspace $M_{k,l}(W)^{\times} = M_{k,l}(W) \setminus j_{k,*}(M_{1,l}(W))$.
\end{remark}

Finally, consider the sequence of morphisms
\begin{equation*}
\cdots  \overset{p_{k+2}}{\rightarrow}  G_{k+1,l}  \overset{p_{k+1}}{\rightarrow}  G_{k,l}  \overset{p_{k}}{\rightarrow}  G_{k-1,l}  \overset{p_{k-1}}{\rightarrow} \cdots,
\end{equation*}
and the induced maps between the character groupoids
\begin{equation*}
\cdots \overset{p_{k+2,*}}{\rightarrow} \mathcal{M}_{k+1,l}(W) \overset{p_{k+1,*}}{\rightarrow}\mathcal{M}_{k,l}(W) \overset{p_{k,*}}{\rightarrow} \mathcal{M}_{k-1,l}(W) \overset{p_{k-1,*}}{\rightarrow} \cdots.
\end{equation*}
These maps correspond to sending a $(k+1)$-st order foliation to the underlying $k$-th order foliation. As a result, the image of each map consists of those $k$-th order foliations which can be extended to a $(k+1)$-st order foliation. Consequently, we have the following result. 
\begin{corollary}\label{cor:extimage}
The image $p_{k+1,*}(\mathcal{M}_{k+1}(W,E))$ coincides with those elements in $\mathcal{M}_{k}(W,E)$ with zero-extension class.
\end{corollary}

\subsubsection{Extension class}
We now show that the extension class studied in Sections \ref{sec:extending} and \ref{sec:extending2} arises from a section of a vector bundle over the $k$-jet character stack. To avoid technicalities, we fix a flat connection $(E, \nabla)$ and restrict our attention to $k$-th order foliations extending $(E, \nabla)$. Equivalently, we fix a representation  
\[
\rho : \pi_{1}(W, x) \to G_{1, l}.
\]
Let $G_1(\rho) \subseteq G_{1, l}$ denote the stabilizer of $\rho$. Define $G_{k}(\rho) = \pi_{k}^{-1}(G_{1}(\rho)) \subseteq G_{k,l}$ and $\mathrm{Hom}_{\rho}(\pi_{1}(W,x), G_{k,l}) = \pi_{k,*}^{-1}(\rho) \subseteq \Hom(\pi_1(W, x), G_{k,l}) $. Then 
\[
\mathcal{M}_{k}(\rho) := \pi_{k, *}^{-1}(G_{1}(\rho) \ltimes \{ \rho \}) = G_{k}(\rho) \ltimes \mathrm{Hom}_{\rho}(\pi_{1}(W,x), G_{k,l})
\]
is the full subcategory of the $k$-jet character groupoid consisting of representations with fixed $1$-jet given by $\rho$. 

Let $\varphi \in \mathrm{Hom}_{\rho}(\pi_{1}(W,x), G_{k,l})$. By Theorem \ref{th:secondextensionclass}, the problem of lifting $\varphi$ to a homomorphism $\tilde{\varphi}: \pi_{1}(W,x) \to G_{k+1,l}$ is controlled by the extension class 
\[
\varphi^*(E_{k+1,l}) \in H_{\rho}^2(\pi_1(W,x),A_{k+1,l}),
\]
where $A_{k+1,l} = \Sym^{k+1}((\rr^{l})^*) \otimes \rr^l$. By Lemma \ref{eq:groupseq}, the $\pi_{1}(W,x)$-module structure on $A_{k+1,l} $ is the one induced from its structure as a $\mathrm{GL}(\mathbb{R}^l)$-representation by pulling back via the map $\pi_{k} \circ \varphi = \rho$. Hence, this module structure is constant over the space $\mathrm{Hom}_{\rho}(\pi_{1}(W,x), G_{k,l})$, and we may consider the trivial vector bundle 
\[
H^2_{\rho,k} := H_{\rho}^2(\pi_1(W,x),A_{k+1,l}) \times \mathrm{Hom}_{\rho}(\pi_{1}(W,x), G_{k,l}) \to \mathrm{Hom}_{\rho}(\pi_{1}(W,x), G_{k,l}). 
\]
It is for this reason that it is convenient to fix the representation $\rho$. Since $A_{k+1,l}$ is a representation of $G_{k,l}$ via the map $\pi_{k}$, the cohomology group $H_{\rho}^2(\pi_1(W,x),A_{k+1,l})$ is naturally a $G_{k}(\rho)$-representation. As a result, $H^2_{\rho,k}$ can be given the structure of a $G_{k}(\rho)$-equivariant vector bundle over $\mathrm{Hom}_{\rho}(\pi_{1}(W,x), G_{k,l})$. In other words, it is a representation of the groupoid $\mathcal{M}_{k}(\rho)$. 

By using the canonical splitting $s_{k}$ from Corollary \ref{canonicalsplitting}, we can write down an explicit representative for the extension class $\varphi^*(E_{k+1,l}) $: 
\[
e_{k}(\varphi) : \pi_{1}(W, x) \times \pi_{1}(W,x) \to A_{k+1, l}, \qquad (\gamma_1, \gamma_2) \mapsto s_{k}(\varphi(\gamma_1))s_{k}(\varphi(\gamma_2))s_{k}(\varphi(\gamma_1 \gamma_2))^{-1}.
\]
This allows us to view the extension class as a continuous section $e_{k} : \mathrm{Hom}_{\rho}(\pi_{1}(W,x), G_{k,l}) \to H^2_{\rho,k}$. 

\begin{proposition} \label{prop:extsection}
The section $e_{k}$ is invariant under the $G_{k}(\rho)$-action, in the sense that  
\[
e_{k}(C_{g} \circ \varphi) = \pi_k(g) \ast e_{k}(\varphi),
\]
for $g \in G_{k}(\rho)$ and $\varphi \in \mathrm{Hom}_{\rho}(\pi_{1}(W,x), G_{k,l})$. As a result, $e_{k}$ descends to a section of the vector bundle $H^2_{\rho, k}/G_{k}(\rho)$ over the character stack determined by $\mathcal{M}_{k}(\rho)$. 
\end{proposition}
\begin{proof}
Let $s_{k} : G_{k,l} \to G_{k+1, l}$ be the splitting from Corollary \ref{canonicalsplitting}. It gives rise to the cocycle
\[
\alpha_{s_{k}} : G_{k,l} \times G_{k,l} \to A_{k+1,l}, \qquad (g_1, g_2) \mapsto s_{k}(g_1)s_{k}(g_2) s_{k}(g_1 g_2)^{-1}
\]
representing the extension class $E_{k+1, l}$. The section $e_k$ is given by $e_k(\varphi) = \alpha_{s_{k}} \circ \varphi$. Now let $g \in G_{k}(\rho)$, let $\tilde{g} = s_{k}(g) \in G_{k+1}(\rho)$ be a lift, and define 
\[
s_{k}^{\tilde{g}} := C_{\tilde{g}^{-1}} \circ s_{k} \circ C_{g} : G_{k,l} \to G_{k+1, l}, \qquad h \mapsto \tilde{g}^{-1}s_{k}(g h g^{-1}) \tilde{g}. 
\]
It is straightforward to check that $s_{k}^{\tilde{g}}$ also defines a section of $p_{k+1}$ and hence a cocycle $\alpha_{s^{\tilde{g}}_{k}}$ representing $E_{k+1,l}$.  A calculation shows that 
\[
\alpha_{s^{\tilde{g}}_{k}} = C_{\tilde{g}^{-1}} \circ \alpha_{s_{k}}  \circ C_{g} = \pi_{k}(g^{-1}) \ast \alpha_{s_{k}}  \circ C_{g},
\]
where the last equality follows from the fact that $\alpha_{s_{k}}$ is valued in $A_{k+1,l}$, where the conjugation action of $G_{k+1, l}$ factors through the projection to $G_{1,l}$. Hence, $\alpha_{s^{\tilde{g}}_{k}} \circ \varphi = \pi_{k}(g^{-1}) \ast e_k(C_{g} \circ \varphi)$. Since $\alpha_{s_{k}}$ and $\alpha_{s^{\tilde{g}}_{k}} $ both represent $E_{k+1,l}$, they are cohomologous. Indeed, $\alpha_{s^{\tilde{g}}_{k}} = \alpha_{s_{k}} + \delta(u^{\tilde{g}})$, where $u^{\tilde{g}} : G_{k,l} \to A_{k+1,l}$ is the map satisfying $s_{k}^{\tilde{g}}(g') = u^{\tilde{g}}(g')s_{k}(g')$. Therefore, $\pi_{k}(g^{-1}) \ast e_k(C_{g} \circ \varphi) = e(\varphi)$.
\end{proof}
\begin{remark}
By Corollary \ref{cor:extimage}, the image of the map $p_{k+1, *}|_{\mathcal{M}_{k+1}(\rho)} : \mathcal{M}_{k+1}(\rho) \to \mathcal{M}_{k}(\rho)$ coincides with the zero locus of the section $e_k$ from Proposition \ref{prop:extsection}. 
\end{remark}

\subsubsection{Extension spaces}
Now we will study the space of $(k+1)$-st order foliations extending a given $k$-th order foliation $F$ on the framed vector bundle $(E \to W,\phi)$. We continue to let $\rho : \pi_{1}(W,x) \to \mathrm{GL}(\mathbb{R}^l)$ denote the representation corresponding to the flat connection $\nabla$ on $E$ induced by $F$.

The $k$-th order foliation $F$ corresponds, via Theorem \ref{topologicalRHfoliation}, to a representation $\varphi \in \mathrm{Hom}_{\rho}(\pi_{1}(W,x), G_{k,l})$. We assume that it has vanishing extension class, meaning that it lies in the zero locus of the section $e_{k}$. Consider the full subgroupoid over this point, $G_{k}(\varphi) \ltimes \{ \varphi \} \subseteq \mathcal{M}_{k}(\rho)$, where $G_{k}(\varphi) \subseteq G_{k,l}$ is the stabilizer subgroup of $\varphi$. The preimage 
\[
\mathcal{M}_{k+1}(\varphi) := p_{k+1, *}^{-1}(G_{k}(\varphi) \ltimes \{ \varphi \}) = G_{k+1}(\varphi) \ltimes \Hom_{\varphi}(\pi_1(W, x), G_{k+1,l})
\]
is a full subgroupoid of $\mathcal{M}_{k+1}(\rho)$, where $G_{k+1}(\varphi) = p_{k+1}^{-1}(G_{k}(\varphi))$ is an abelian extension of $G_{k}(\varphi)$ by $A_{k+1,l}$, and $\Hom_{\varphi}(\pi_1(W, x), G_{k+1,l})$ is the set of all extensions of $\varphi$, which is non-empty because we assume $e_k(\varphi) = 0$. The base $\Hom_{\varphi}(\pi_1(W, x), G_{k+1,l})$ is a torsor for the group of $1$-cocycles 
\[
Z_{\rho}^{1}(\pi_{1}(W,x), A_{k+1, l}) = \{ a : \pi_{1}(W, x) \to A_{k+1,l} \ | \ a(\gamma_1 \gamma_2) = a(\gamma_1) + \rho(\gamma_1) \ast a(\gamma_2) \}. 
\]
Fix a reference extension $\psi_0$ in order to identify $\Hom_{\varphi}(\pi_1(W, x), G_{k+1,l}) \cong Z_{\rho}^{1}(\pi_{1}(W,x), A_{k+1, l})$. Then the subgroup $A_{k+1, l} \subseteq G_{k+1}(\varphi)$ acts via the differential of group cohomology as follows 
\[
C_{g}(a) = a - \delta(g),
\]
where $g \in A_{k+1, l}$ and $a \in Z_{\rho}^{1}(\pi_{1}(W,x), A_{k+1, l})$. Therefore, the stabilizer of a point is the cohomology group $H_{\rho}^0(\pi_{1}(W,x), A_{k+1,l})$ and the quotient $\Hom_{\varphi}(\pi_1(W, x), G_{k+1,l})/A_{k+1,l}$ is a torsor for the cohomology group $H_{\rho}^1(\pi_{1}(W,x), A_{k+1,l})$. There is a residual $G_{k}(\varphi)$ action on $\Hom_{\varphi}(\pi_1(W, x), G_{k+1,l})/A_{k+1,l}$. Identifying the later space with $H_{\rho}^1(\pi_{1}(W,x), A_{k+1,l})$, the action of $G_{k}(\varphi)$ is affine-linear. Therefore, the orbit space of $\mathcal{M}_{k+1}(\varphi)$ is given by 
\[
M_{k+1}(\varphi) = \Hom_{\varphi}(\pi_1(W, x), G_{k+1,l})/G_{k+1}(\varphi) \cong H_{\rho}^1(\pi_{1}(W,x), A_{k+1,l})/G_{k}(\varphi). 
\]
This space is homeomorphic to the space of isomorphism classes of $(k+1)$-st order foliations extending $F$. 

Using the van Est map and Morita equivalence, the cohomology group $H_{\rho}^1(\pi_{1}(W,x), A_{k+1,l})$ can be identified with the de Rham cohomology group 
\[
H_{\nabla}^1(W, \mathrm{Sym}^{k+1}(E^*)\otimes E),
\]
where the flat connection on $\mathrm{Sym}^{k+1}(E^*)\otimes E$ is the one induced by $\nabla$ on $E$. 

Now consider the problem of extending representations in $ \Hom_{\varphi}(\pi_1(W, x), G_{k+1,l})$. This is controlled by the extension class $e_{k+1} : \Hom_{\varphi}(\pi_1(W, x), G_{k+1,l}) \to H^2_{\rho,k+1}$. This is equivariant with respect to the $G_{k+1}(\varphi)$-action by Proposition \ref{prop:extsection}. But since $A_{k+1, l} \subseteq G_{k+1}(\varphi)$ lies in the kernel of $\pi_{k}$, we see that
\[
e_{k+1}(C_{a} \circ \psi) = e_{k+1}(\psi),
\]
for $a \in A_{k+1,l}$ and $\psi \in \Hom_{\varphi}(\pi_1(W, x), G_{k+1,l})$. Therefore, $e_{k+1}$ descends to a $G_{k}(\varphi)$-equivariant map 
\[
e_{k+1} : H_{\rho}^1(\pi_{1}(W,x), A_{k+1,l}) \to H_{\rho}^2(\pi_1(W,x),A_{k+2,l}).
\]
By Proposition \ref{ComparisonmapVE}, we may identify this with the extension class from Definition \ref{def:extension}. We summarize the observations of this section in the following theorem. 

\begin{theorem} \label{thm:extensionspacegroupoid}
Let $(E \to W, \phi)$ be a linearly framed vector bundle equipped with a $k$-th order foliation $F$ along $W$ with vanishing extension class. Let $\nabla$ be the  flat connection on $E$ underlying $F$. By Theorem \ref{splittingprop}, $F$ corresponds to a flat connection $\sigma$ on the $k$-th order frame bundle $\mathrm{Fr}^{k}(E)$. Let $G_{k}(\varphi)$ denote the subgroup of gauge transformations which preserve $\sigma$. Choosing a reference extension of $F$ determines an affine-linear action of $G_{k}(\varphi)$ on
\[
H_{\nabla}^1(W, \mathrm{Sym}^{k+1}(E^*)\otimes E),
\]
the de Rham cohomology group with coefficients in $\mathrm{Sym}^{k+1}(E^*)\otimes E$ equipped with the flat connection induced by $\nabla$. There is a fibration of topological groupoids from the groupoid of $(k+1)$-st order foliations in $E$ along $W$ extending $F$ to the action groupoid 
\[
G_{k}(\varphi) \ltimes H_{\nabla}^1(W, \mathrm{Sym}^{k+1}(E^*)\otimes E). 
\]
Therefore, the space of isomorphism classes of $(k+1)$-st order foliations extending $F$ is homeomorphic to the topological space 
\[
H_{\nabla}^1(W, \mathrm{Sym}^{k+1}(E^*)\otimes E)/G_{k}(\varphi). 
\]
Furthermore, the extension class for $(k+1)$-st order foliations defines a $G_{k}(\varphi)$-equivariant map 
\[
e_{k+1} : H_{\nabla}^1(W, \mathrm{Sym}^{k+1}(E^*)\otimes E) \to H_{\nabla}^2(W, \mathrm{Sym}^{k+2}(E^*)\otimes E) 
\]
corresponding to a section of a vector bundle over the quotient stack. Its zero locus consists of those $(k+1)$-st order foliations that can be extended to $(k+2)$-nd order foliations.
\end{theorem}

\subsubsection{The $2$-jet character variety}
Restricting to the case $k = 2$, there are several simplifications. We have the isomorphism 
\[
G_{2,l} \cong A_{2,l} \rtimes \mathrm{GL}(\mathbb{R}^l)
\]
from Corollary \ref{cor:g2lgroup}. The splitting $s_{1} = j_{2}: G_{1,l} \to G_{2,l}$ is a homomorphism, implying that the extension class $E_{2,l}$ vanishes. In other words, every $1$-st order foliation extends to a $2$-nd order foliation, and there is a canonical choice given by the linear extension corresponding to the map $j_2$. Consequently, the $2$-jet character groupoid $\mathcal{M}_{2,l}(W)$ can be given a simple description.  

Fixing a representation $\rho : \pi_{1}(W, x) \to G_{1, l}$ we have the isomorphism 
\[
G_{2}(\rho) \cong A_{2,l} \rtimes G_{1}(\rho).
\]
Using $j_{2} \circ \rho : \pi_{1}(W, x) \to G_{2,l}$ as a reference connection, we have 
\[
 \mathrm{Hom}_{\rho}(\pi_{1}(W,x), G_{2,l}) \cong Z_{\rho}^{1}(\pi_{1}(W,x), A_{2, l}). 
\]
Therefore, 
\[
\mathcal{M}_{2}(\rho) \cong (A_{2,l} \rtimes G_{1}(\rho)) \ltimes Z_{\rho}^{1}(\pi_{1}(W,x), A_{2, l}),
\]
with the action given by 
\[
(k, g) \ast a = g \ast a - \delta(k). 
\]
Note that the induced action of $G_{1}(\rho)$ on $H_{\rho}^{1}(\pi_{1}(W,x), A_{2, l})$ is linear. 
\subsection{The variety of hypersurface algebroids}
In this section we focus on the case of codimension $l = 1$, where the $k$-jet character groupoids classify hypersurface algebroids. Note first that $G_{1,1} = \mathbb{R}^{*} \cong \mathbb{Z}/2 \times \mathbb{R}$. Therefore, 
\[
\Hom(\pi_1(W, x), G_{1,1}) \cong H^{1}(W, \mathbb{Z}/2) \times H^{1}(W, \mathbb{R}). 
\]
Given a rank $1$ flat connection $(L, \nabla)$, its image in $H^{1}(W, \mathbb{Z}/2)$ is the first Stiefel-Whitney class $w_{1}(L)$ of $L$ and classifies the line bundle up to isomorphism. The image of $(L, \nabla)$ in $H^{1}(W, \mathbb{R})$ classifies the flat connection $\nabla$ up to gauge equivalence. The element $0 \in H^{1}(W, \mathbb{R})$ corresponds to the flat connection $\nabla_{0}$ on $L$ whose monodromy representation has its image contained in $\{ \pm 1 \}$. Since $G_{1,1}$ is abelian, its action on $\Hom(\pi_1(W, x), G_{1,1})$ is trivial. Therefore, 
\[
\mathcal{M}_{1,1}(W) = \mathbb{R}^* \times (H^{1}(W, \mathbb{Z}/2) \times H^{1}(W, \mathbb{R})),
\]
and 
\[
\mathcal{M}_{1}(W, L) = \mathbb{R}^* \times H^{1}(W, \mathbb{R}). 
\]
Letting $k > 1$, we have 
\[
\mathcal{M}_{k}(W, L) = G_{k,1} \ltimes \mathrm{Hom}_{L}(\pi_{1}(W,x), G_{k,1}),
\]
where $\mathrm{Hom}_{L}(\pi_{1}(W,x), G_{k,1}) = \pi_{k, *}^{-1}( \{ L \} \times H^{1}(W, \mathbb{R}))$, and 
\[
\mathcal{M}_{k}(W, L, \nabla) = G_{k,1} \ltimes \mathrm{Hom}_{(L, \nabla)}(\pi_{1}(W,x), G_{k,1}),
\]
where $\mathrm{Hom}_{(L,\nabla)}(\pi_{1}(W,x), G_{k,1}) = \pi_{k, *}^{-1}(L, \nabla) $. 

Now let $W \subset M$ be a codimension $1$ hypersurface and let $\nu_{W}$ be the normal bundle of $W$ in $M$. Then by Corollary \ref{cor:hypersurfclass} and Lemma \ref{lem:Dprincbundle}, the groupoid of hypersurface algebroids of $b^{k+1}$-type for $(W,M)$, taken up to isotopy, fibres over the following groupoid 
\[
\mathcal{M}^0_{k}(W,M) = G_{k,1}^{0} \ltimes \mathrm{Hom}_{\nu_W}(\pi_{1}(W,x), G_{k,1}). 
\]
Consequently, the quotient space $M^0_{k}(W,M) = \mathrm{Hom}_{\nu_W}(\pi_{1}(W,x), G_{k,1})/G_{k,1}^0$ classifies hypersurface algebroids up to isotopy. Note that, as noted in Remark \ref{rem:bktangentbundleScott}, the $b^{k+1}$-tangent bundles defined by Scott \cite{MR3523250} all correspond to the trivial representation in $\mathrm{Hom}_{W \times \mathbb{R}}(\pi_{1}(W,x), G_{k,1})$. 
\setcounter{subsection}{1}
\subsubsection{The $G_{2,1}$-character variety}
We explicitly describe the groupoids $\mathcal{M}_{2}(W, \nu_W, \nabla)$ and $\mathcal{M}^0_{2}(W,M, \nabla)$ classifying hypersurface algebroids of $b^3$-type with fixed flat connection $\nabla$ on the normal bundle $\nu_W$. The isomorphism of Corollary \ref{cor:g2lgroup} simplifies to $G_{2,1} \cong \mathbb{R} \rtimes \mathbb{R}^*$, with product $(a_1, a_0)(b_1, b_0) = (a_1 + a_0^{-1} b_1, a_0 b_0)$. Fix the homomorphism $\rho_0 : \pi_1(W,x) \to \mathbb{R}^*$ corresponding to the pair $(\nu_W, \nabla)$. Then 
\[
\mathcal{M}_{2}(W, \nu_W, \nabla) \cong (\mathbb{R} \rtimes \mathbb{R}^*) \ltimes Z^{1}_{\rho_0^{-1}}(\pi_{1}(W,x), \mathbb{R}), \qquad \mathcal{M}^0_{2}(W, M, \nabla) \cong (\mathbb{R} \rtimes \mathbb{R}_{>0}) \ltimes Z^{1}_{\rho_0^{-1}}(\pi_{1}(W,x), \mathbb{R}),
\]
where $Z^{1}_{\rho_0^{-1}}(\pi_{1}(W,x), \mathbb{R})$ is the group of $1$-cocyles for the representation $\rho_0^{-1}$, and the action of $\mathbb{R} \rtimes \mathbb{R}^*$ on $Z^{1}_{\rho_0^{-1}}(\pi_{1}(W,x), \mathbb{R})$ is given by 
\[
(a_1, a_0) \ast \rho_1 = a_0^{-1} \rho_1 - \delta(a_1). 
\]
The quotient spaces are then given by 
\[
M_{2}(W, \nu_W, \nabla) \cong H^{1}_{\rho_0^{-1}}(\pi_{1}(W,x), \mathbb{R})/\mathbb{R}^*, \qquad M^0_{2}(W, M, \nabla)  \cong H^{1}_{\rho_0^{-1}}(\pi_{1}(W,x), \mathbb{R})/\mathbb{R}_{>0}.
\]
These spaces respectively classify isomorphism classes of $b^{3}$-algebroids for $(W, \nu_{W})$, and isotopy classes of $b^{3}$-algebroids for $(W,M)$, with fixed flat connection $\nabla$ on the normal bundle $\nu_W$. The point $0$ corresponds to the linear extension of $\nabla$ and it is in the closure of every other point. Removing it, we obtain the homeomorphisms
\[
M_{2}(W, \nu_W, \nabla) \setminus \{ 0 \} \cong \mathbb{RP}^{b_{1}(\nabla^{-1})-1}, \qquad M^0_{2}(W, M, \nabla) \setminus \{ 0 \} \cong S^{b_{1}(\nabla^{-1})-1},
\]
where $b_{1}(\nabla^{-1})$ is the dimension of the cohomology group $H^{1}_{\rho_0^{-1}}(\pi_{1}(W,x), \mathbb{R}) \cong H^{1}_{\nabla}(W, \nu_W^{-1})$. 

\subsubsection{The $G_{3,1}$-character variety}
We explicitly describe the groupoids $\mathcal{M}_{3}(W, \nu_W, \nabla)$ and $\mathcal{M}^0_{3}(W,M, \nabla)$ classifying hypersurface algebroids of $b^4$-type with fixed flat connection $\nabla$ on the normal bundle $\nu_W$. 

Making use of the isomorphism $G_{3,1} \cong \mathbb{R}^2 \rtimes \mathbb{R}^*$ from Corollary \ref{cor:g31group}, a homomorphism $\rho : \pi_{1}(W,x) \to G_{3,1}$ decomposes as $\rho = (\rho_1, \rho_2, \rho_0)$, where 
\begin{itemize}
\item $\rho_0 : \pi_{1}(W, x) \to \mathbb{R}^*$ is a homomorphism representing the pair 
\[
(\nu_W, \nabla) \in H^{1}(W, \mathbb{Z}/2) \times H^{1}(W, \mathbb{R}),
\] 
\item $\rho_{1} : \pi_{1}(W,x) \to \mathbb{R}$ satisfies 
\[
\rho_{1}(\gamma \eta) = \rho_{1}(\gamma) + \rho_{0}^{-1}(\gamma) \rho_{1}(\eta),
\]
meaning that it is a $1$-cocycle $\rho_{1} \in Z^{1}_{\rho_0^{-1}}(\pi_{1}(W,x), \mathbb{R})$ for the representation $\rho_0^{-1}$,
\item $\rho_{2} : \pi_{1}(W, x) \to \mathbb{R}$ satisfies 
\[
\rho_{2}(\gamma \eta) = \rho_{2}(\gamma) + \rho_{0}^{-2}(\gamma) \rho_{2}(\eta),
\]
meaning that it is a $1$-cocycle $\rho_{2} \in Z^{1}_{\rho_0^{-2}}(\pi_{1}(W,x), \mathbb{R})$ for the representation $\rho_{0}^{-2}$. 
\end{itemize} 
Therefore, $\mathrm{Hom}_{(\nu_W, \nabla)}(\pi_{1}(W,x), G_{3,1}) \cong Z^{1}_{\rho_0^{-1}}(\pi_{1}(W,x), \mathbb{R}) \times  Z^{1}_{\rho_0^{-2}}(\pi_{1}(W,x), \mathbb{R}) $ and we obtain the following description of the groupoids 
\begin{align*}
\mathcal{M}_{3}(W, \nu_W, \nabla) &\cong G_{3,1} \ltimes \big( Z^{1}_{\rho_0^{-1}}(\pi_{1}(W,x), \mathbb{R}) \times  Z^{1}_{\rho_0^{-2}}(\pi_{1}(W,x), \mathbb{R}) \big) \\
\mathcal{M}^0_{3}(W, M, \nabla) &\cong G^{0}_{3,1} \ltimes \big( Z^{1}_{\rho_0^{-1}}(\pi_{1}(W,x), \mathbb{R}) \times  Z^{1}_{\rho_0^{-2}}(\pi_{1}(W,x), \mathbb{R}) \big).
\end{align*}
Furthermore, the subgroup $K_{3,1} \cong \mathbb{R}^2$ acts via the coboundary operators as follows 
\[
(a_1, a_{2}) \ast (\rho_1, \rho_2, \rho_0) = (\rho_1 - \delta(a_1), \rho_2 - \delta(a_2), \rho_0),
\]
so that 
\begin{align*}
\mathrm{Hom}_{(L, \nabla)}(\pi_{1}(W,x), G_{3,1})/K_{3,1} &\cong H^{1}_{\rho_0^{-1}}(\pi_{1}(W,x), \mathbb{R}) \times  H^{1}_{\rho_0^{-2}}(\pi_{1}(W,x), \mathbb{R}) \\
 &\cong H^1_{\nabla}(W, \nu_W^{-1}) \times H^{1}_{\nabla}(W, \nu_W^{-2}). 
\end{align*}
There is a residual $\mathbb{R}^*$-action on this space, which has weight $-1$ on $H^1_{\nabla}(W, \nu_W^{-1})$ and weight $-2$ on $H^1_{\nabla}(W, \nu_W^{-2})$.

\begin{proposition} \label{prop:g31character}
Let $W \subset M$ be a closed codimension $1$-hypersurface, let $\nu_{W}$ be the normal bundle of $W$ in $M$, and fix a flat connection $\nabla$ on $\nu_W$. Then the space of isotopy classes of $b^4$-algebroids for $(W,M)$ inducing a fixed flat connection $\nabla$ on $\nu_W$ is homeomorphic to the following quotient space 
\[
M^0_{3}(W,M, \nabla) = \big( H^1_{\nabla}(W, \nu_W^{-1}) \times H^{1}_{\nabla}(W, \nu_W^{-2}) \big)/\mathbb{R}_{>0}. 
\]
The point $0$ corresponds to the $b^4$-algebroid induced by the linear extension of $\nabla$. This point lies in the closure of every other point. Removing it, we obtain the following homeomorphism
\[
M^0_{3}(W,M, \nabla)  \setminus \{ 0 \} \cong S^{b_1(\nabla^{-1}) + b_1(\nabla^{-2}) -1}, 
\]
where $b_1(\nabla^{-1}) $ and $b_1(\nabla^{-2})$ are, respectively, the dimensions of the cohomology groups $ H^1_{\nabla}(W, \nu_W^{-1}) $ and $ H^1_{\nabla}(W, \nu_W^{-2}) $. 
\end{proposition}
\begin{remark}
The moduli space $M^0_{3}(W,M, \nabla)$ is the fibre of the map 
\[
\pi_{3,*} : M_{3}^0(W, M) \to M_{1}^0(W, M) \cong H^{1}(W, \mathbb{R})
\]
above the cohomology class corresponding to the connection $\nabla$. We see from Proposition \ref{prop:g31character} that the dimension of this fibre depends strongly on the point $\nabla \in H^{1}(W, \mathbb{R})$. The subset of points of $H^{1}(W, \mathbb{R})$ for which the fibre $M^0_{3}(W,M, \nabla)$ consists of more than one point is similar to a resonance variety, as in \cite{falk1997arrangements}. 
\end{remark}

The space $M^0_{3}(W,M, \nabla)$ inherits a residual $\mathbb{Z}/2$-action. It acts by multiplying the factor $H^1_{\nabla}(W, \nu_W^{-1})$ by $-1$ and the factor $H^{1}_{\nabla}(W, \nu_W^{-2})$ by $1$. In terms of the isomorphism $M^0_{3}(W,M, \nabla)  \setminus \{ 0 \} \cong S^{b_1(\nabla^{-1}) + b_1(\nabla^{-2}) -1}$ of Proposition \ref{prop:g31character}, the action is by an involution which fixes a sphere of dimension $b_1(\nabla^{-2}) -1$ and acts by the antipodal map on the complementary sphere of dimension $b_1(\nabla^{-1}) -1$. The space of isomorphism classes of $b^4$-algebroids for $(W, \nu_W)$ inducing a fixed flat connection $\nabla$ on $\nu_W$ is therefore homeomorphic to the quotient space 
\[
M_{3}(W, \nu_W, \nabla) = M^0_{3}(W,M, \nabla)/(\mathbb{Z}/2).
\] 

\begin{example}
We revisit Example \ref{ex:genusgsurfaceexampleextension}. Let $W = S$ be a surface of genus $g$, let $M = \mathbb{R} \times W$ be the trivial bundle, and let $\nabla = d$ be the trivial connection. Then 
\[
M_{3}^{0}(W, W \times \mathbb{R}, d) = (H^{1}(S) \times H^{1}(S))/\mathbb{R}_{>0} \cong \mathbb{R}^{4g}/\mathbb{R}_{>0},
\]
so that $M_{3}^{0}(W, W \times \mathbb{R}, d) \setminus \{ 0 \} \cong S^{4g - 1}$. Hence, the parametrization given in Example \ref{ex:genusgsurfaceexampleextension} covers every isomorphism class. 

The extension class is an $\mathbb{R}_{>0}$-equivariant map 
\[
e_{3} : H^{1}(S) \times H^{1}(S) \to H^2(\pi_1(S), A_{4,1}) \cong H^2(S) \cong \mathbb{R},
\]
where the space $H^2(S) \cong \mathbb{R}$ is equipped with the weight $-3$-action of $\mathbb{R}_{>0}$. An expression for $e_{3}$ can be determined from the explicit expression for $E_{4,1}$ in Example \ref{Ex:explicitext41}. Given an element $\rho = (\rho_1, \rho_2) \in H^{1}(S) \times H^{1}(S)$, the cocycle $e_{3}(\rho) \in H^2(\pi_1(S), A_{4,1})$ is given by 
\[
e_{3}(\rho)(\gamma, \eta) = \frac{1}{2}(\rho_1(\gamma) \rho_2(\eta) - \rho_{1}(\eta) \rho_2(\gamma)),
\]
for $\gamma, \eta \in \pi_{1}(S)$. Comparing with Example \ref{ex:genusgsurfaceexampleextension}, we see that this matches the expression for the extension class given there. 
\end{example}

\newpage

\appendix
\newpage
\section{Topological groupoids}\label{sec:topgroupoids}

In this paper we make use of the topological analogue of a Lie groupoid. That is, we will not simply work with groupoids in topological spaces, but insist that the source and target maps are also \emph{topological submersions}. A topological submersion $f: N \to M$ is a continuous map with the property that for every $n \in N$, there is an open neighbourhood $U \subset M$ of $m = f(n)$ and a local section $s: U \to N$ of $f$ such that $s(m) = n$. Topological submersions are closed under composition and are stable under pullback. Furthermore, topological submersions are open maps. A \emph{topological groupoid} is then defined to be a groupoid $\mathcal{G} \rightrightarrows M$ in topological spaces such that the target and source maps are surjective topological submersions. 

We recall the main definitions and results regarding topological groupoids which are used in this paper. Let $\mathcal{G} \rightrightarrows M$ be a topological groupoid over $M$ with $t$ and $s$ the target and source maps, respectively. Let $f: N \to M$ be a surjective topological submersion. The \emph{pullback} $f^{!}(\mathcal{G}) \rightrightarrows N$ is naturally a topological groupoid over $N$ whose space of arrows is given by the fibre product 
\[
f^{!}(\mathcal{G}) = (N \times N) \times_{M \times M} \mathcal{G},
\]
with respect to the maps $f \times f : N \times N \to M \times M$ and $(t, s) : \mathcal{G} \to M \times M$. In other words, the arrows from $n_{1}$ to $n_{2}$ are given by arrows $g \in \mathcal{G}$ such that $s(g) = f(n_{1})$ and $t(g) = f(n_{2})$. The projection map $\pi: f^{!}(\mathcal{G}) \to \mathcal{G}$ is a homomorphism which is a surjective topological submersion between the spaces of arrows. 

\begin{lemma} \label{Moritatohomeo}
The orbit spaces of $f^{!}(\mathcal{G})$ and $\mathcal{G}$ are homeomorphic. 
\end{lemma}
\begin{proof}
The groupoid $\mathcal{G}$ induces an equivalence relation $\sim$ on $M$ whereby $m_{1} \sim m_{2}$ if there is a groupoid element $g \in \mathcal{G}$ such that $s(g) = m_{2}$ and $t(g) = m_{1}$. Let $q: M \to X = M/\sim$ be the corresponding quotient space. Similarly, $f^{!}(\mathcal{G})$ induces an equivalence relation $\sim'$ on $N$ which has quotient space $q': N \to Y = N/\sim'$. It is straightforward to see that $f$ induces a continuous bijection $\tilde{f}: Y \to X$. To see that $\tilde{f}$ is an open map, and hence a homeomorphism, it suffices to note that $f$ is an open map and that $q^{-1}(\tilde{f}(U)) = f((q')^{-1}(U))$ for subsets $U \subseteq Y$. 
\end{proof}

Suppose $\mathcal{H} \rightrightarrows N$ and $\mathcal{G} \rightrightarrows M$ are topological groupoids. A morphism of groupoids $F: \mathcal{H} \to \mathcal{G}$ covering a surjective topological submersion $f: N \to M$ always factors as $F = \pi \circ F^{!}$, where $F^{!} : \mathcal{H} \to f^{!}(\mathcal{G})$. The following definition is due to Mackenzie \cite{MR896907} in the smooth category. 
\begin{definition}
Let $F: \mathcal{H} \to \mathcal{G}$ be a morphism of topological groupoids covering a map $f: N \to M$. Then $F$ is a \emph{fibration} if both $f$ and $F^{!} : \mathcal{H} \to f^{!}(\mathcal{G})$ are surjective topological submersions. 
\end{definition}
Among all fibrations we can focus on those $F: \mathcal{H} \to \mathcal{G}$ that are moreover \emph{Morita equivalences}. This means that the induced map $F^{!} : \mathcal{H} \to f^{!}(\mathcal{G})$ is an isomorphism. If that is the case, $\mathcal{H}$ and $\mathcal{G}$ have homeomorphic orbit spaces, according to Lemma \ref{Moritatohomeo}. In fact, this holds for arbitrary fibrations:
\begin{lemma} \label{fibrationlemma}
Let $F : \mathcal{H} \to \mathcal{G}$ be a fibration of topological groupoids. This induces a short exact sequence of topological groupoids over $N$
\[
1 \to \mathcal{K} \to \mathcal{H} \to f^{!}\mathcal{G} \to 1.
\]
Furthermore, the orbit spaces of $\mathcal{H}$ and $\mathcal{G}$ are homeomorphic. 
\end{lemma}
\begin{proof}
Note first that $\mathcal{G}$ and $f^{!}(\mathcal{G})$ have homeomorphic orbit spaces by Lemma \ref{Moritatohomeo}. Hence, it suffices to prove this lemma under the assumption that $M = N$ and that the base map $f$ is the identity. Let $\mathcal{K} = \mathrm{Ker}(F) \subseteq \mathcal{H}$, a topological subgroupoid which is a bundle of groups. Indeed, observe that the homomorphism $F$ restricts to the source map on $\mathcal{K}$, which is therefore a topological submersion. The groupoid $\mathcal{K}$ acts on $\mathcal{H}$ on the right by composition and $F$ induces a homeomorphism $\mathcal{H}/\mathcal{K} \cong \mathcal{G}$. It is then clear that $\mathcal{H}$ and $\mathcal{G}$ have the same orbits and hence their orbit spaces are naturally identified. 
\end{proof}
\begin{remark} In Lemma \ref{fibrationlemma}, the subgroupoid $\mathcal{K}$ is the kernel of $F^{!}: \mathcal{H} \to f^{!}\mathcal{G}$. We will abuse terminology and call it the kernel of $F$. 
\end{remark}

\begin{lemma} \label{continuoushomsub}
A continuous homomorphism $\phi: G \to Q$ of topological groups is a fibration if and only if it is surjective and there is a local section of $\phi$ in a neighbourhood of the identity in $Q$. In this case, $\phi : G \to Q$ is a principal $K$-bundle, where $K = \mathrm{Ker}(\phi)$. 
\end{lemma}

The main examples of topological groupoids in this paper come from group actions. Let $G$ be a topological group acting continuously on a topological space $X$. Then the action groupoid $G \ltimes X$ is a topological groupoid over $X$. Indeed, the source map is simply the projection onto $X$ which is a surjective topological submersion. The following situation shows up in the study of the Riemann-Hilbert correspondence. 

\begin{proposition} \label{Moritaequivalencefromquotientprop}
Let $\phi : G \to Q$ be a fibration of topological groups with kernel $H = \mathrm{Ker}(\phi)$. Let $X$ be a topological space equipped with a continuous $G$-action such that the subgroup $H$ acts principally with quotient $Y = X/H$. Write $\pi: X \to Y$ for the quotient map. Then:
\begin{itemize}
    \item $Q$ acts continuously on $Y$.
    \item $F := \phi \times \pi: G \ltimes X \to Q \ltimes Y$ is a fibration and a Morita equivalence.
\end{itemize}
\end{proposition}
\begin{proof}
Let us first comment on the case in which $Q$ is trivial. Then $G \ltimes X$ is $H \ltimes X$ and $Q \ltimes Y$ is the trivial groupoid $Y \rightrightarrows Y$. It follows that the pullback groupoid is $\pi^{!}Y = X \times_Y X$ and $F^{!}: H \ltimes X \rightarrow X \times_Y X$ is the map that sends $(h, x)$ to $(h \ast x, x)$. $F$ being a Morita equivalence is equivalent to $F^{!}$ being a homeomorphism, which is equivalent to $\pi$ being principal. It will be convenient for us to denote $\mathrm{div}: X \times_Y X \rightarrow H \ltimes X$ for the inverse of $F^{!}$.

Let us move on to the general case. Since $\pi$ is a principal $H$-bundle, it is also a surjective topological submersion. The action of $G$ on $X$ descends to a continuous action of $Q$ on $Y$ by setting $(q,[x]) \mapsto [gx]$, where $q \in Q$, $x \in X$, and $g \in G$ is a preimage of $Q$. According to Lemma \ref{continuoushomsub}, $F = \phi \times \pi : G \ltimes X \to Q \ltimes Y$ is a fibration that covers $\pi$. By construction, $F$ is a groupoid homomorphism, so it remains to show that it is a Morita equivalence.

We must show that the induced map $F^{!} : G \ltimes X \to \pi^{!}(Q \ltimes Y)$ is an isomorphism. Note first that $\pi^{!}(Q \ltimes Y) = X \times_{Y} (Q \times X)$ and $F^{!}(g, x) = (g \ast x, \phi(g), x)$. Since $\pi$ is a principal $H$-bundle, $F^{!}$ is a bijection. Hence, it suffices to show that $F^{!}$ is an open map, which we will do by showing that it is a topological submersion. 

Let $(g, x) \in G \times X$ and let $(z, q, x) = F^{!}(g, x)$ be its image. Using that $\phi$ is a fibration, we can find an open subset $U \subset Q$ containing $q$ and a local section $\lambda : U \to G$ such that $\lambda(q) = g$. It follows that $W = X \times_{Y}(U \times X) \subseteq  \pi^{!}(Q \ltimes Y)$ is open and contains $(z, q, x)$. The composition 
\[
W \to X \times_{Y} (G \times X) \to X \times_{Y} X \cong H \times X, \qquad (a, k, b) \mapsto \mathrm{div}(a, \lambda(k) \ast b)
\]
defines a continuous map $\eta: W \to H$ such that $\eta(a,k,b) \lambda(k) \ast b = a$. Let $\gamma: W \to G \times X$ be the map $\gamma(a,k,b) = (\eta(a,k,b)\lambda(k), b)$. Then $\gamma(z,q,x) = (g,x)$, meaning that $F^{!} \circ \gamma = id_{W}$, as desired.
\end{proof}


\section{Some Differential Topology} \label{sec:DiffTop}

\subsection{Jets of sections} \label{ssec:jetsSections}

Fix a manifold $M$, a submanifold $W$, and a fibre bundle $p: A \rightarrow M$. For a positive integer $k$, we can consider the bundle of $k$-jets $J^k(A) \rightarrow M$.
\begin{definition}
A section $F: W \rightarrow J^k(A)|_W$ is \emph{holonomic} if around each $x \in W$ there is a neighbourhood $U \subset M$ and a section $f: U \rightarrow A$ whose $k$-jet restricted to $W$ is $F|_{U \cap W}$.
\end{definition}

The set of holonomic sections $J^k_W\Gamma(A)$ is a subspace of $\Gamma(J^k(A)|_W)$ and, as such, can be endowed with the Whitney topologies. The jet map $\Gamma(A) \rightarrow J^k_W\Gamma(A)$ is continuous when both source and target are endowed with the weak $C^\infty$-topology.

\begin{example} \label{ex:diffeos}
In Section \ref{sec:Riemann--Hilbert} we make use of the following concrete example. Set $A := M \times M \to M$. Then, elements of $J^k_W\Gamma(A)$ are holonomic $k$-jets along $W$ of maps from $M$ to $M$. One can then restrict to the subspace of those maps that fix $W$ pointwise; this is a constraint on the underlying $0$-jet. If $k\geq1$ we can furthermore restrict to those jets whose differential along $W$ is a pointwise isomorphism; this is a constraint on the underlying $1$-jet. This defines a subspace $J^k\Diff(M,W) \subset J^k_W\Gamma(A)$. Lemma \ref{lem:polynomialRepresentatives} below shows that any element in $J^k\Diff(M,W)$ can be represented by a diffeomorphism defined in a neighbourhood of $W$ and fixing $W$ pointwise.
\end{example}

\subsubsection{Realising holonomic jets by germs}

The space $J^k_W\Gamma(A)$ admits a more down-to-earth description if we work semi-locally around $W$. To do so, we fix a tubular neighbourhood embedding $\Phi: \nu_{W} \to M$, allowing us to assume that $M$ is a vector bundle over $W$ with projection map $\pi: M \to W$. We then fix a section $s : M \to A$ and consider the vertical bundle $V = s^*(\ker(dp)) \to M$. We can then use a fibrewise metric on $A$ to produce an embedding $V \rightarrow A$, fibered over $M$, whose image is a neighbourhood of $s$. Moreover, using a connection we can identify $V$ with $\pi^*(V|_W)$. Putting these ingredients together produces an embedding $\nu_{W} \oplus V|_W \rightarrow A$, fibered over $W$, whose image is a neighbourhood of $s|_W$ in $A$.

At this point we can note that the existence of a section $s$ as above, defined over a neighbourhood of $W$, is not restrictive. Namely, as long as $J^k_W\Gamma(A)$ is non-empty we can take a section $F$ of $J^k_W\Gamma(A)$, extract its underlying $0$-jet $W \rightarrow A|_W$, and extend it to a section of $A$ defined on a neighbourhood of $W$. This shows that we can represent the $0$-jet of $F$; the aim now is to correct the higher derivatives.

Consider now the bundle $\Delta: = \oplus_{i = 0}^k \Sym^{i}(\nu_{W}^*) \otimes V|_W \to W$. Its sections correspond naturally to polynomial sections of $V$ over $\nu_{W}$; from this perspective, the chosen section $s$ is the zero section. We endow $\Gamma(\Delta)$ with the weak $C^\infty$-topology and deduce:
\begin{lemma} \label{lem:polynomialRepresentatives}
Taking $k$-jets along $W$ defines an embedding $\Gamma(\Delta) \rightarrow J^k_W\Gamma(A)$. Its image is a $C^0$-neighbourhood of $j^k_W s$.
\end{lemma}
Concretely, observe that the claimed $C^0$-neighbourhood of $j^k_W s$ consists of those $k$-jets whose $0$-jet takes values in the image of $V|_W$ in $A|_W$. An immediate corollary is that every holonomic section can be represented by an actual section defined on a neighbourhood of $W$.

\begin{example}
We continue with Example \ref{ex:diffeos}. Let $s: M \rightarrow A = M \times M$ be the graph of the identity. Its image is the diagonal. We claim that the image of $\Gamma(\Delta)$ contains all $j^k_Wf$ as $f$ ranges over all diffeomorphism germs fixing $W$ pointwise. Indeed, fix $V$ as above and consider its image $V'$ in $A$ (under some fibrewise tubular neighbourhood embedding); this is a neighbourhood of the diagonal in $M \times M$. Given any germ $f$ of diffeomorphism we can restrict its domain until the image of its graph is completely contained in $V'$, proving the claim.
\end{example}

\begin{example} \label{ex:grass2}
Our main example, used throughout the paper, appeared already in Remark \ref{rem:grass}. Given a rank $i$, we can consider $A := \Gr(TM,i)$, the Grassmanian bundle of $i$-dimensional subspaces of $TM$. Suppose $W \subset M$ has dimension exactly $i$. Fixing a tubular neighbourhood embedding $\Phi: \nu_{W} \to M$ allows us to consider the subbundle $B \subset A$ of those subspaces that are transverse to the fibres of $\pi: M \rightarrow W$. We think of $B$ as the subbundle of splittings/connections of $\pi$. Functoriality provides an embedding $J^k_W\Gamma(B) \rightarrow J^k_W\Gamma(A)$. By construction, its image is a $C^0$-open. Different choices of $\Phi$ yield different choices of $B$ and thus different opens.

In the paper we are interested in elements of $J^k_W\Gamma(A)$ that are tangent to $W$ and are additionally involutive up to order $k$. We denote this subspace by $J^{k}\FolObj(W,M)$, which is in bijection with the set of $k$th order foliations along $W$. Observe that $J^{k}\FolObj(W,M)$ lies in the image of $J^k_W\Gamma(B)$, regardless of our choice of $\Phi$. Involutivity was explained in Section \ref{sec:jets}: it means that the sheaf of $k$-jets of vector field tangent to the $k$th order distribution should be involutive.
\end{example}

\begin{example} \label{ex:grass3}
The previous examples come together as follows. We claim that there is a continuous action of $J^k\Diff(M,W)$ on $J^k_W\Gamma(\Gr(TM,i))$. To show this, we consider $C \subset J^k(M,M) \rightarrow M$, the subbundle of $k$-jets of diffeomorphisms. Its source and target maps turn it into a groupoid over $M$. We would want to claim that $C$ acts on $J^k\Gr(TM,i)$, the bundle of $j$-kets of distribution, but this is not the case, due to the loss of a derivative (as in Subsection \ref{ssec:LieRiehart}). Instead, since the pushforward is a differential operator of order $1$ in the diffeomorphism and of order zero in the distribution, it defines a smooth map
\[ C \times_M J^k\Gr(TM,i) \rightarrow J^{k-1}\Gr(TM,i) \]
for each $k$.

Now we restrict our attention to $W$. We consider $C|_W$ and consider the subbundle $C'$ of jets that fix $W$ pointwise. This is now a bundle of Lie groups over $W$. Moreover, we consider the subbundle $D \subset J^k\Gr(TM,i)|_W \rightarrow W$ of jets of distributions that are tangent to $W$. Due to the tangency condition, we now get an action of $C'$ on $D$, without a loss of derivative. This can be checked in local coordinates and is analogous to the phenomenon observed in Subsection \ref{ssec:LieRiehart}.

The claim concludes by taking holonomic sections. The ingredient we are using here is that, by definition of the Whitney topology, continuity at the level of spaces of holonomic sections follows from continuity of the action at the level of jet bundles. Indeed, $J^k\Diff(M,W)$ is the space of holonomic sections of $C'$. Similarly, $J^k_W\Gamma(\Gr(TM,i))$ is the space of holonomic sections of $D$. Moreover, the action descends to an action on the subspace $J^{k}\FolObj(W,M) \subset J^k_W\Gamma(\Gr(TM,i))$ of involutive jets, since the pushfoward preserves involutivity.
\end{example}

\subsection{Tubular neighborhood embeddings} \label{ssec:TNE}
Let $M$ be a manifold, $W \subset M$ a submanifold, and $\nu_{W}$ the normal bundle. A tubular neighbourhood embedding is a map $\Phi : \nu_{W} \to M$ which restricts to the identity on $W$ and such that $\nu(\Phi) = id_{\nu_{W}}$. 
\begin{remark}
Any such $\Phi$ is automatically a local embedding close to $W$, but we do not require that it is an embedding globally. From a homotopical viewpoint this makes no difference, but it simplifies some arguments.
\end{remark}

We write $\TNE(W,M)$ for the space of tubular neighbourhood embeddings, which we endow with the weak $C^\infty$ topology. A folklore result often attributed to Cerf \cite{Cer61} states that this space is weakly contractible. An inspection of the proof of \cite[Proposition 31]{Godin} establishes the following stronger statement:
\begin{lemma} \label{lem:betterCerf}
Fix $\phi_0 \in \TNE(W,M)$. Then, there is a deformation retraction $F :  \TNE(W,M) \times [0,1] \to \TNE(W,M)$ to $\phi_0$ such that:
\begin{itemize}
    \item For each $\phi \in \TNE(W,M)$, the path $t \mapsto F(\phi,t)$ defines an element of $\TNE(W \times [0,1],M\times [0,1])$.
    \item The map $\TNE(W,M) \rightarrow \TNE(W \times [0,1], M\times [0,1])$ defined in this manner is continuous.
\end{itemize}
\end{lemma}

We also need a variation of this lemma in which the $k$-jet remains fixed.
\begin{lemma} \label{lem:dilating}
Consider $f,g \in \TNE(W,M)$ with the same $k$-jet, $k \geq 1$. Then the two are homotopic through tubular neighbourhood embeddings with fixed $k$-jet.
\end{lemma}
\begin{proof}
We may assume that $M=E$ is a vector bundle and that $g = id_E$. Consider then the scaling action on $E$ and define $f_t(e) := \frac{f(te)}{t}$. This is a homotopy of maps between $f_1 = f$ and $f_0 = id_E$ that is smooth in $t$.  In local coordinates one can verify that it acts by weighted scaling on Taylor coefficients, keeping the linear term fixed. Since the Taylor coefficients of $f$ are zero in degrees between $1$ and $k+1$, it follows that $f_t$ is a homotopy of maps with fixed $k$-jet, which in particular are tubular neighbourhood embeddings. The general case is obtained by applying the above construction to $f \circ g^{-1}$.
\end{proof}

\subsection{Diffeomorphisms and isotopies} \label{ssec:DiffeosAndIsotopies}

Consider a manifold $M$, a compact submanifold $W$, and a tubular neighbourhood embedding $\Phi: \nu_{W} \rightarrow M$. We are interested in $\Diff(M,W)$, the group of diffeomorphisms of $M$ which fix $W$ pointwise, and $\Diff_0(M,W)$, the connected component of the identity. We write $J^k\Diff(M,W)$ and $J^k\Diff_0(M,W)$ for the correponding groups of $k$-jets of diffeomorphisms along $W$. We can endow all of these with the Whitney $C^\infty$-topology (see Subsection \ref{ssec:jetsSections}). We let $k\geq 1$.

\begin{lemma} \label{lem:takingJets}
The jet map $J^k : \Diff_0(M,W) \rightarrow J^k\Diff_0(M,W)$ is a fibration.
\end{lemma}
\begin{proof}
According to Lemma \ref{continuoushomsub}, it is sufficient to show that $J^k$ is surjective and admits a local section in a neighbourhood of the identity. First, we establish surjectivity. Any $f \in J^k\Diff_0(M,W)$ can be represented by a germ of isotopy around $W$ that is then extended to a global element of $\Diff_0(M,W)$ via the isotopy extension theorem. 

Second, we establish the existence of a local section of $J^k$ in a neighbourhood of the identity $id_M \in \Diff_0(M,W)$.  For this, we may assume that $M = \nu_{W}$ as long as we realize every element of $J^k\Diff_0(\nu_{W},W)$ by an isotopy supported in the unit disc bundle.

As an intermediate step, we claim that there is a deformation retraction $R : U \times [0,1] \to U$ from an open neighbourhood $U \subset J^k\Diff_0(\nu_{W},W)$ to the identity. Write $\Aut(\nu_{W})$ for the space of vector bundle automorphisms. By using fibrewise dilations as in Lemma \ref{lem:dilating}, we can first show that there is a deformation retraction from $J^k\Diff(\nu_{W},W)$ to $J^k\Aut(\nu_{W})$. The latter space is readily seen to be homeomorphic to $\Aut(\nu_{W})$ itself, which indeed admits local deformation retractions around all points. 

Next, we apply Lemma \ref{lem:polynomialRepresentatives} to take polynomial representatives of all elements in $U$. This is continuous on $U$ and can be composed with $R$ to yield a continuous map:
\[ s: U \times [0,1] \rightarrow C^\infty(\nu_{W},\nu_{W}) \]
such that $s(u,1) = id_{\nu_{W}}$ for all $u \in U$. Reasoning as in \cite[Proposition 31]{Godin} tells us that there is a continuous function $\rho: U \rightarrow (0,1]$ such that $s(u,t)$ is an embedding for every $t$ when restricted to the disc bundle of radius $\rho(u)$. Consider then a family $\chi_\rho: [0,\infty) \rightarrow [0,1]$ of bump functions, continuous on $\rho \in (0,1]$, that is identically $1$ over $[0,\rho/2]$ and identically zero outside of $[0,\rho]$. This family can be used to produce a time-dependent vector field $V(u)$ that generates $t \mapsto s(u,t)$ over the disc bundle of radius $\rho(u)/2$ and is identically zero if the radius is larger than $\rho(u)$. Since the family $s$ is continuous, and differentiating in $t$ is continuous, the assignment $u \mapsto V(u)$ is continuous. Integrating $V(u)$ to its time-$1$ flow is continuous in the vector field. This produces the desired local section of $J^k$. 
\end{proof}

\section{Principal bundles and holonomy}\label{sec:principal}

In this appendix we review the results needed from the theory of principal bundles and flat connections. We discuss the Riemann-Hilbert correspondence in Section \ref{holonomymap} and its upgrade to a Morita equivalence of topological groupoids in Section \ref{TopRH}. We then discuss the classification of flat connections up to isotopy in Section \ref{isotopyRH}. 

\subsection{Holonomy and the Riemann-Hilbert correspondence} \label{holonomymap}

Let $G$ be a Lie group and let $\pi : P \to W$ be a principal $G$-bundle. We denote the Atiyah groupoid of $P$ by $\At(P) \rightrightarrows W$ and its Lie algebroid, the Atiyah algebroid, by $\at(P) \to W$. Furthermore, let $\Ad(P) \rightrightarrows W$ denote the bundle of isotropy groups of $P$, known as the gauge groupoid, and let $\ad(P) \to W$ be the adjoint bundle of $P$, which is the Lie algebroid of $\Ad(P)$. A flat connection on $P$ is by definition a Lie algebroid splitting $\nabla : TW \to \at(P)$ of the Atiyah sequence 
\[ 0 \to \ad(P) \to \at(P) \to TW \to 0. \]
This splitting can be integrated via Lie's second theorem \cite{mackenzie2000integration, moerdijk2002integrability} to a Lie groupoid homomorphism 
\[ \hol^{\nabla} : \Pi_{1}(W) \to \At(P), \]
which is called the \emph{holonomy morphism}. This morphism can be restricted to the fundamental group of $W$ at any point $x \in W$, thus giving a group homomorphism: 
\[ \hol_{x}^{\nabla} : \pi_{1}(W,x) \to \At(P)_{x} = \Ad(P)_{x} = \Aut_{G}(P_{x}), \]
into the group of $G$-equivariant automorphisms of the fibre $P_{x}$. This fibre $P_{x}$ is a $G$-torsor and so is non-canonically isomorphic to $G$. The set of framings $\phi: G \cong P_{x}$ is in bijection with $P_{x}$ itself.
\begin{definition}
The \emph{monodromy representation} of a framed flat connection $(P, \nabla, \phi)$ is defined to be the homomorphism 
\[ \hol_{x}^{\nabla, \phi} : \pi_{1}(W, x) \to G, \qquad \gamma \mapsto \phi^{-1} \circ \hol_{x}^{\nabla} (\gamma) \circ \phi. \hfill \qedhere \]
\end{definition}

The assignment of the monodromy representation to a framed flat connection is part of an equivalence of categories which is called the \emph{Riemann-Hilbert correspondence}. The categories involved in this correspondence are as follows:
\begin{definition}
Let $\FlatCat(W,G)$ denote the \emph{category of flat principal $G$-bundles} $(P, \nabla)$ whose morphisms are principal bundle isomorphisms that intertwine the flat connections. We write $\FlatObj(W,G)$ for the class of objects.
\end{definition}

\begin{definition}
Let $\FlatCat(W, x, G)$ be the \emph{category of framed flat $G$-bundles}. Its objects are flat principal bundles $(P, \nabla)$ equipped with the data of a framing $\phi : G \to P_{x}$ above $x \in W$. Morphisms are principal bundle isomorphisms that intertwine the flat connections, but are not required to preserve $\phi$.
\end{definition}
Since morphisms in $\FlatCat(W, x, G)$ are not required to preserve $\phi$, the forgetful functor $\FlatCat(W, x, G) \to \FlatCat(W,G)$ is an equivalence. 

\begin{definition}
Let $\FlatCat_{*}(W, x, G)$ denote the wide subcategory of $\FlatCat(W, x, G)$ consisting of morphisms which preserve the framings. 
\end{definition}

\begin{definition}
The \emph{category of $G$-representations of $\pi_{1}(W,x)$} is the action groupoid 
\[ \Rep(\pi_{1}(W, x), G) := G \ltimes \Hom(\pi_{1}(W, x), G). \]
Concretely, its objects $\Hom(\pi_{1}(W, x), G)$ are homomorphisms $\rho : \pi_{1}(W, x) \to G$. Elements of $G$ act on $\Hom(\pi_{1}(W, x), G)$ by conjugation.
\end{definition}

Our main object of study is the following functor from the category of framed flat connections to the category of $G$-representations. 
\begin{definition}
The \emph{Riemann-Hilbert functor}
\[ RH : \FlatCat(W, x, G) \to \Rep(\pi_{1}(W, x), G) \]
sends a framed flat bundle $(P, \nabla, \phi)$ to its monodromy representation $\hol_{x}^{\nabla, \phi}$. It sends a morphism of flat bundles $f: (P_{1}, \nabla_{1}, \phi_{1}) \to (P_{2}, \nabla_{2}, \phi_{2})$ to the morphism
\[ RH(f) = (\phi_{2}^{-1} \circ f_{x} \circ \phi_{1}, \hol_{x}^{\nabla_{1}, \phi_{1}}). \qedhere \]
\end{definition}

\begin{theorem}[Riemann-Hilbert correspondence] \label{MainRHtheorem}
The Riemann-Hilbert functor $RH$ determines an equivalence of categories 
\[  \FlatCat(W,G) \cong \Rep(\pi_{1}(W, x), G). \]
\end{theorem}

This theorem follows from a combination of Lie's second theorem and the Morita equivalence between $\Pi_{1}(W)$ and $\pi_{1}(W, x)$ given by the universal cover $\widetilde{W}$ of $W$. We will prove it in two steps, by showing that it is essentially surjective (Lemma \ref{Moritaconstruction}) and fully faithful (Lemma \ref{fulfaith}). These two steps will later be upgraded to include the natural topology on the spaces of flat connections and representations. 

\begin{lemma} \label{Moritaconstruction}
Let $\rho: \pi_{1}(W, x) \to G$ be a representation. Then 
\[ P(\rho) = (\widetilde{W} \times G )/\pi_{1}(W,x) \]
is a principal $G$-bundle equipped with a flat connection $\nabla$ and a canonical framing $\phi : G \to P(\rho)_{x}$. Furthermore, $RH(P(\rho), \nabla, \phi) = \rho$, implying that $RH$ is essentially surjective. 
\end{lemma}
\begin{proof}
Let $\widetilde{W}$ be the space of homotopy classes of paths starting at $x$. This is a model for the universal cover of $W$. It has an action of $\Pi_{1}(W)$ on the left and of $\pi_{1}(W,x)$ on the right, both by concatenation of paths. It also has a distinguished basepoint $1_{x}$, given by the constant path at $x$. Consider the space $\widetilde{W} \times G$ as the trivial flat principal $G$-bundle over $\widetilde{W}$, with the identity framing at $1_{x}$. There is a natural action of $\pi_{1}(W,x)$ on this space by flat $G$-equivariant isomorphisms. It is given by $(w, g) \ast \gamma = (w \ast \gamma, \rho(\gamma)^{-1} g)$. The quotient $P(\rho)$ is thus a principal $G$-bundle over $W$ with a flat connection and a framing inherited from $\widetilde{W} \times G$. The holonomy along $\gamma$ sends $[1_{x}, g]$ to $[\gamma, g] = [1_{x}, \rho(\gamma)g]$. Hence  $RH(P(\rho), \nabla, \phi) = \rho$.  
\end{proof}

\begin{lemma} \label{fulfaith}
The Riemann-Hilbert functor $RH$ is fully faithful. 
\end{lemma}
\begin{proof}
Let $(P_{i}, \nabla_{i}, \phi_{i})$, for $i = 1, 2$, be a pair of framed flat connections. A morphism between these objects is a map $f: P_{1} \to P_{2}$ of principal bundles which satisfies 
\begin{equation} \label{flatmorphismeq}
f_{\gamma(1)} \circ \hol^{\nabla_{1}}(\gamma) = \hol^{\nabla_{2}}(\gamma) \circ f_{\gamma(0)}
\end{equation}
 for every path $\gamma: [0,1] \to M$. Hence, $f$ is uniquely determined by $f_{x} : P_{1, x} \to P_{2, x}$, or equivalently by $RH(f)$. Furthermore, Equation \ref{flatmorphismeq} gives a formula for a map between $P_{1}$ and $P_{2}$, given $f_{x}$, which is well-defined precisely when $f_{x}$ intertwines the two monodromy representations. 
\end{proof}

\begin{corollary}[Framed Riemann-Hilbert] \label{framedRHcorr}
The Riemann-Hilbert map restricts to an equivalence of categories 
\[ RH_{*} : \FlatCat_{*}(W, x,G) \to \Hom(\pi_{1}(W, x), G), \]
where the right hand side is a set viewed as a category with only identity arrows. 
\end{corollary}
\begin{proof}
If a morphism $f$ between flat principal bundles intertwines framings then $\phi_{2}^{-1} \circ f_{x} \circ \phi_{1} = 1$. This shows that $RH_{*}$ is well-defined. It follows automatically that $RH_{*}$ is faithful and essentially surjective. To see that it is full, let $(P_{i}, \nabla_{i}, \phi_{i})$, for $i = 1, 2$, be a pair of framed flat connections which both get sent to the same representation $\rho$. Then $f_{x} = \phi_{2} \circ \phi_{1}^{-1} : P_{1, x} \to P_{2, x}$ defines a map which extends, via Equation \ref{flatmorphismeq}, to a flat isomorphism of bundles. By construction, this map lies in $\FlatCat_{*}(W, x,G)$. 
\end{proof}

\subsection{Topological Riemann-Hilbert correspondence} \label{TopRH}

In this section, we endow the categories of flat connections and fundamental group representations with natural topologies, and show that the Riemann-Hilbert correspondence can be upgraded to a Morita equivalence of topological groupoids.
\begin{assumption}
In this section we assume that $W$ is closed and connected and that $G$ is an algebraic group. 
\end{assumption} 

\subsubsection{The space of flat bundles} 

We will now topologise the space of flat connections on a fixed principal $G$-bundle $\pi : P \to W$. Having fixed $P$, we can also make the auxiliary choice of a framing $\phi: G \to P_{x}$ above a point $x \in W$.
\begin{definition} \label{def:framedflatbundle}
We write $\FlatCat(W, \phi, P)$ for the full subcategory of $\FlatCat(W, x, G)$ consisting of flat connections on the framed bundle $(P,\phi)$.
\end{definition}
In other words, we fix $(P, \phi)$ and let the flat connection $\nabla$ vary. The set of objects is thus $\FlatObj(P)$, the set of all flat connections on $P$. Let $\Gau(P)$ denote the \emph{gauge group} of $P$, which is the group of sections of $\Ad(P)$. There is an action of $\Gau(P)$ on $\FlatObj(P)$ and the category of flat connections on $(P, \phi)$ turns out to be the corresponding action groupoid 
\[ \FlatCat(W, \phi, P) = \Gau(P) \ltimes \FlatObj(P). \]

Observe that $\Gau(P)$, being the space of sections of $\Ad(P)$, can be equipped with the Whitney $C^{\infty}$-topology. With respect to this topology, $\Gau(P)$ is a topological group. Similarly, we can endow $\FlatObj(P)$ with the $C^{\infty}$-topology seeing it as a subspace of $\Conn(P)$, the space of principal connections of $P$. Recall that this is an affine space modelled on the space of equivariant forms $\Omega^1(P,\mathfrak{g})^G$.
\begin{lemma} \label{continuousaction}
$\Gau(P)$ acts continuously on $\FlatObj(P)$, and hence $\FlatCat(W, \phi, P)$ is a topological groupoid. 
\end{lemma}
\begin{proof}
The action $\Gau(P) \times \Conn(P) \rightarrow \Conn(P)$ is a differential operator and therefore continuous with respect to the $C^\infty$-topology. That it is a differential operator can be checked locally. Over a chart $U$ we can identify $\Gau(P)$ with $C^\infty(U,G)$ and $\Conn(P)$ with $\Omega^1(U,\mathfrak{g})$. Then, the gauge action has the familiar expression $(g,A) \mapsto \ad_g(A) + g^*\mu$, where $\mu$ is the Maurer-Cartan form, showing that it is of first order in the gauge transform and of zero order in the connection.
\end{proof}

\subsubsection{The based gauge group}

Consider now $\Gau_{*}(P)$, the subgroup of \emph{based} gauge transformations. These are the sections of $\Ad(P)$ which are the identity at $x$. Equivalently, $\Gau_{*}(P)$ is the kernel of the ``restriction homomorphism'' $r: \Gau(P) \to G$ that sends a gauge transformation $f$ to the element $\phi^{-1} \circ f_{x} \circ \phi \in G$.
\begin{definition}
We write $\FlatCat_{*}(W, \phi, P)$ for the action groupoid $\Gau_{*}(P) \ltimes \FlatObj(P)$.
\end{definition}
Observe that $\FlatCat_{*}(W, \phi, P)$ is the intersection of $\FlatCat(W, \phi, P)$ with the wide subcategory $\FlatCat_{*}(W, x, G) \subset \FlatCat(W, x, G)$.

\begin{lemma} \label{surjectivegauge}
The map $r: \Gau(P) \to G$ is a topological submersion. Let $G(\phi) \subseteq G$ be the image of $r$, which is a union of connected components. Then there is an exact sequence of topological groups 
\[ 1 \to \Gau_{*}(P) \to \Gau(P) \to G(\phi) \to 1. \]
\end{lemma}
\begin{proof}
Because $r$ is a group homomorphism, it suffices to prove the topological submersion property at the identity gauge transformation (cf. Lemma \ref{continuoushomsub}). Choose a local trivialization $P|_{U} \cong U \times G$, where $U$ is a neighbourhood of $x \in W$. Then $\Gau(P|_{U}) \cong C^{\infty}(U,G)$. Let $V \subset \mathfrak{g} = \Lie(G)$ be a convex neighbourhood of $0$ which is diffeomorphic, via the exponential map, to a neighbourhood of the identity in $G$. Choose a bump function $\psi \in C^{\infty}(M)$ which has compact support contained in $U$. Then define the map 
\[ s': V \to C^{\infty}(U, G), \qquad v \mapsto \exp(\psi v). \]
Note that $s'(v)$ is the identity outside of a compact subset of $U$, and so can be extended as the identity map to all of $P$. In this way we obtain a continuous section $s: V \to  \Gau(P)$ of $r$.
\end{proof}
\begin{remark} \label{Gphiconj} The group $G(\phi)$ depends on the choice of framing $\phi$, but different framings produce conjugate subgroups of $G$. In particular, if $G(\phi)$ is a normal subgroup, then it only depends on the principal bundle $P$ and not on the particular choice of framing $\phi$. This is the case, for example, if $G$ has only $2$-connected components. 
\end{remark}
\subsubsection{The space of fundamental group representations}

We now turn our attention to $\Rep(\pi_{1}(W, x), G)$. Because of the assumption that $W$ is closed, its fundamental group $\pi_{1}(W, x)$ is finitely presented. Therefore, the set of representations $\Hom(\pi_{1}(W,x), G)$ naturally has the structure of an algebraic variety and $\Rep(\pi_{1}(W, x), G) = G \ltimes \Hom(\pi_{1}(W,x), G)$ is a groupoid in the category of algebraic varieties. In particular, it is a topological groupoid. 

\subsubsection{The Riemann-Hilbert Morita equivalence}

Let $RH_{(P,\phi)}$ denote the restriction of the Riemann-Hilbert functor to $\FlatCat(W, \phi, P)$.
\begin{lemma} \label{continuityofRH}
The restricted Riemann-Hilbert functor 
\[ RH_{(P,\phi)} : \FlatCat(W, \phi, P) \to \Rep(\pi_{1}(W, x), G) \]
is continuous. 
\end{lemma}
\begin{proof}
We first prove continuity at the level of objects. Fix a finite collection of loops $\gamma_i: [0,1] \rightarrow W$ so that the classes $[\gamma_i]$ generate $\pi_1(W,x)$. Then, the continuity of $RH_{(P,\phi)}$ amounts to the continuity of the assignments $\nabla \mapsto \hol_{x}^{\nabla, \phi}([\gamma_i])$. Each of these assignments factors into two maps. The first is the pullback map $\Conn(P) \rightarrow \Conn(\gamma_i^*P)$ given by $\nabla \mapsto \gamma_i^*\nabla$, which is readily seen to be continuous because it is a differential operator.

The second map $F_i: \Conn(\gamma_i^*P) \rightarrow G$ computes the holonomy of $\gamma_i^*\nabla$ with respect to $\phi$. We can trivialise the bundle to yield an identification $\Conn(\gamma_i^*P) \cong C^\infty([0,1],\mathfrak{g})$ as spaces. The latter is the space of time-dependent left-invariant vector fields in $G$ and the holonomy map $F_i$ amounts to taking the time-$1$ flow. I.e. the continuity of $F_i$ is the claim that solutions of an ODE depend continuously on the coefficients of the ODE. For the $C^0$ topology this is immediate from Gronwall's lemma. Similarly, for the $C^r$ topology we can differentiate the ODE $r$ times and apply Gronwall. Since this is true for all $r$, the claim follows.

The continuity at the level of morphisms follows from the continuity of the restriction homomorphism $r: \Gau(P) \to G$.
\end{proof}

The image of $RH_{(P,\phi)}$ is a full subcategory of $\Rep(\pi_{1}(W, x), G)$. The following lemma and corollary describe its base space $\Hom_{(P,\phi)}(\pi_{1}(W,x), G)$.
\begin{lemma} \label{splittingofRHlemma}
Fix a point $u_{0} \in \Hom_{(P,\phi)}(\pi_{1}(W, x), G)$ and suppose $U$ is a sufficiently small neighbourhood of $u_0$. Then $U \subseteq \Hom_{(P,\phi)}(\pi_{1}(W, x), G)$ and there is a continuous section $s: U \to \FlatObj(P)$ of $RH_{(P,\phi)}$. 
\end{lemma}
\begin{proof}
Since $\Hom(\pi_{1}(W, x), G)$ is an algebraic variety, it is triangulable \cite{Whitney57}. It follows that $u_0$ admits a contractible compact neighbourhood $U$ that is moreover a simplicial complex. For instance, we can perform barycentric subdivision once and let $U$ be the star of the simplex containing $u_0$. We now implement a parametric version of Lemma \ref{Moritaconstruction}, with $U$ as parameter space.

Consider the following principal $G$-bundle over $U \times W$
\[ P(U) = (U \times \widetilde{W} \times G)/\pi_{1}(W, x), \]
where the action of $\gamma \in \pi_{1}(W,x)$ is given by $(u, w, g) \ast \gamma = (u, w \ast \gamma, u(\gamma)^{-1}g)$. We may view $P(U)$ as a continuous $U$-family of smooth principal $G$-bundles over $W$. Furthermore, for each $u \in U$, the restriction $P(U)|_{u} \to W$ is equipped with a smooth flat connection $D_{u}$ that varies continuously in $U$. There is also a family of framings $\psi: U \to P(U)$ given by $\psi(u) = [u, 1_{x}, 1]$ and the monodromy representation of $(P(U)|_{u}, D_{u}, \psi(u))$ is given by $u$. 

By Corollary \ref{framedRHcorr} and our assumption on $u_{0}$, there is an isomorphism $f_{0} : (P(U)|_{u_{0}}, D_{u_{0}}, \psi(u_{0})) \cong (P, \nabla, \phi)$, for some connection $\nabla$. This isomorphism is compatible with framings. Since $U$ is contractible, $f_{0}$ can be extended to a \emph{continuous} isomorphism of principal bundles $f: P(U) \to U \times P$. We can view $f$ as a continuous section of the fibre bundle of isomorphisms $\mathrm{Iso}(P(U), U \times P)$. Using smoothing (parametrically in $U$, using that it is a simplicial complex) we may approximate $f$ in the Whitney topology by an isomorphism $\tilde{f}$ which is a $U$-continuous family of smooth isomorphisms $P(U)|_{u} \to P$. Pushing-forward $D_{u}$ now yields a $U$-continuous family of smooth flat connections on $P$. 

The isomorphism $\tilde{f}$ may fail to push-forward the framing $\psi(u)$ to $\phi$, although we can assume that this holds at $u_{0}$. However, we may correct this discrepancy by lifting $(\tilde{f} \circ \psi)^{-1} \phi : U \to G(\phi)$ to a family of gauge transformations. For this, we use the fact, implied by Lemma \ref{surjectivegauge}, that $r$ is a principal bundle. The resulting family of connections $\nabla_{u}$ provides the desired splitting of $RH_{(P,\phi)}$. 
\end{proof}

\begin{corollary} \label{connectivityofHoms}
$\Hom_{(P,\phi)}(\pi_{1}(W, x), G)$ is a union of connected components of $\Hom(\pi_{1}(W, x), G)$. 
\end{corollary}
\begin{proof}
According to Corollary \ref{framedRHcorr}, the space of homomorphisms $\Hom(\pi_{1}(W, x), G)$ is the disjoint union of all $\Hom_{(P,\phi)}(\pi_{1}(W, x), G)$, taken over isomorphism classes of framed principal bundles $(P, \phi)$. Therefore, it suffices to show that each $\Hom_{(P,\phi)}(\pi_{1}(W, x), G)$ is open. This is the case according to Lemma \ref{splittingofRHlemma}. 
\end{proof}

\begin{lemma} \label{principalbundlestatement}
$RH_{(P,\phi)} : \FlatObj(P) \to \Hom_{(P,\phi)}(\pi_{1}(W, x), G)$ is a principal $\Gau_{*}(P)$-bundle. In particular, it is a surjective topological submersion. 
\end{lemma}
\begin{proof}
By Lemma \ref{continuousaction}, $\Gau_{*}(P)$ acts continuously on $\FlatObj(P)$. By the framed Riemann-Hilbert correspondence (Corollary \ref{framedRHcorr}) the action is free and set-theoretically the quotient can be identified with $\Hom_{(P,\phi)}(\pi_{1}(W, x), G)$. By Lemma \ref{continuityofRH}, we therefore have a continuous bijection $\FlatObj(P)/\Gau_{*}(P) \to \Hom_{(P,\phi)}(\pi_{1}(W, x), G)$. 

According to Lemma \ref{splittingofRHlemma}, $RH_{(P,\phi)}$ admits continuous local sections. Concretely, given a small open $U \subset \Hom_{(P,\phi)}(\pi_{1}(W, x), G)$ we have a section $s: U \to \FlatObj(P)$. We can use $s$ to produce a continuous $\Gau_{*}(P)$-equivariant bijection 
\[ \varphi_{s} : \Gau_{*}(P) \times U \to \FlatObj(P)|_{U}, \qquad (f, u) \mapsto f \ast s(u). \]
In order to show that $\varphi_{s}$ is a homeomorphism and thus complete the proof, it suffices to prove the continuity of the map 
\[ \mathrm{div} : \FlatObj(P) \times_{\Hom_{(P,\phi)}(\pi_{1}(W, x), G)} \FlatObj(P) \to \Gau_{*}(P),\]
which associates to a pair of flat connections $(\nabla_{1}, \nabla_{2})$ the unique based gauge transformation taking $\nabla_{2}$ to $\nabla_{1}$. But this gauge transformation is given by 
\[  f_{y} = \hol^{\nabla_{1}}(\gamma) \circ (\hol^{\nabla_{2}}(\gamma))^{-1},  \]
where $\gamma$ is any path from $x$ to $y$. Hence, this is continuous because holonomy varies continuously with the connection (as shown in the proof of Lemma \ref{continuityofRH}).
\end{proof}

\begin{corollary}
The image of $RH_{(P,\phi)}$ is the topological groupoid 
\[ \Rep_{(P,\phi)}(\pi_{1}(W, x), G) := G(\phi) \ltimes \Hom_{(P,\phi)}(\pi_{1}(W, x), G). \] 
\end{corollary}

We can now state the main result of this section. 
\begin{theorem}[Topological Riemann-Hilbert] \label{topologicalRHforPB}
The restricted Riemann-Hilbert map 
\[ RH_{(P,\phi)} : \FlatCat(W,\phi,P) \to \Rep_{(P,\phi)}(\pi_{1}(W, x), G) \]
is a Morita equivalence. In particular, the space of isomorphism classes of flat connections on $P$ is homeomorphic to the quotient space $\Hom_{(P,\phi)}(\pi_{1}(W, x), G)/G(\phi)$. 
\end{theorem}
\begin{proof}
The theorem follows from Proposition \ref{Moritaequivalencefromquotientprop}, using Lemma \ref{surjectivegauge} and Lemma \ref{principalbundlestatement} to check the conditions.
\end{proof}

\subsubsection{Dependence on the framing}

Up to isomorphism, the category $\FlatCat(W,\phi,P)$ depends solely on the isomorphism class of $P$. However, how it sits in $\FlatCat(W,x,G)$ does depend on the framing $\phi$. Moreover, the Riemann-Hilbert functor $RH_{(P, \phi)}$ and its image $\Rep_{(P,\phi)}(\pi_{1}(W, x), G)$ depend on the choice of $\phi$ as well. We make a few observations about this dependence. 

Let $\mathrm{Frame}(P,x)$ denote the set of framings $\phi : G \to P_{x}$ above $x$. There is a left action of $\Gau(P)$ on $\mathrm{Frame}(P,x)$ given by $f \ast \phi = f_{x} \circ \phi$. We say that two framings are isomorphic if they are related by a gauge transformation. Note also that $\mathrm{Frame}(P,x)$ is a right $G$-torsor. 

\begin{lemma} \label{isoclassesframing}
The set of isomorphism classes of framings $\mathrm{Frame}(P,x)/\Gau(P)$ is in bijection with $G/G(\phi)$. 
\end{lemma}
\begin{proof}
Choosing a framing $\phi$, we obtain a bijection $G \to \mathrm{Frame}(P,x)$ sending $g \in G$ to $\phi \circ L_{g}$, where $L_{g}$ is left-multiplication by $g$. Using this bijection to transfer the action of gauge transformations to an action on $G$, we see that it is given by 
\[ f \ast g = r(f)g, \]
where $r$ is the homomorphism from Lemma \ref{surjectivegauge} which has image $G(\phi)$. Hence, the quotient is $G/G(\phi)$. 
\end{proof}

Every element $g \in G$ determines an automorphism $C_{g}$ of the category of representations $\Rep(\pi_{1}(W, x), G) = G \ltimes \Hom(\pi_{1}(W,x), G)$. Namely, the conjugation $C_{g}(h, \rho) = (ghg^{-1}, g \rho g^{-1})$. 

\begin{lemma}\label{lem:dependenceframing}
Let $\phi \in \mathrm{Frame}(P,x)$ and $g \in G$. Then $RH_{(P, \phi \circ g)} = C_{g^{-1}} \circ RH_{(P, \phi)}$. In particular, $C_{g^{-1}}$ determines an isomorphism between $\Rep_{(P,\phi \circ g)}(\pi_{1}(W, x), G)$ and $\Rep_{(P,\phi)}(\pi_{1}(W, x), G)$. 
\end{lemma}

Now suppose that $G = \mathrm{GL}(n, \mathbb{R})$. It has two connected components and hence, for a given principal $\mathrm{GL}(n, \mathbb{R})$-bundle, we must decide whether $G(\phi) = \mathrm{GL}(n, \mathbb{R})$ or $G(\phi) = \mathrm{GL}^+(n, \mathbb{R})$, the connected component of the identity. Recall that by Remark \ref{Gphiconj}, $G(\phi)$ only depends on principal bundle and not the particular choice of framing. 

\begin{lemma} \label{Gphidetermination}
Let $E \to W$ be a real rank $n$ vector bundle over $W$ and let $P$ be the associated principal $\mathrm{GL}(n, \mathbb{R})$-bundle of frames. Then $G(\phi) = \mathrm{GL}(n, \mathbb{R})$ if and only if there is an automorphism $f: E \to E$ which is orientation reversing on each fibre. 
\end{lemma}
\begin{proof}
By Lemma \ref{isoclassesframing}, the set of isomorphism classes of framings consists of a single point if $G(\phi) = \mathrm{GL}(n, \mathbb{R})$ and consists of two points if $G(\phi) = \mathrm{GL}^+(n, \mathbb{R})$. In the later case, these two points correspond to the two orientations of $E_{x}$. This happens if there are no automorphisms which reverse orientation. 
\end{proof}

When the rank $n$ of a vector bundle $E$ is odd, the automorphism $-id : E \to E$ is orientation reversing on each fibre. Therefore we obtain the following. 
\begin{corollary} \label{oddn}
Let $n$ be odd. Then $G(\phi) = \mathrm{GL}(n, \mathbb{R})$ for every principal $\mathrm{GL}(n, \mathbb{R})$-bundle. 
\end{corollary}

Consider now a vector bundle $E$ of rank $2$ and suppose that there exists an orientation reversing automorphism $f : E \to E$. In other words, the determinant $\det(f) : W \to \mathbb{R}^{*}$ is negative. This implies that $f$ has distinct eigenvalues $\lambda_{+}, \lambda_{-} : W \to \mathbb{R}$ such that $\lambda_{-} < 0 < \lambda_{+}$. Hence, $E$ admits a decomposition into the eigenline subbundles $E = E_{+} \oplus E_{-}$. Conversely, if $E$ admits a line subbundle $L < E$, then using a metric we may decompose $E = L \oplus L^{\perp}$, and the automorphism $f = (-id) \oplus id$ is orientation reversing in the fibres. Therefore we obtain the following. 
\begin{corollary}\label{evenn}
Let $E$ be a rank $2$-vector bundle over $W$ and let $P$ be the associated bundle of frames. Then $G(\phi)  = \mathrm{GL}(2, \mathbb{R})$ if and only if $E$ admits a line subbundle. 
\end{corollary}

\subsection{Isotopy classification} \label{isotopyRH}

Our last goal in this appendix is to produce a Riemann-Hilbert statement classifying flat connections up to isotopy. To this end, we consider $\Gau^{0}(P)$, the identity component of the gauge group, consisting of those gauge transformations which are isotopic to the identity. This determines the wide subgroupoid 
\[ \FlatCat_{0}(W, \phi, P) := \Gau^{0}(P) \ltimes \FlatObj(P) \subset \FlatCat(W, \phi, P) \]
of flat connections modulo isotopy, which is what we want to study. We pursue the same strategy as before, which amounts to understanding how $\Gau_{*}(P)$ interacts with $\Gau^0(P)$, in order to use an analogue of Lemma \ref{principalbundlestatement}.

Write $G^0$ for the connected component of the identity and recall the homomorphism $r: \Gau(P) \to G$ of Lemma \ref{surjectivegauge} which has image $G(\phi)$. We hence have a sequence of subgroups $G^0 \subset G(\phi) \subset G$. It then holds that:
\[ \Gamma(\Ad(P)^0) = \ker(\Gau(P) \rightarrow \pi_0(G)), \]
where $\Ad(P)^0 \subset \Ad(P)$ is the subbundle consisting of the fibrewise connected components of the identities. Observe that $\Gau^0(P)$ sits in $\Gamma(\Ad(P)^0)$ as its identity component.
\begin{definition}
We denote
\[ D(P) := \dfrac{\Gamma(\Ad(P)^0)}{\Gau^0(P)} = \pi_0(\Gamma(\Ad(P)^0)) = \ker(\pi_0(\Gau(P)) \rightarrow \pi_0(G)). \qedhere \]
\end{definition}
Applying the long exact sequence of homotopy groups to the short exact sequence from Lemma \ref{surjectivegauge} we deduce:
\[  D(P) = \ker(\pi_0(\Gau(P)) \rightarrow \pi_0(G)) \cong \im(\pi_0(\Gau_*(P)) \rightarrow \pi_0(\Gau(P))) \cong \pi_0(\Gau_*(P))/\pi_1(G^0). \]

Introduce the following notation 
\[ \Gau_{*}^0(P) := \Gau^{0}(P) \cap \Gau_{*}(P), \]
for the group of based gauge transformations which are connected to the identity. Note that they need not be connected to the identity via based transformations. Hence, $\Gau_{*}^0(P)$ need not be the identity component of $\Gau_{*}(P)$. 
\begin{lemma}
$\Gau_{*}^0(P) \subset \Gau_{*}(P)$ is a disjoint union of connected components. Moreover,
\[  \pi_0(\Gau_{*}^0(P)) = \im(\pi_1(G^0)). \]
\end{lemma}
\begin{proof}
Being connected to the identity only depends on the path-component of $\Gau_{*}(P)$, so the first claim follows. For the second claim we take an element $f$ of $\Gau_{*}^0(P)$, we connect it to the identity using some path $f_s$, which we evaluate at $x \in W$, yielding a loop $\gamma$. By definition, the connecting homomorphism of the homotopy long exact sequence precisely takes $[\gamma] \in \pi_1(G^0)$ to the path-component of $f$.
\end{proof}

Observe that the inclusion $\Gau_{*}(P) \subset \Gamma(\Ad(P)^0)$ is surjective at the level of $\pi_0$, since $r$ is a principal bundle. Moreover, according to the lemma, $\Gau_{*}^0(P) = \ker(\Gau_{*}(P) \rightarrow D(P))$. The first isomorphism theorem then implies:
\begin{corollary}
$D(P) \cong \Gau_{*}(P)/\Gau_{*}^0(P).$ 
\end{corollary}

Since $\Gau_{*}^0(P)$ is a subgroup of $\Gau_{*}(P)$, its action on $\FlatObj(P)$ is principal (Lemma \ref{principalbundlestatement}). We can then take the quotient. According to Proposition \ref{Moritaequivalencefromquotientprop}, this quotient will inherit a principal action of $D(P)$. 
\begin{corollary} \label{cor:constructingH}
The space
\[ \IsotopyObj(P,\phi) := (D(P) \times \FlatObj(P))/\Gau_{*}(P) \cong \FlatObj(P)/\Gau_{*}^0(P) \]
is a principal $D(P)$-bundle over $\Hom_{(P,\phi)}(\pi_{1}(W, x), G)$.
\end{corollary}

Consequently, we obtain the following version of the Riemann-Hilbert correspondence for the subgroupoid $\FlatCat_{0}(W,\phi,P)$ classifying flat connections up to isotopy.
\begin{proposition} \label{prop:RHisotopy}
There is a continuous action of $G^0$ on $\IsotopyObj(P,\phi)$, yielding an action groupoid
\[ \IsotopyCat(P,\phi) := G^0 \ltimes \IsotopyObj(P,\phi). \]
Moreover, there is a Riemann-Hilbert functor
\[ RH^{0}_{(P,\phi)}: \FlatCat_{0}(W,\phi,P) \rightarrow \IsotopyCat(P,\phi) \]
that is a fibration and a Morita equivalence.
\end{proposition}
\begin{proof}
This follows from Proposition \ref{Moritaequivalencefromquotientprop}, using Lemma \ref{surjectivegauge} and Lemma \ref{principalbundlestatement} to check the conditions. 
\end{proof}

\begin{remark}
As a topological space, $\IsotopyObj(P,\phi)$ does not depend on the choice of $\phi$. However, the actions of $D(P)$ and $G^0$ on it, given to us by Proposition \ref{Moritaequivalencefromquotientprop}, do rely on $\phi$.
\end{remark}

In the case that the group $G$ is abelian, we obtain the following.  
\begin{lemma} \label{lem:DPabelian}
Suppose that $G$ is abelian. Then:
\[ \Gau(P) \cong C^\infty(W,G), \qquad \Gamma(\Ad(P)^0) \cong C^\infty(W,G^0), \qquad D(P) \cong \pi_0(C^\infty(W,G^0)). \]
\end{lemma}

\subsubsection{The case of line bundles}\label{sec:linebundles}

Now suppose that the structure group $G$ is $\mathrm{GL}(1, \mathbb{R}) = \mathbb{R}^{\times}$. In this case, $P \to W$ is the frame bundle of a line bundle $L$. By Corollary \ref{oddn}, we have $G(\phi) = G$. We also have $G^{0} = \mathbb{R}_{>0}$, $\Gau(P) = C^{\infty}(W, \mathbb{R}^*)$ and $\Gau^{0}(P) = C^{\infty}(W, \mathbb{R}_{>0})$. It then follows that $D(P)$ is the trivial group and hence that $\IsotopyObj(P,\phi) = \Hom_{(P,\phi)}(\pi_{1}(W, x), G)$. Note furthermore that 
\[ \Hom(\pi_{1}(W, x), G) \cong H^{1}(W, \mathbb{R}^*) \cong H^{1}(W, \mathbb{Z}/2) \times H^{1}(W, \mathbb{R}), \]
where the last isomorphism uses the isomorphism $\mathbb{R}^* \cong \mathbb{Z}/2 \times \mathbb{R}$. The framing $\phi$ is unimportant and the isomorphism type of $P$ is encoded by the first Whitney class $w_{1}(L) \in H^{1}(W, \mathbb{Z}/2)$. Hence, we obtain the following. 

\begin{corollary} \label{rank1connectedRH}
Let $P$ be a principal $\mathrm{GL}(1, \mathbb{R})$-bundle over $W$. Then, there is a Morita equivalence 
\[ \FlatCat_{0}(W, \phi, P) \cong \mathbb{R}_{>0} \ltimes H^{1}(W, \mathbb{R}), \]
where the $\mathbb{R}_{>0}$-action is trivial.  
\end{corollary}

\subsubsection{Vector bundles of rank $2$} \label{sec:2bundles}

Now suppose that the structure group is $G = \mathrm{GL}(2, \mathbb{R})$. In this case, $P \to W$ is the frame bundle of a rank $2$ vector bundle $E$. By choosing a metric on $E$, we can reduce the structure group to $O(2) \cong S^1 \rtimes \mathbb{Z}/2$. If $E$ was orientable we would apply Lemma \ref{lem:DPabelian} and deduce that
\[ D(P) \cong \pi_0(\C^\infty(W,S^1)) \cong H^{1}(W, \mathbb{Z}). \]
A similar result holds if we drop the orientability assumption, as we now explain. We write 
\[ L_P := P/\mathrm{GL}^+(2, \mathbb{R}) \]
for the \emph{orientation bundle}, which is a principal $\mathbb{Z}/2$-bundle. To it we can associate the local coefficient system:
\[ \mathbb{Z}_{P} := (L_{P} \times \mathbb{Z})/(\mathbb{Z}/2). \]
\begin{lemma}\label{lem:final}
Let $E$ be a rank $2$ vector bundle with frame bundle $P \to W$. Then 
\[ D(P) \cong H^{1}(W, \mathbb{Z}_{P}). \]
\end{lemma}
\begin{proof}
Since $D(P)$ arises after taking connected components (i.e. applying $\pi_0$), we can replace $\mathrm{GL}(2, \mathbb{R})$ with its homotopy equivalent subgroup $O(2) \cong S^1 \rtimes \mathbb{Z}/2$. Hence, in what follows we assume that $G = O(2)$ and that $P$ is a principal $O(2)$-bundle. $\Ad(P)$ is thus a bundle of $O(2)$-torsors and $\Ad(P)^0$ is a bundle of $S^1$-torsors.

Just as we defined $\mathbb{Z}_{P}$, we can consider the associated bundles of $\mathbb{R}$ and $S^1$ torsors
\[ \mathbb{R}_{P} := (L_{P} \times \mathbb{R})/(\mathbb{Z}/2), \qquad \mathbb{S}_{P}^1 := (L_{P} \times S^1)/(\mathbb{Z}/2). \]
These sit in a fibrewise exact exponential sequence
\[ 0 \to \mathbb{Z}_{P} \to \mathbb{R}_{P} \to \mathbb{S}^1_{P} \to 1 \]
whose associated long exact sequence at the level of sections yields:
\[ ... \to \Gamma(\mathbb{R}_{P}) \to \Gamma(\mathbb{S}^1_{P}) \to H^{1}(W, \mathbb{Z}_{P}) \to H^{1}(W, \mathbb{R}_{P}) \to ... \]
Observe now that $\mathbb{R}_{P}$ is a fine sheaf so the right-most term is zero. It follows that the middle map is surjective. Since $\mathbb{R}_{P}$ is the fibrewise universal cover of $\mathbb{S}^1_{P}$, we then conclude 
\[ \pi_{0}(\Gamma(\mathbb{S}^1_{P})) \cong \Gamma(\mathbb{S}^1_{P})/\Gamma(\mathbb{R}_{P}) \cong H^{1}(W, \mathbb{Z}_{P}). \]

Lastly, we observe, by taking the quotient in stages, that:
\[ \Ad(P)^0  \cong ((P \times S^1)/S^1)/(\mathbb{Z}/2) \cong (L_{P} \times S^1)/(\mathbb{Z}/2) = \mathbb{S}^1_{P}, \]
where the last isomorphism stems from the fact that $S^1$ is abelian and so $(P \times S^1)/S^1 \cong P/S^1 \times S^1$. This concludes the proof.
\end{proof}

\bibliographystyle{alpha}
\bibliography{references} 
\end{document}